\documentclass[a4paper,twoside,12pt]{amsart}
\usepackage{amsfonts}
\usepackage{mathrsfs}
\usepackage{amsmath}
\usepackage{graphics}
\usepackage{amssymb}
\usepackage{indentfirst,latexsym,bm,amsmath,pstricks,amssymb,amsthm,graphicx,fancyhdr,float,fullpage}
\usepackage[all]{xy}
\usepackage{amsthm}
\usepackage{amscd}
\usepackage[CJKbookmarks=true]{hyperref}
\usepackage{color}
\usepackage{aliascnt}
\usepackage{hyperref}
\usepackage{appendix}
\usepackage{todonotes}
\usepackage{ulem} 

\newtheorem{theorem}{Theorem}[section]

\newtheorem*{acknowledgement*}{\protect\acknowledgementname}

\newaliascnt{setup}{theorem}
\newtheorem{setup}[setup]{Setup}
\aliascntresetthe{setup}

\newaliascnt{question}{theorem}
\newtheorem{question}[question]{Question}
\aliascntresetthe{question}

\newaliascnt{lemma}{theorem}
\newtheorem{lemma}[lemma]{Lemma}
\aliascntresetthe{lemma}

\newaliascnt{conjecture}{theorem}
\newtheorem{conjecture}[conjecture]{Conjecture}
\aliascntresetthe{conjecture}

\newaliascnt{proposition}{theorem}
\newtheorem{proposition}[proposition]{Proposition}
\aliascntresetthe{proposition}

\newaliascnt{corollary}{theorem}
\newtheorem{corollary}[corollary]{Corollary}
\aliascntresetthe{corollary}

\newaliascnt{problem}{theorem}

\aliascntresetthe{problem}

\newaliascnt{claim}{theorem}

\aliascntresetthe{claim}

\theoremstyle{definition}

\newaliascnt{definition}{theorem}
\newtheorem{definition}[definition]{Definition}
\aliascntresetthe{definition}

\newaliascnt{example}{theorem}

\aliascntresetthe{example}

\theoremstyle{remark}

\newaliascnt{remark}{theorem}
\newtheorem{remark}[remark]{Remark}
\aliascntresetthe{remark}

\newaliascnt{remarks}{theorem}

\aliascntresetthe{remarks}

\providecommand{\acknowledgementname}{Acknowledgement}

\numberwithin{equation}{section}

\makeatletter
\@mparswitchfalse
\makeatother
\normalmarginpar

\newcommand{\et}{{\normalfont\textrm{\'et}}}
\newcommand{\proet}{{\normalfont\textrm{pro\'et}}}
\newcommand \MF[1]{\mathcal{MF}^{\text{tor}}(#1)}

\usepackage{setspace}
\begin{document}
\title{A Lefschetz theorem for crystalline representations}
\author{Raju Krishnamoorthy}
\email{krishnamoorthy@alum.mit.edu}
\address{Institut f\"ur Mathematik Zimmer 405, Humboldt Universit\"at Berlin, Rudower Chaussee 25, 12489 Berlin Germany}
\author{Jinbang Yang}
\email{yjb@mail.ustc.edu.cn}
\address{School of Mathematical Sciences, University of Science and Technology of China, Hefei, Anhui 230026, PR China}
\author{Kang Zuo}
\email{zuok@uni-mainz.de}
\address{School of Mathematics and Statistics, Wuhan University, Luojiashan, Wuchang, Wuhan, Hubei, 430072, P.R. China.}
\address{Institut f\"ur Mathematik, Universit\"at Mainz, Mainz 55099, Germany}

\begin{abstract}
As a corollary of nonabelian Hodge theory, Simpson proved a strong Lefschetz theorem for complex polarized variations of Hodge structure. We show an arithmetic analog. Our primary technique is $p$-adic nonabelian Hodge theory. Conditional on certain foundational results in \emph{logarithmic} $p$-adic Hodge theory, we also show a logarithmic analog.
\end{abstract}

\maketitle

\section{Introduction}\label{section:introduction}
An easy corollary of Simpson's nonabelian Hodge theorem is the following.

\begin{theorem}\label{theorem:simpson}\cite[Corollary 4.3]{Sim92} Let $X$ and $Y$ be smooth projective complex varieties, and let $f\colon Y\rightarrow X$ be a morphism that induces a surjection of topological fundamental groups:
$$f_*\colon \pi_1(Y)\twoheadrightarrow \pi_1(X).$$
Let $\mathbb{L}$ be a $\mathbb{C}$-local system on $X$ such that $f^*\mathbb{L}$ underlies a complex polarized variation of Hodge structures on $Y$. Then $\mathbb{L}$ underlies a complex polarized variation of Hodge structures on $X$.
\end{theorem}

As Simpson notes, this is especially useful when $Y$ is the smooth complete intersection of smooth hyperplane sections of $X$. This article is concerned with an arithmetic analog of \autoref{theorem:simpson}. Here, the arithmetic analog\footnote{
To specify a crystalline local system is equivalent to specifying a Fontaine-Faltings module, see \autoref{section FFM}, which consists of a filtered de Rham bundle and a divided Frobenius structure which is horizontal and satisfies strong $p$-divisibility. This is analogous to how a polarization over complex numbers satisfies the Hodge Riemann bilinear relations; see also \cite[Section 1]{LSZ13b}. Another justification for the analogy: both arise from the relative cohomology of smooth projective families.}
of a complex polarized variation of Hodge structure is a crystalline local system, which was introduced first by Fontaine-Laffaille \cite{FoLa82} for $X=\mathrm{Spec}W$ and later generalized by Faltings \cite{Fal89} to the general case. To state this arithmetic analog precisely, we need the following notation.

\begin{setup}\label{setup_proj}
Let $k=\mathbb{F}_q$ be a finite field of characteristic $p>2$ with cardinality $q$ (where $q$ is a power of $p$). Let $W:=W(k)$ denote the ring of Witt vectors over $k$, and let $K:=\mathrm{Frac}(W)$ be its field of fractions. Let $X/W$ be a smooth projective scheme of relative dimension at least $2$ with geometrically connected generic fiber. Let $j\colon D \hookrightarrow X$ be a relative smooth ample divisor, flat over $W$. It follows from Grothendieck's Lefschetz theory \cite[Corollaire 2.6]{Gro68SGA2} that there is a natural continuous surjective homomorphism of \'etale fundamental groups
 \[j_{K*}\colon \pi^{\et}_1(D_K)\twoheadrightarrow \pi^{\et}_1(X_K).\]
\end{setup}

 \begin{question}\label{ques:Raju_proj}
 In the context of \autoref{setup_proj}, let $\rho_{X} \colon \pi_1^\et(X_K)\rightarrow \mathrm{GL}_N({\mathbb{Z}}_{p^f})$ be a continuous $p$-adic representation for some $f\in \mathbb{N}$. Restricting $\rho_X$ via $j_{K*}$, one gets a representation $\rho_D \colon \pi_1^\et(D_K)\rightarrow \mathrm{GL}_N({\mathbb{Z}}_{p^f})$. Suppose $\rho_D$ is crystalline. Is $\rho_X$ also crystalline?
 \end{question}

 In this article, we answer \autoref{ques:Raju_proj} under several additional assumptions.
 \begin{theorem}\label{theorem:main_proj}
 In the context of \autoref{setup_proj}, let $f\in \mathbb{N}$ and let $\rho_{X} \colon \pi_1^\et(X_K)\rightarrow \mathrm{GL}_N({\mathbb{Z}}_{p^f})$ be a continuous representation such that $\rho_D$ is crystalline (as in \autoref{ques:Raju_proj}). Suppose further that:
 \begin{enumerate}
 \item $N^2 < p-\dim X$; and
 \item $\rho_D$ is geometrically absolutely residually irreducible, i.e., the composite
 $$\pi_1^\et(D_{\overline K})\rightarrow \pi_1^\et(D_K)\rightarrow \mathrm{GL}_N(\mathbb{Z}_{p^f})\rightarrow \mathrm{GL}_N(\overline{\mathbb{F}}_{p})$$ is an irreducible representation.
 \end{enumerate}
 Then $\rho_X$ is crystalline.
 \end{theorem}

\begin{remark}\label{rmk_HTwt}
In this article, we restrict our attention to the case where $p\geq3$.
This assumption is motivated by the fact that both Faltings' theory of crystalline representations \cite{Fal89}
and the Lan-Sheng-Zuo theory of non-abelian Hodge theory require $p\geq3$ to hold.
For a crystalline representation $\rho$, it is know that its Hodge-Tate weights lie within some interval $[a,a+p-2]$ for some $a\in \mathbb{Z}$. Since integral Tate twists do not affect the property of ``being crystalline", we may without loss of generality assume that the Hodge-Tate weights of $\rho_D$ are contained in $[0,p-2]$. This normalization is necessary for the application of the Lan-Sheng-Zuo theory.
\end{remark}

Faltings developed the notion of a crystalline representation in \cite{Fal89}.
This notion has recently been reworked in the pro-\'etale language by Tan-Tong \cite{tantong2019}.
Faltings further claims in \cite[i) p.43]{Fal89} that the theory of crystalline representations extends to the logarithmic context.
However, it seems as though the details of this construction have never appeared in the literature.
We may formulate an analog of \autoref{theorem:main_proj} in the logarithmic setting,
conditional on two foundational results in the theory of logarithmic $p$-adic Hodge theory.
The first conjecture, \autoref{conj:Faltings_Dlog} is that Faltings' theory indeed extends to the logarithmic setting,
and the second conjecture (dependent on the first),
\autoref{conjecture:compatibility_filtered_dR},
is a compatibility of two natural filtered de Rham bundles (\autoref{def_Hodgefiltration}) that one may construct:
the first from the definition of a logarithmic Fontaine-Faltings module and the second from \cite[Theorem 1.1]{DLLZ}.

In particular, modulo several basic results in logarithmic $p$-adic Hodge theory (which are, as far as we understand, not yet known or at least carefully written up) our technique extends to the logarithmic setting. We now explain this

\begin{setup}\label{setup}
Let $k=\mathbb{F}_q$ be a finite field of characteristic $p>2$ with cardinality $q$ (where $q$ is a power of $p$). Let $W:=W(k)$ denote the ring of Witt vectors over $k$, and let $K:=\mathrm{Frac}(W)$ be its field of fractions. Let $X/W$ be a smooth projective scheme of relative dimension at least $2$ with geometrically connected generic fiber. Let $S\subset X$ be a strict normal crossings divisor, flat over $W$. Let $j\colon D \hookrightarrow X$ be a relative smooth ample divisor, flat over $W$ that intersects $S$ transversely, so that $S\cap D\subset D$ is a strict normal crossings divisor. Let $X^\circ_K=X_K-S_K$ and let $D^\circ_K=D_K-(D_K\cap S_K)$. By \cite[Theorem 1.1(a)]{EK16}, there is a natural continuous surjective homomorphism
 \[j_{K*}\colon \pi_1^\et(D^\circ_K)\twoheadrightarrow \pi_1^\et(X^\circ_K).\]
\end{setup}

 \begin{question}\label{ques:Raju}
 In the context of \autoref{setup},
 let $\rho_{X} \colon \pi_1^\et(X^{\circ}_K)\rightarrow \mathrm{GL}_N({\mathbb{Z}}_{p^f})$ be a continuous $p$-adic representation.
 Restricting $\rho_X$ via $j_{K*}$, one gets a representation $\rho_D \colon \pi_1^\et(D^{\circ}_K)\rightarrow \mathrm{GL}_N({\mathbb{Z}}_{p^f})$.
 Suppose \autoref{conj:Faltings_Dlog} holds for $(X,S)$ and $(D,D\cap S)$ and that moreover $\rho_D$ is logarithmic crystalline. Is $\rho_X$ also logarithmic crystalline?
 \end{question}

 \begin{theorem}\label{theorem:main}
 In the context of \autoref{ques:Raju}, suppose further that
 \begin{enumerate}
 \item $N^2 < p-\dim X$;
 \item $\rho_D$ is geometrically absolutely residually irreducible, i.e., the composite
 $$\pi_1^\et(D^\circ_{\overline K})\rightarrow \pi_1^\et(D^\circ_K)\rightarrow \mathrm{GL}_N(\mathbb{Z}_{p^f})\rightarrow \mathrm{GL}_N(\overline{\mathbb{F}}_{p})$$ is an irreducible representation;
 \item the line bundle $\mathcal{O}_X(D-S)$ is ample on $X$; and
 \item if $S\neq \emptyset$, then \autoref{conjecture:compatibility_filtered_dR} holds for $\rho_D$.
 \end{enumerate}
 Then $\rho_X$ is logarithmic crystalline.
 \end{theorem}

 \begin{remark} \autoref{theorem:main} is strictly more general than \autoref{theorem:main_proj}, for the following reasons:
 \begin{enumerate}
 \item \autoref{conj:Faltings_Dlog} asserts that Faltings' theory of crystalline representations extends naturally to the logarithmic setting. This conjecture has very recently been proved by Z. Liu and the second and third authors in \cite{LYZ25}.
 \item The compatibility in \autoref{conjecture:compatibility_filtered_dR} will follow from work on Tan-Tong, as described in the next remark (in particular, \autoref{conjecture:compatibility_filtered_dR} holds when the boundary divisor is empty).
 \end{enumerate}
 We shall therefore focus on proving \autoref{theorem:main}.
 \end{remark}

 \begin{remark}In \autoref{theorem:main}(4), when $S\neq \emptyset$, we have an extra condition, namely, that \autoref{conjecture:compatibility_filtered_dR} holds for $\rho_D$. This is a natural conjecture in relative $p$-adic Hodge theory, but let us briefly explain what it means for \autoref{conjecture:compatibility_filtered_dR} to hold for $\rho_D$. (We maintain notation from \autoref{setup}.) Suppose \autoref{conj:Faltings_Dlog} holds. Let $\rho_D$ be a logarithmic crystalline representation, associated to a logarithmic Fontaine-Faltings module $(M,\nabla,\mathrm{Fil},\varphi)_{D}$.
 Then the conjecture says the following two filtered de Rham bundles over the scheme $(U\cap D)_K$ are isomorphic:
 \begin{itemize}
 \item $(M,\nabla,\mathrm{Fil})^{\vee}\mid_{D^{\circ}_K}$ and
 \item $\mathbb{D}_{\mathrm{dR}}^{\mathrm{alg}}(\rho_D\mid_{D^{\circ}_K}),$
 \end{itemize} where $\mathbb{D}_{\mathrm{dR}}^{\mathrm{alg}}$ is the algebraic $p$-adic Riemann-Hilbert functor of \cite[Theorem 1.1]{DLLZ}. (For this to make sense in the context of \cite{DLLZ}, it suffices to note that the local system $\rho_D\mid_{D^{\circ}_K}$ is crystalline and hence de Rham by \cite{tantong2019}.)

 To prove this conjecture, one would need to develop a logarithmic variant of \cite{tantong2019}. In particular, one needs to construct the logarithmic version of a crystalline period sheaf as in \cite{tantong2019} such that it naturally embeds into the logarithmic de Rham period sheaf constructed in \cite[Definition 2.2.10]{DLLZ}.
 \end{remark}

 We do not yet know how to relax the other assumptions in Theorem \ref{theorem:main}.
 The assumptions (1) and (2) of the theorem are potentially removable.
 However, assumption (3) cannot be removed.
 The proof of Theorem \ref{theorem:main} uses both serious $p$-adic Hodge theory and $p$-adic nonabelian Hodge theory.
 Subtracting the $p$-adic Hodge theory, one is left with the following result,
 which immediately follows from our main technique (and makes no reference to (logarithmic) crystalline representations). Recall that a \emph{logarithmic Higgs bundle over $(X,S)$} is a vector bundle $E$
 over $X$ together with an $\mathcal{O}_X$-linear map $\theta\colon E\rightarrow E\otimes \Omega_X^1(\log S)$
 such that $\theta\wedge \theta = 0$.

\begin{theorem}\label{theorem:lefschetz_higgs}
Let $k$, $W$, $K$, $X$, $S$, and $D$ be as in \autoref{setup}. Let $(E,\theta)$ be a logarithmic Higgs bundle on $(X,S)$ of rank $N$. Suppose further that:
\begin{enumerate}
\item the logarithmic Higgs bundle $(E,\theta)\mid_{(D,D\cap S)}$ over $(D,D\cap S)$ initiates an $f$-periodic Higgs-de Rham flow for some integer $f\geq 1$;
\item $N^2<p-\dim(X)$;
\item the logarithmic Higgs bundle $(E,\theta)\mid_{(D,D\cap S)\mod p}$ over the special fiber of $(D,D\cap S)$ is stable;
\item the line bundle $\mathcal{O}_X(D-S)$ is ample on $X$.

\end{enumerate}
Then $(E,\theta)$ initiates an $f$-periodic Higgs-de Rham flow on $(X,S)$, extending the flow on $(E,\theta)\mid_{(D,D\cap S)}$.
\end{theorem}

We now move on to two applications of our main theorem, which avoid the logarithmic setup. As a corollary to \autoref{theorem:main_proj}, we have the following.
\begin{corollary}\label{corollary:abelian_mixed}
Setup as in \autoref{setup_proj}. Let $A_D\rightarrow D$ be an abelian scheme. Let $\rho_D\colon \pi_1^\et(D_K)\rightarrow \mathrm{GL}_{2g}(\mathbb{Z}_{p})$ be the representation induced from the $p$-adic Tate module of $A_D\rightarrow D$ over the generic fiber $D_K$. Suppose that:
\begin{enumerate}
\item $4g^2<p-\dim X$;
\item $\rho_D$ is geometrically absolutely residually irreducible, as in \autoref{theorem:main_proj}.
\end{enumerate}
Then the following are equivalent.
\begin{itemize}
\item $A_D\rightarrow D$ extends to an abelian scheme $A_X\rightarrow X$;
\item the local system $\rho_D\colon \pi_1^\et(D_K)\rightarrow \mathrm{GL}_{2g}(\mathbb{Z}_p)$ extends to a local system $\rho_X\colon \pi_1^\et(X_K)\rightarrow \mathrm{GL}_{2g}(\mathbb{Z}_p)$.
\end{itemize}
\end{corollary}

 Note that if $\dim(X)\geq 3$, then $j_{K*}\colon \pi_1^\et(D_K)\rightarrow \pi_1^\et(X_K)$ is an isomorphism, see \cite[Th\'eor\`eme 3.10]{Gro68SGA2}; hence every local system on $D_K$ extends to a local system on $X_K$. Using our proof technique and a standard spreading out argument, one recovers a very special case of a corollary of Simpson's \autoref{theorem:simpson}: $X/\mathbb{C}$ is a smooth projective variety of dimension at least $2$ and $D\subset X$ is a smooth ample divisor, then an abelian scheme $A_D\rightarrow D$ extends to $X$ if and only if the local system coming from the relative first singular cohomology extends to a local system on $X$.\footnote{That this follows from \autoref{theorem:simpson} amounts to the following fact: that the category polarizable $\mathbb{Z}$-variations of Hodge structures on $X$ is equivalent to the category of abelian schemes on $X$.}

\begin{corollary}\label{corollary:abelian_C}Let $X/\mathbb{C}$ be a smooth projective variety of dimension at least 2, let $D\subset X$ be a smooth ample divisor, and let $f_D\colon A_D\rightarrow D$ be an abelian scheme such that the associated graded Higgs bundle attached to de Rham cohomology
$$(E_D,\theta_D):=\mathrm{Gr}_{\mathrm{Fil}_{\mathrm{Hodge}}}(\mathcal{H}^1_{\rm dR}(A_D/D),\nabla)$$ is stable.\footnote{The classical nonabelian Hodge theory \cite{Sim92} tell us that this stability is equivalent to the local system $R^1f_{D*}(\mathbb{C})$ being irreducible.} Then the following are equivalent.
\begin{itemize}
\item $A_D\rightarrow D$ extends to an abelian scheme $A_X\rightarrow X$;
\item the graded Higgs bundle $(E_D,\theta_D)$ extends to a graded Higgs bundle $(E_X,\theta_X)$ on $X$.
\end{itemize}
Moreover, if $\dim(X)\geq 3$, then $A_D\rightarrow D$ always extends to an abelian scheme $A_X\rightarrow X$.
\end{corollary}
While \autoref{corollary:abelian_C} is indeed true without any stability assumption on $(E_D,\theta_D)$ by Simpson's \autoref{theorem:simpson}, the proof in \cite{Sim92} is highly analytic. On the other hand, our method is arithmetic/algebro-geometric.

\autoref{ques:Raju} was inspired by a question of Simpson. To state this question, we first need a definition.
\begin{definition}Let $X/\mathbb{C}$ be a smooth variety and let $\mathbb{L}$ be a $\mathbb{C}$-local system on $X$. We say that $\mathbb{L}$ is \emph{motivic} if there exists a dense open $U\subset X$, a smooth projective morphism $f\colon \mathcal{Y}\rightarrow U$, and an integer $i\geq 0$, such that $\mathbb{L}\mid_U$ is a subquotient \footnote{equivalently, by a theorem of Deligne \cite[Corollaire (4.2.9)]{Del71}, a summand} of $R^if_*\mathbb{C}$.
\end{definition}
\begin{question}[Simpson]\label{ques:Simpson}
Let $X/\mathbb{C}$ be a smooth, projective variety of dimension at least 2, and let $D\subset X$ be a smooth ample divisor. Let $\mathbb{L}$ be a $\mathbb{C}$-local system on $X$ such that $X\mid_D$ is motivic. Then is $\mathbb{L}$ motivic?
\end{question}

\autoref{ques:Simpson} is compatible with other conjectural characterizations of motivic local systems, e.g., Simpson's Standard conjecture \cite[p. 372]{Simpson90}. We show that \autoref{theorem:main} has the following sample application, which provides some evidence for \autoref{ques:Simpson}. Notably, since the following four terms involved each constitute a necessary condition for $\mathbb{L}$ to be motivic.

\begin{corollary}\label{corollary:almost_all_p}
Let $K$ be a number field with ring of integers $\mathcal{O}_K$, and fix a positive integer $\mathcal{N}$. Let $X$ be a smooth projective scheme over $\mathcal{O}_K[1/\mathcal{N}]$, of relative dimension at least $3$, with geometrically connected generic fiber. Let $D\subset X$ be a relative smooth ample divisor. Let $f_D\colon Y_D\rightarrow D$ be a smooth projective morphism and $i\geq 0$. Suppose the following equivalent conditions hold.
\begin{itemize}
\item The associated graded Higgs bundle
\[\mathrm{Gr}_{\mathrm{Fil}_\mathrm{Hodge}}(\mathcal{H}^i_{\rm dR}(Y_{D}/{D}),\nabla)\]
is a stable Higgs bundle over $D_K$.
\item For an(y) embedding $\iota\colon K\hookrightarrow \mathbb{C}$, the complex local system $$R^if_{D_{\mathbb{C}}*}(\mathbb{C})$$
on $D_{\mathbb{C}}:=D\times_{K,\iota}\mathbb{C}$ is irreducible.
\end{itemize}
Then, after potentially replacing $K$ by a finite extension and $\mathcal{N}$ by a larger integer, we have the following:
\begin{enumerate}
\item the Gau\ss-Manin connection together with its Hodge filtration $(\mathcal{H}^i_\mathrm{dR}(Y_D/D),\nabla_{\mathrm{GM}},\mathrm{Fil}_{\mathrm{Hodge}})$ canonically extends to a filtered de Rham bundle $(\mathcal{H},\nabla,\mathrm{Fil})$ on $X$.
\item For all primes $\mathfrak p\gg 0$ of $K$, the $\mathfrak p$-adic completion of $(\mathcal{H},\nabla,\mathrm{Fil})$ underlies a filtered Frobenius crystal.
\item For all primes $\mathfrak p\gg 0$ of $K$, the $\mathfrak p$-adic completion of the associated graded Higgs bundle
$$(E,\theta):=\mathrm{Gr}_{\mathrm{Fil}}(\mathcal{H},\nabla)$$
is $\mathfrak p$-adically 1-periodic
\item For any embedding $\iota\colon \mathcal{O}_K\hookrightarrow \mathbb{C}$, the filtered de Rham bundle $(\mathcal{H},\nabla, \mathrm{Fil})\times_{K,\iota}\mathbb{C}$ on $X\times_{K,\iota}\mathbb{C}$ underlies a $\mathbb{Z}$-polarized variation of Hodge structures.
\end{enumerate}
\end{corollary}

We find it very likely that each of conditions (2) and (3) in fact characterizes those integrable connections which are motivic. While (2) formally implies (3), we first prove (3) and use it to deduce (2). Note that condition (2) in particular guarantees that the integrable connection $(\mathcal{H}_X,\nabla)$ is \emph{globally nilpotent}, i.e., that for all $\mathfrak p\gg 0$, the $p$-curvature is nilpotent on the mod $\mathfrak p$ reduction of $X$. The condition of being globally nilpotent conjecturally characterizes irreducible motivic integrable connections. We do not know of an alternative proof of the global nilpotence of the integrable connection over $X$.

We briefly explain the structure of the proof of \autoref{theorem:main}. According \autoref{rmk_HTwt}, we first assume Hodge-Tate weights of $\rho_D$ are all contained in $[0,p-2]$.
 \begin{enumerate}
 \item[\textbf{Step 1}]In \autoref{section:LSZ}, we transform the question into a problem about extending a periodic Higgs-de Rham flow.
 The theory of Higgs-de Rham flows has its origins in the seminal work of Ogus-Vologodsky on nonabelian Hodge theory in characteristic $p$ \cite{OgVo07}.
 This theory has recently been enhanced to a $p$-adic theory by Lan-Sheng-Zuo. According to the theory of Lan-Sheng-Zuo,
 there is an equivalence between the category of certain crystalline representations (with bounds on the Hodge-Tate weights) and periodic Higgs-de Rham flows.
 (Indeed, it is here where our convention that crystalline representations have width $\leq p-2$ is relevant.)
 Denote by $\mathcal{D}$ (resp. $\mathcal{S}$ and $\mathcal{X}$) the $p$-adic formal completion of $D$ (resp. $S$ and $X$) along its special fiber.
 Let $\mathrm{HDF}_{\mathcal{D}}$ be the logarithmic Higgs-de Rham flow over $(\mathcal{D},\mathcal{S}\cap \mathcal{D})$ associated to the representation $\rho_D$:

 \begin{equation}\label{equ:HDF}
 \xymatrix@C=0.2cm{
 & (V_{\mathcal{D}},\nabla_{\mathcal{D}},\mathrm{Fil}_{\mathcal{D}})_0 \ar[dr]
 && (V_{\mathcal{D}},\nabla_{\mathcal{D}},\mathrm{Fil}_{\mathcal{D}})_1 \ar[dr]
 && \cdots \\
 (E_{\mathcal{D}},\theta_{\mathcal{D}})_0 \ar[ur]
 && (E_{\mathcal{D}},\theta_{\mathcal{D}})_1 \ar[ur]
 && (E_{\mathcal{D}},\theta_{\mathcal{D}})_2 \ar[ur]\\
 }
 \end{equation}
 Then we need to extend $\mathrm{HDF}_{\mathcal{D}}$ to some periodic Higgs-de Rham flow over $(\mathcal{X},\mathcal{S})$.

 \item[\textbf{Step 2}] In \autoref{section:crystalline_dR_Higgs}, we extend $(E_{\mathcal{D}},\theta_{\mathcal{D}})_i$ to a graded logarithmic semistable Higgs bundles $(E_{\mathcal{X}},\theta_{\mathcal{X}})_i$ over $(\mathcal{X},\mathcal{S})$.
 Using Scholze's notion of \emph{de Rham local systems} together with a rigidity theorem due to Liu-Zhu (and Diao-Lan-Liu-Zhu), we construct a graded logarithmic Higgs bundle $(E_{\mathcal{X}_K},\theta_{\mathcal{X}_K})$ over $(\mathcal{X}_K,\mathcal{S}_K)$ such that
 \[(E_{\mathcal{X}_K},\theta_{\mathcal{X}_K})\mid_{\mathcal{D}_K} = (E_{\mathcal{D}},\theta_{\mathcal{D}})\mid_{\mathcal{D}_K}.\]

 One gets graded Higgs bundles $(E_{\mathcal{X}_K},\theta_{\mathcal{X}_K})_i$ extending $(E_{\mathcal{D}},\theta_{\mathcal{D}})_i\mid_{\mathcal{D}_K}$.
 In \autoref{section:langton}, using a result of Langer (extending work of Langton), we extend $(E_{\mathcal{X}_K},\theta_{\mathcal{X}_K})=(E_{\mathcal{X}_K},\theta_{\mathcal{X}_K})_0$ to a semistable Higgs torsion-free sheaf $(E_{\mathcal{X}},\theta_{\mathcal{X}})$ on $(\mathcal{X},\mathcal{S})$.

 We show that this extension $(E_{\mathcal{X}},\theta_{\mathcal{X}})$ is unique up to an isomorphism and has trivial Chern classes in \autoref{sect:local_constancy}, which implies that $(E_{\mathcal{X}},\theta_{\mathcal{X}})$ is locally free using work of Langer.

\item[\textbf{Step 3}]We show that $(E_{\mathcal{X}},\theta_{\mathcal{X}})$ constructed in Step 2 has stable reduction modulo $p$ and is graded. To do this, we prove a Lefschetz theorem for semistable Higgs bundles with vanishing Chern classes using a vanishing theorem of Arapura in \autoref{section:arapura}. It is here that our assumptions transform from $N<p$ to $N^2<p-\dim(X)$. The argument that $(E_{\mathcal{X}},\theta_{\mathcal{X}})$ is graded is contained in \autoref{section:graded}.

 \item[\textbf{Step 4}]In \autoref{section:HDR_X1}, we extend $\mathrm{HDF}_{\mathcal{D}}\mid_{D_1}$ to a Higgs-de Rham flow $\mathrm{HDF}_{X_1}$ over $X_1$ (here, for any $n\in\mathbb{N}$, we denote by $X_n$ (resp. $D_n$) the reduction modulo $p^n$ of $X$ (resp. $D$)). By the stability of the Higgs bundle, this flow extends $\mathrm{HDF}_{\mathcal{D}}\mid_{D_1}$.

 \item[\textbf{Step 5}] In \autoref{section:HDR_X}, we deform $\mathrm{HDF}_{X_1}$ to a ($p$-adic, periodic) Higgs-de Rham flow $\mathrm{HDF}_{\mathcal{X}}$ over $\mathcal{X}$, which extends $\mathrm{HDF}_{\mathcal{D}}$. To do this, we use results of Krishnamoorthy-Yang-Zuo~\cite{KYZ} that explicitly calculate the obstruction class of deforming each piece of the Higgs-de Rham flow together with a Lefschetz theorem relating this obstruction class to the obstruction class over $D$.
 \end{enumerate}
 To prove \autoref{theorem:lefschetz_higgs}, one simply follows steps 3-5 above; we leave it to readers. In \autoref{applications}, we explain the proofs of \autoref{corollary:abelian_mixed}, \autoref{corollary:abelian_C}, and \autoref{corollary:almost_all_p}.

 We have two appendices. In Appendix \ref{section FFM}, we review the theory of Fontaine-Faltings modules and crystalline representations. In Appendix \ref{section:log_FFM}, we explain the construction of the category of logarithmic Fontaine-Faltings modules and carefully state \autoref{conj:Faltings_Dlog} and \autoref{conjecture:compatibility_filtered_dR}, the two basic results that are required to formulate and prove \autoref{theorem:main} when the boundary divisor is non-empty.

\begin{acknowledgement*}
We thank Adrian Langer for several useful emails about his work. We thank Ruochuan Liu for explanations and clarifications on his joint work with Xinwen Zhu. We thank Carlos Simpson for the suggestion of \autoref{ques:Simpson} as well as comments on this article and his interest in this work. We thank Ariyan Javanpeykar for his help with the rigid analytic Riemann existence theorem.

We heartily thank the anonymous referees for devoting a great deal of time and effort to conducting meticulous reviews of our paper,
as well as for their insightful comments and suggestions, which have significantly improved the quality of this article and enhanced its readability. We spent several years trying to work out the theory of logarithmic crystalline representations carefully. While \autoref{conj:Faltings_Dlog} has now been resolved (as addressed in \cite{LYZ25}), \autoref{conjecture:compatibility_filtered_dR} remains unresolved to date.
In our efforts to build this delicate theory, several anonymous referees were enormously helpful in identifying errors and gaps.

R.K. thanks Johan de Jong, Ambrus P\'al and Carlos Simpson for early conversations which suggested that \autoref{ques:Raju} could possibly have an affirmative answer. R.K. also thanks Philip Engel and Daniel Litt for several helpful discussions. R.K. thanks the Universit\"at Mainz, where much of this work was conducted, for pleasant working conditions. R.K gratefully acknowledges support from NSF Grant No. DMS-1344994 and ERC Grant No. 101020009 (project TameHodge). J.Y. is supported by National Natural Science Foundation of China Grant No. 12201595, USTC Research Funds of the Double First-Class Initiative No. YD0010002006, and the Fundamental Research Funds for the Central Universities and CAS Project for Young Scientists in Basic Research Grant No. YSBR-032. K.Z. is supported by National Natural Science Foundation of China Key Program Grant No. 12331002, and Science Fund for Global Challenges and Sustainability No. W2441003.
\end{acknowledgement*}
\section{Notation}
The following notation is in full force for \autoref{section:LSZ}-\autoref{section:HDR_X}.
 \begin{itemize}
 \item $p$ is an odd prime number.
 \item[$\bullet$] $k=\mathbb{F}_q$ is a finite field of characteristic $p$ and cardinality $q$ (a power of $p$)
 \begin{itemize}
 \item $W := W(k)$,
 \item $K := {\rm Frac}\,W$.
 \end{itemize}
 \item[$\bullet$]$(X,S)$: $X$ is a smooth projective scheme over $\mathrm{Spec}(W)$ with geometrically connected generic fiber and $S\subset X$ is a relative (strict) normal crossings divisor, flat over $W$.
 \begin{itemize}
 \item $(X_n,S_n)$: the reduction of $(X,S)$ modulo $p^n$ for any $n\geq 1$;
 \item $(X_K,S_K)$: the generic fiber of $(X,S)$;
 \item $X^\circ=X-S$
 \item $\mathcal{X}$: the $p$-adic formal completion of $X$ along the special fiber $X_1$;
 \item $\mathcal{S}$: the $p$-adic formal completion of $S$ along the special fiber $S_1$;
 \item $\mathcal{X}_K$: the rigid-analytic generic fiber of the formal scheme $\mathcal{X}$.
 \item $\mathcal{S}_K$: the rigid-analytic generic fiber of the formal scheme $\mathcal{S}$.

 \end{itemize}
 \item[$\bullet$] $D\subset X$: a relative smooth ample divisor, flat over $W$, that intersects $S$ transversely.
 \begin{itemize}
 \item Same notation for $D_1$, $D_n$, $D_K$, $\mathcal{D}$, $\mathcal{D}_K$, $D^\circ$, etc.
 \end{itemize}

 \end{itemize}

 \section{The theory of Lan-Sheng-Zuo}\label{section:LSZ}

 The following fundamental theorems are a combination of work of Lan-Sheng-Zuo (\cite[Theorem 1.4]{LSZ13a} for non-logarithmic case and joint with Y.~Yang in \cite[Theorem 1.1]{LSYZ14} for the logarithmic setting) together with the work of Faltings.

 \begin{theorem}\label{thm:periodic_crystalline}(Lan-Sheng-Zuo) Let $X/W$ be a smooth projective scheme and let $S\subset X$ be a relative simple normal crossings divisor, and let $X^\circ_K=X_K\backslash S_K$. Suppose \autoref{conj:Faltings_Dlog} holds for $(X,S)/W$. Then for each natural number $f\in \mathbb{N}$, there is an equivalence between the category of logarithmic crystalline representations $\pi_1^\et(X^\circ_K)\rightarrow \mathrm{GL}_N({\mathbb{Z}}_{p^f})$ with Hodge-Tate weights in the interval $[0,p-2]$ and the category of $f$-periodic Higgs-de Rham flows over $(\mathcal{X},\mathcal{S})$ where the exponents of nilpotency are less than or equal to $p-2$.
 \end{theorem}

 We assume that crystalline representations have width $\leq p-2$ exactly to relate the problem to a corresponding problem about Higgs-de Rham flows. By comparison, Faltings assumed in \cite[Theorem 2.3]{Fal89} that the width was merely $\leq p-1$. However, when constructing the crystalline representations, he needed to restrict the width $\leq p-2$ in \cite[Theorem 2.6]{Fal89}.

 Let us use Theorem \ref{thm:periodic_crystalline} to investigate \autoref{ques:Raju}. Since $\rho_D$ is logarithmic crystalline, there exists an integral Tate twist $\rho_D(n)$ with Hodge-Tate weights in $[0,p-2]$. Twisting by a power of the cyclotomic character, we may assume that the Hodge-Tate weights of $\rho_D$ are in $[0,p-2]$. Then there is a periodic Higgs-de Rham flow $\mathrm{HDF}_{\mathcal{D}}$ on $\mathcal{D}$ associated to this logarithmic crystalline representation under \autoref{thm:periodic_crystalline}.

 \begin{lemma}\label{lem:lift_HDF}
 Setup as in \autoref{ques:Raju} and suppose $N\leq p-2$. Then \autoref{ques:Raju} has an affirmative answer if and only if there exists a periodic Higgs-de Rham flow over $(\mathcal{X},\mathcal{S})$ extending the Higgs-de Rham flow $\mathrm{HDF}_{\mathcal{D}}$ on $(\mathcal{D},\mathcal{D}\cap \mathcal{S})$.
 \end{lemma}
 \begin{proof}
 If $\rho_X$ is logarithmic crystalline, we choose the Higgs-de Rham flow $\mathrm{HDF}_{\mathcal{X}}$ associated to $\rho_X$.

 Conversely, suppose there exists a Higgs-de Rham flow $\mathrm{HDF}'_{\mathcal{X}}$ on $\mathcal{X}$ extending the Higgs-de Rham flow $\mathrm{HDF}_{\mathcal{D}}$ on $(\mathcal{D},\mathcal{D}\cap \mathcal{S})$. As $N\leq p-2$, the exponents of nilpotency are all $\leq p-2$. Then by \autoref{thm:periodic_crystalline}, one obtains a logarithmic crystalline representation $\rho_X'$ extending $\rho_D$. As the map
 $$\pi_1^\et(D^\circ_K)\rightarrow \pi_1^\et(X^\circ_K)$$
 is surjective (by e.g. \cite[Theorem 1.1(a)]{EK16}), we see that $\rho_X$ is isomorphic to $\rho'_{X}$; hence $\rho_X$ is logarithmic crystalline as desired.
 \end{proof}

We extend a useful preperiodicity result, see \cite[Theorem 1.5]{LSZ13a} and \cite[Proposition 1]{Lan15}, to the logarithmic setting. The proof is the same as the non-logarithmic version; we include it here for the convenience of the readers.
\begin{lemma} \label{footnote:preperiodic}
let $(X_1,S_1)$ be a smooth projective variety and simple normal crossings divisor defined over $\overline{\mathbb{F}}_p$. Suppose the pair $(X_1,S_1)$ has a smooth lifting to $W_2$. Let $(E,\theta)$ be a logarithmic Higgs bundle with nilpotent Higgs field over $(X_1,S_1)$ and with $\mathrm{rank}(E)\leq p$. Then $(E,\theta)$ is semistable with vanishing Chern classes if and only if it is preperiodic.
\end{lemma}

\begin{proof}
The inverse Cartier $C^{-1}$ sends slope semistable Higgs bundles to slope semistable integrable connections. Any semistable integrable log-connection has a distinguished gr-semistable Griffiths transverse filtration, the \emph{Simpson filtration}, whose field of definition is the same as the field of definition of the integrable connection \cite[Theorem 5.5]{Lan14}.

Let $\mathcal{M}$ be the moduli space of semistable logarithmic Higgs bundles of rank $r$ on $(X_1,S_1)$ with trivial $\mathbb{Q}_\ell$-Chern classes. This is a finite dimensional moduli space by a boundedness result of Langer \cite[Theorem 1.2]{Lan19}. There exists a finite field $\mathbb{F}_q$ such that $X_1$ and this moduli space are both defined. For any $f\geq1$ and any
\[(E,\theta) \in \mathcal{M}(\mathbb{F}_{q^f}),\]
by running the Higgs-de Rham flow initial with $(E,\theta)$, we obtain
\[(\mathrm{Gr}\circ C^{-1})^m((E,\theta)) \in \mathcal{M}(\mathbb{F}_{q^f}) \qquad \text{ for any } m\geq0,\]
where $\mathrm{Gr}$ is the associated graded of the Simpson filtration. Indeed, here we are implicitly using that the Simpson filtration is defined over the field of definition of the integrable connection and also that the associated graded is semi-stable. Since the set of all $\mathbb{F}_{q^f}$-points of $\mathcal{M}$ is finite, the Higgs bundle $(E,\theta)$ is preperiodic.

Coversely, by assumption there exists a filtration $0=N^t\subset N^{t-1}\subset \cdots \subset N^0=(E,\theta)$ by sub-Higgs sheaves such that $N_i=N^i/N^{i-1}$ has trivial Higgs field. Then $C^{-1}(N_i)=F_{X_1/\overline{\mathbb{F}}_p} N_i$ and therefore
\[c_i\left(\mathrm{Gr}\circ C^{-1}(E,\theta)\right) = c_i(C^{-1}(E,\theta)) = p c_i(E,\theta)\]
where $c_i$ denotes the $i$-th Chern class. Thus preperiodicity of $(E,\theta)$ implies vanishing of the Chern classes. Next, we show the semistability. Suppose $(E,\theta)$ is not semistable; then there exists a sub-Higgs sheaf $(F,\theta)\subset (E,\theta)$ with slope $\mu(F)>0$. By running the Higgs-de Rham flow, one obtains a sequence of sub-Higgs bundles with unbounded slopes
\[(\mathrm{Gr}\circ C^{-1})^m((F,\theta)) \subset (\mathrm{Gr}\circ C^{-1})^m((E,\theta)).\]
This is impossible, because there are finitely many isomorphism classes of Higgs bundles in
\[\{\mathrm{Gr}\circ C^{-1})^m((E,\theta)) \mid m\geq 0\}\]
by the preperiodicity of $(E,\theta)$.
\end{proof}

\section{de Rham/crystalline Local systems and graded Higgs bundles}\label{section:crystalline_dR_Higgs}

 \begin{setup}\label{setup:scheme_rigid}Let $Y/W$ be a smooth (not necessary projective) scheme with geometrically connected generic fiber. Denote by
 \begin{enumerate}
 \item $\mathcal{Y}$ the $p$-adic formal completion of $Y$ along the special fiber $Y_1$; and by
 \item $\mathcal{Y}_K$ the rigid-analytic generic fiber of the formal scheme $\mathcal{Y}$, which is an open subset of $Y_K^{\rm an}$, the analytic space associated to the generic fiber $Y_K$ as in \cite[p. 529]{Con99}.
 \end{enumerate}
 \end{setup}

\emph{lisse $\mathbb{Z}_p$-sheaf} over the small \'etale site $Y_{K,\et}$
(resp. $\mathcal{Y}_{K,\et}$)
is an inverse system $\mathbb{L} = \{\mathbb{L}_n\}_{n\geq1}$
where $\mathbb{L}_n$ is an \'etale $\mathbb{Z}/p^n\mathbb{Z}$-local system over $Y_K$ (resp. $\mathcal{Y}_K$)
and the transition maps $\mathbb{L}_{n+1}\to\mathbb{L}_n$ induce isomorphisms $\mathbb{L}_n\cong \mathbb{L}_{n+1}\pmod{p^{n}}$.
A \emph{lisse $\mathbb{Z}_p$-sheaf} over $Y_{K,\et}$ (resp. $\mathcal{Y}_{K,\et}$)
is also called an \emph{(\'etale) $\mathbb{Z}_p$-local system} over $Y_K$ (resp. $\mathcal{Y}_K$).
A lisse $\mathbb{Z}_p$-sheaf over $\mathcal{Y}_{K,\et}$ can be viewed as
a lisse $\widehat{\mathbb{Z}}_p$-sheaf over $\mathcal{Y}_{K,\proet}$,
where $\widehat{\mathbb{Z}}_p = \varprojlim \mathbb{Z}/p^n\mathbb{Z}$ as sheaves on $\mathcal{Y}_{K,\proet}$.

 A $\mathbb{Z}_p$-local system $\mathbb{L}$ over the rigid-analytic space $\mathcal{Y}_K$ is called \emph{crystalline},
 if it has an associated Fontaine-Faltings module $M=(V,\nabla,\mathrm{Fil},\varphi)_{\mathcal{Y}}$, see the definition in the \autoref{section FFM}, over $\mathcal{Y}$ such that $\mathbb{L} = \mathbb{D}(M)$, where $\mathbb{D}$ is Faltings' $\mathbb{D}$-functor.
If $Y/W$ is proper, then every finite \'etale cover of $\mathcal{Y}_K$ extends to a finite \'etale cover on $Y_K$. Therefore, every local system over $\mathcal{Y}_K$ can be extended to a unique local system on $Y_K$.\footnote{This argument occurs on p. 42 of \cite[Theorem 2.6*]{Fal89}. We note that the details of the argument involving formal GAGA is not given in \cite{Fal89} and was only later provided in \cite[Appendix]{Tsu96}.} Under the assumption $Y/W$ is proper, we will call a $\mathbb{Z}_p$-local system over the $K$-scheme $Y_K$ \emph{crystalline}, if its restriction on $\mathcal{Y}_K$ is crystalline. For the detailed definitions, see Appendix \ref{section FFM}.

In \cite[Definition 7.5 and Definition 8.3]{Sch13}, Scholze defined de Rham $\mathbb{Z}_p$-local system over proper smooth adic space. This definition extends to non-proper rigid analytic spaces without modification. A $\mathbb{Z}_p$-local system $\mathbb{L}$ on $\mathcal{Y}_K$ is said to be \emph{de Rham} if there exists a filtered de Rham bundle $(\mathcal{E},\nabla,\mathrm{Fil})_{\mathcal{Y}_K}$, see \autoref{def_Hodgefiltration}, over $\mathcal{Y}_K$ such that
\[\mathcal{E}_{\mathcal{Y}_K} \otimes_{\mathcal{O}_{\mathcal{Y}_K}} \mathcal{O}\mathbb{B}_{\mathrm{dR}} \simeq \mathbb{L} \otimes_{\mathbb{Z}_p} \mathcal{O}\mathbb{B}_{\mathrm{dR}}.\footnote{Here, this isomorphism is of sheaves on the pro-\'etale site of $\mathcal{Y}_K$, the rigid-analytic generic fiber.}\]
According \cite[Theorem 3.9(iv)]{LZ17}, a local system is de Rham if its stalk at a classical point regarded as a $p$-adic representation of the residue field is de Rham.
A $\mathbb{Z}_p$-local system over $Y_K$ is called de Rham, if the restriction of the local system on $\mathcal{Y}_K$ is de Rham.

Note that Faltings also defined an ``associated relation'' between filtered convergent $F$-isocrystals and smooth $\mathbb{Q}_p$-adic \'etale sheaves in the paragraph before \cite[Lemma 5.5]{Fal89}. If a filtered convergent $F$-isocrystal $\mathcal{E}$ comes from a Fontaine-Faltings module $M$, then there is a representation associated to $\mathcal{E}$, which is just the dual of $\mathbb{D}(M)\otimes \mathbb{Q}_p$, see the last remark in \cite[section V]{Fal89}. Therefore, in the case of $\mathbb{Q}_p$-coefficients, the ``associated relation'' is a generalization of crystalline representations defined in \cite[Theorem 2.6*]{Fal89}. This was reformulated by Tan-Tong in \cite[Definition 3.10]{tantong2019} using the pro-\'etale site.
In order to distinguish these two notions, we refer to Tan-Tong's reformulation as \emph{TT-crystalline local systems (or representations)} in this note. More explicitly, Tan-Tong first constructed the crystalline period sheaf $\mathcal{O}\mathbb{B}_{\mathrm{cris}}$. A pro-\'etale local system $\mathbb{L}$ on $\mathcal{Y}_K$ is then called a \emph{TT-crystalline local system} if there exists a filtered $F$-isocrystal on $Y_1$ with realization $(\mathcal{E},\nabla,\mathrm{Fil})_{\mathcal{Y}_K}$ over $\mathcal{Y}_K$ such that there is an isomorphism
\[\mathcal{E}_{\mathcal{Y}_K} \otimes_{\mathcal{O}_{\mathcal{Y}_K}} \mathcal{O}\mathbb{B}_{\mathrm{cris}} \simeq \mathbb{L} \otimes_{\mathbb{Z}_p} \mathcal{O}\mathbb{B}_{\mathrm{cris}}\]
 preserving the connection, filtration, and Frobenius.

 Using the natural inclusion $\mathcal{O}\mathbb{B}_{\mathrm{cris}}\hookrightarrow \mathcal{O}\mathbb{B}_{\mathrm{dR}}$ \cite[Corollary 2.25(1)]{tantong2019}, any TT-crystalline local system over $\mathcal{Y}_K$ is also de Rham.
 For a given crystalline local system $\mathbb{L}$ on $\mathcal{Y}_K$, according to the paragraph before \cite[Lemma 5.5]{Fal89} and \cite[Proposition 3.21]{tantong2019}, one gets a TT-crystalline local system.

\begin{remark}\label{rem:de Rham crystalline}
Assume that $\mathbb{L}$ is a crystalline local system over $\mathcal{Y}_K$. Then there is an associated Fontaine-Faltings module $(M,\nabla,\mathrm{Fil},\varphi)_{\mathcal{Y}}$ over $\mathcal{Y}$. Since crystalline local systems are always TT-crystalline, they are always de Rham. Thus there is also a filtered de Rham bundle $(\mathcal{E},\nabla,\mathrm{Fil})_{\mathcal{Y}_K}$ over $\mathcal{Y}_K$ associated to $\mathbb{L}$. One then has
 \[(M,\nabla,\mathrm{Fil})_{\mathcal{Y}}\mid_{\mathcal{Y}_K} = (\mathcal{E},\nabla,\mathrm{Fil})_{\mathcal{Y}_K}^\vee.\]
The appearance of a dual is simply because Faltings' original $\mathbb{D}$-functor is contravariant.
\end{remark}

Now, we consider the logarithmic case. We emphasize here: while Faltings' claimed that his results on crystalline representations extend to the logarithmic setting, the details have never been carefully written down. Therefore, everything we write here is dependent on \autoref{conj:Faltings_Dlog}, which more or less says that Faltings' $\mathbb{D}$ functor extends to the logarithmic setting. We denote this extended functor as $\mathbb{D}^{\log}$.
\begin{setup}\label{setup:scheme_rigid_log}
Let $Y$, $Y_K$, $\mathcal{Y}$ and $\mathcal{Y}_K$ be given as in \autoref{setup:scheme_rigid}. let $Z/S$ be a relative simple normal crossing divisor in $Y$ and denote $U$ the complement of $Z$ in $Y$. Set $Z_K$, $\mathcal{Z}$, $\mathcal{Z}_K$, $U_K$, $\mathcal{U}$ and $\mathcal{U}_K$ to be spaces constructed analogously to those in \autoref{setup:scheme_rigid}.\end{setup}
Note that $\mathcal{U}_K$ is a $p$-adic rigid analytic open subset of $\mathcal{Y}^\circ_K:=\mathcal{Y}_K-\mathcal{Z}_K$, and in general it is strictly smaller. For instance, consider the following example: $Y=\mathrm{Spec}(W[T])$ and $Z$ is defined by equation $T=0$. Then $\mathcal{U}_K$ is the annulus in the analytification of the affine line $\mathbb{A}^1_K$ given by $|T|_p=1$, but $\mathcal{Y}^\circ_K$ is the annulus given by $0<|T|_p\leq1$.

Suppose \autoref{conj:Faltings_Dlog} holds. A $\mathbb{Z}_p$-local system over $\mathcal{Y}^\circ_K$ is called \emph{logarithmic crystalline} over $(\mathcal{Y}_K,\mathcal{Z}_K)$, if it comes from a logarithmic Fontaine-Faltings module over $(\mathcal{Y},\mathcal{Z})$ under the functor $\mathbb{D}^{\log}$.
A $\mathbb{Z}_p$-local system over $Y^\circ_K:=Y_K-Z_K$ is called \emph{logarithmic crystalline} over $(Y_K,Z_K)$, if its restriction on $\mathcal{Y}^\circ_K$ is logarithmic crystalline over $(\mathcal{Y}_K,\mathcal{Z}_K)$. See Appendix \ref{section:log_FFM} for precise details on the category of Fontaine-Faltings modules.

We rewrite a basic consequence of \autoref{conj:Faltings_Dlog} and \autoref{rem:algebraic} in the language of local systems.
\begin{proposition}\label{log loc.sys.} Suppose \autoref{conj:Faltings_Dlog} holds. Let $M$ be a logarithmic Fontaine-Faltings module over $(\mathcal{Y},\mathcal{Z})$. Denote by $\mathbb{L}_{\mathcal{U}_K}$ the local system over $\mathcal{U}_K$ associated to the representation $\mathbb{D}(M\mid_{\mathcal{U}})$ of $\pi_1^\et(\mathcal{U}_K)$. Then
\begin{enumerate}
 \item the local system $\mathbb{L}_{\mathcal{U}_K}$ extends to a local system $\mathbb{L}_{\mathcal{Y}^\circ_K}$ over $\mathcal{Y}^\circ_K$.
 \item Suppose $Y$ is proper over $W$. Then the local system $\mathbb{L}_{\mathcal{Y}^\circ_K}$ is algebraic. That is, there exists a local system $\mathbb{L}_{Y^\circ_K}$ over $Y_K^\circ=Y_K-Z_K$ such that $\mathbb{L}_{\mathcal{Y}^\circ_K} = \mathbb{L}_{Y^\circ_K}\mid_{\mathcal{Y}_K^\circ}.$
\end{enumerate}
Abusing notation, we denote by $\mathbb{D}^{\log}(M)$ both local systems $\mathbb{L}_{\mathcal{Y}^\circ_K}$ and $\mathbb{L}_{Y^\circ_K}$.
\end{proposition}

As explained in the introduction, when $S\neq \emptyset$, for our argument to work we require a fundamental compatibility in logarithmic $p$-adic Hodge theory: \autoref{conjecture:compatibility_filtered_dR}.
\begin{proposition} Maintain notation as in \autoref{setup} and \autoref{ques:Raju}. Suppose further that if $S\neq \emptyset$, then \autoref{conjecture:compatibility_filtered_dR} holds for $\rho_D$. There is an associated logarithmic Fontaine-Faltings module $(M,\nabla,\mathrm{Fil},\varphi)_{\mathcal{D}}$ over $(\mathcal{D},\mathcal{D}\cap\mathcal{S})$ to $\rho_D$. By taking the associated graded of the underlying filtered logarithmic de Rham bundle $(M,\nabla,\mathrm{Fil})_\mathcal{D}$, we obtain a logarithmic Higgs bundle $(E,\theta)_{\mathcal{D}}$ over $(\mathcal{D},\mathcal{D}\cap\mathcal{S})$. Then we have the following.
\begin{itemize}
 \item[(1).] There exists a filtered logarithmic de Rham bundle $(M,\nabla,\mathrm{Fil})_{X_K}$ over $(X_K,S_K)$ such that
 \[(M,\nabla,\mathrm{Fil})_{\mathcal{D}}\mid_{\mathcal{D}_K} = (M,\nabla,\mathrm{Fil})_{X_K} \mid_{\mathcal{D}_K}.\]
 Furthermore, the connection has nilpotent residues around $S_K$.
 \item[(2).] There exists a logarithmic graded semistable Higgs bundle $(E,\theta)_{X_K}$ over $(X_K,S_K)$ extending $(E,\theta)_\mathcal{D}\mid_{\mathcal{D}_K}$.
 \item[(3).] All Chern classes of $M_{X_K}$ and $E_{X_K}$ vanish.
\end{itemize}
\end{proposition}

\begin{proof}
To prove (1), we must construct a filtered logarithmic de Rham bundle.

If $S=\emptyset$, then the representation $\rho_X$ is de Rham by \cite[Theorem 1.3]{LZ17}. Then the existence of $(M,\nabla,\mathrm{Fil})_{X_K}$ follows from by \autoref{rem:de Rham crystalline}. In the logarithmic setting, we will use the functor $D^{\rm alg}_{\rm dR}$ in \cite[Theorem 1.1]{DLLZ} and then show its dual satisfies the requirements in (1).

Denote by $\mathbb{L}_{D_K^\circ}$ and $\mathbb{L}_{X_K^\circ}$ the corresponding local systems associated to $\rho_D$ and $\rho_X$. We first prove that the condition in \cite[Theorem 1.1]{DLLZ} is satisfied for the local system $\mathbb{L}_{X_K^\circ}\otimes \mathbb{Q}_p$ over $X^\circ_K$. i.e. $\mathbb{L}_{X_K^\circ}\mid_{(X^\circ_K)^{an}}$ is de Rham. Denote by $\mathcal{U}$ the $p$-adic completion of $D-D\cap S$ along its special fiber, and denote by $\mathcal{U}_K$ the analytic generic fiber of $\mathcal{U}$ which is an open subset of $(D^\circ_K)^{an}$.
By forgetting the logarithmic structure, the local system $\mathbb{L}_{D_K^\circ}\mid_{\mathcal{U}_K}$ is crystalline (the local version in the sense of \cite[II(g)]{Fal89}) with associated Fontaine-Faltings module $(M,\nabla,\mathrm{Fil},\varphi)_\mathcal{D}\mid_{\mathcal{U}}$. Then by \cite[Corollary 2.25(1)]{tantong2019}, $\mathbb{L}_{D_K^\circ}\mid_{\mathcal{U}_K}$ is also de Rham with attached filtered de Rham bundle
\begin{equation} \label{eq4.1.1}
 D_{\rm dR}(\mathbb{L}_{D_K^\circ}\mid_{\mathcal{U}_K}) = (M,\nabla,\mathrm{Fil})^\vee_\mathcal{D}\mid_{\mathcal{U}_K},
\end{equation}
where $D_{\rm dR}$ is Liu-Zhu's functor in \cite[Theorem 3.8]{LZ17}. Since $\mathbb{L}_{X_K^\circ}$ extends $\mathbb{L}_{D_K^\circ}$, by rigidity of de Rham local system \cite[Theorem 1.5(iii)]{LZ17}, $\mathbb{L}_{X_K^\circ}\mid_{(X^\circ_K)^{an}}$ is also de Rham.

Now, the functor $D^{\rm alg}_{\rm dR}$ in \cite[Theorem 1.1]{DLLZ} give us a filtered logarithmic de Rham bundle $D^{\rm alg}_{\rm dR}(\mathbb{L}_{X_K^\circ}\otimes \mathbb{Q}_p)$ over $(X_K,S_K)$. We denote $(M,\nabla,\mathrm{Fil})_{X_K}$ to be the dual of $D^{\rm alg}_{\rm dR}(\mathbb{L}_{X_K^\circ}\otimes \mathbb{Q}_p)$. That is
\begin{equation} \label{eq411}
 (M,\nabla,\mathrm{Fil})_{X_K} =D^{\rm alg}_{\rm dR}(\mathbb{L}_{X_K^\circ}\otimes \mathbb{Q}_p)^\vee.
\end{equation}
Finally, we show that $(M,\nabla,\mathrm{Fil})_{X_K}$ satisfies our requirement in (1). From the construction of $D^{\rm alg}_{\rm dR}$, it is the algebraization of the functor
$D_{\rm dR,\log}$ in \cite[Theorem 1.7]{DLLZ}. i.e.
\begin{equation}\label{eq412}
 D_{\rm dR,\log}(\mathbb{L}_{\mathcal{X}_K^\circ}\otimes \mathbb{Q}_p) =D^{\rm alg}_{\rm dR}(\mathbb{L}_{X_K^\circ}\otimes \mathbb{Q}_p)\mid_{\mathcal{X}_K}.
\end{equation}
By the functorial property of $D_{\rm dR,\log}$, one has
\begin{equation} \label{eq4.1.3}
D_{\rm dR,\log}(\mathbb{L}_{\mathcal{D}_K^\circ}\otimes \mathbb{Q}_p) = D_{\rm dR,\log}(\mathbb{L}_{\mathcal{X}_K^\circ}\otimes \mathbb{Q}_p) \mid_{\mathcal{D}_K}
\end{equation}
Under the hypothesis that \autoref{conjecture:compatibility_filtered_dR} holds for $\rho_D$, we have
\begin{equation}\label{eq414}
 (M,\nabla,\mathrm{Fil})_{\mathcal{D}}\mid_{\mathcal{D}_K}= D_{\rm dR}^{\rm alg}(\mathbb{L}_{\mathcal{D}_K^\circ} \otimes \mathbb{Q}_p)^\vee.
\end{equation}
Summing up the equations above, one has
\begin{equation*}
\begin{split}
 (M,\nabla,\mathrm{Fil})_{X_K}\mid_{\mathcal{D}_K}
 &
 \overset{\eqref{eq411}}{=}
 D^{\rm alg}_{\rm dR}(\mathbb{L}_{X_K^\circ}\otimes \mathbb{Q}_p)^\vee \mid_{\mathcal{D}_K}
 \overset{\eqref{eq412}}{=}
 D_{\rm dR,\log}(\mathbb{L}_{\mathcal{X}_K^\circ}\otimes \mathbb{Q}_p)^\vee \mid_{\mathcal{D}_K} \\
& \overset{\eqref{eq4.1.3}}{=}
 D_{\rm dR,\log}(\mathbb{L}_{\mathcal{D}_K^\circ}\otimes \mathbb{Q}_p)^\vee
 \underset{\autoref{conjecture:compatibility_filtered_dR}}{{\overset{\eqref{eq414}}{=}}}
 (M,\nabla,\mathrm{Fil})_{\mathcal{D}}\mid_{\mathcal{D}_K}.
\end{split}
\end{equation*}
As $\rho_D$ is logarithmic crystalline, we claim that the connection $(M,\nabla)_{\mathcal{D}}\mid_{\mathcal{D}_K}$ has nilpotent residues. We only need to consider this locally. We first choose a local Frobenius lifting $\Phi$ on $\mathcal{D}$ such that $\Phi^*(t_S) = t_S^p$, where $t_S$ is the local parameter of $D \cap S$. Recall that the Frobenius structure in the Fontaine-Faltings module induces an isomorphism $\Phi^*(M,\nabla)_{\mathcal{D}}\mid_{\mathcal{D}_K} \cong (M,\nabla)_{\mathcal{D}}\mid_{\mathcal{D}_K}$ of integrable connections over $\mathcal{D}_K$. Thus the set of eigenvalues of the residue around $S_K\cap D_K$ is closed under multiplying by $p$ map. As the set is finite, all eigenvalues must be zero, hence the residue is nilpotent. See also \cite[Def. 7.2]{Ked22} for this argument.

We further claim that the connection $\nabla_{X_K}$ also has nilpotent residues around $S_K$.\footnote{Another argument for the nilpotence of residues; the local monodromy of $\rho_D$ along $D_K\cap S_K$ is unipotent. This implies that the local monodromy of $\rho_X$ along $S_K$ is also unipotent because $D_K$ intersects each component of $S_K$ non-trivially and transversally. Then simply apply \cite[Theorem 3.2.12]{DLLZ}.}
We only need to show the eigenvalues of the residues are all zero. For the residue along each component of $S_K$, the coefficients of its characteristic polynomial are global sections of the structure sheaves over each component. On the other hand, by properness, these global sections are constant, hence the eigenvalues of the residue are also constant. Note that $D_K$ meets every component of $S_K$ as $D_K \subset X_K$ is ample. Thus, these eigenvalues are all zero.

To prove (2), by taking the associated graded of the filtered integrable connection in (1), one derives (2) from (1), except for the semistability of the graded Higgs bundle. Consider the associated graded Higgs bundle $(E,\theta)_{X_K}= \mathrm{Gr}((M,\nabla,\mathrm{Fil})_{X_K})$ over $X_K$, which extends the Higgs bundle $(E,\theta)_\mathcal{D}\mid_{D_K}$. Since $(E,\theta)_{X_K}\mid_{D_K}$ is semistable, it follows that $(E,\theta)_{X_K}$ is also semistable (Suppose $(E,\theta)_{X_K}$ is not semistable; then its maximal destabilizing sheaf $(F,\theta)$ is a subsheaf of lower rank with greater slope. Then its restriction $(F,\theta)\mid_{D_K}$ also has greater slope than $(E,\theta)_{X_K}\mid_{D_K}$, which contradicts the semistability of $(E,\theta)_{X_K}\mid_{D_K}$.).

To prove (3), we recall the fact that $M_{X_K}$ admits an integrable connection with logarithmic poles and nilpotent residues implies that its de Rham Chern classes all vanish \cite[Appendix B]{EV86}. By comparison between de Rham and $l$-adic cohomology, this implies that the $\mathbb{Q}_\ell$-Chern classes also vanish. So do the Chern classes of its grading $E_{X_K}$.
\end{proof}

 In summary, we have constructed a logarithmic Higgs bundle $(E,\theta)_{X_K}$ extending the Higgs bundle $(E_\mathcal{D},\theta_\mathcal{D})\mid_{D_K}$ that is semistable, and has (rationally) trivial Chern classes. In the next sections, we will extend $(E,\theta)_{X_K}$ to a Higgs \emph{bundle} (as opposed to merely a Higgs sheaf) on $X$ whose special fiber is also semistable.

\section{A theorem in the style of Langton}\label{section:langton}
In this section, we apply results due to Langer~\cite{Lan14,Lan19} in the vein of Langton to graded semistable logarithmic Higgs bundles. We fix an auxiliary prime $\ell\neq p$. (All subsequent results are independent of the choice of $\ell$.)

\begin{theorem}\label{thm:langton} Let $(E_{X_K},\theta_{X_K})$ be a graded semistable logarithmic Higgs bundle over $X_K$ such that the underlying vector bundle has rank $r\leq p$ and trivial $\mathbb{Q}_\ell$-Chern classes. Then there exists a semistable logarithmic Higgs bundle $(E_X,\theta_X)$ over $X$, which satisfies
 \begin{itemize}
 \item $(E_X,\theta_X)\mid_{X_K} \simeq (E_{X_K},\theta_{X_K})$;
 \item $(E_X,\theta_X)\mid_{X_1}$ is semistable over $X_1$.
 \end{itemize}
\end{theorem}

\begin{proof}
 The proof works in several steps.

 \begin{enumerate}
 \item[Step 1.]
 We first construct an auxiliary logarithmic coherent Higgs sheaf $(F_{X},\Theta_{X})$ on $X$ such that
 \begin{enumerate}
 \item $F_{X}$ is reflexive (and hence torsion-free);
 \item there is an isomorphism $(F_{X},\Theta_{X})\mid_{X_K} \cong (E_{X_K},\theta_{X_K})$.
 \end{enumerate}
 First of all, $E_{X_K}$ admits a coherent extension $F_X$ to $X$. Replacing $F_{X}$ with $F^{**}_{X}$, this is reflexive.

 We must construct a logarithmic Higgs field. We know that $F_{X}\subset F_{X_K}\cong E_{X_K}$. Work locally, over open $W$-affines $\mathcal{U}_{\alpha}$, and pick a finite set of generators $f_i$ for $F_{\mathcal{U}_{\alpha}}$. Then, for each $\alpha$, there exists an $r_{\alpha}\geq 0$ with

 $$\displaystyle \theta_{X_{K}}(f_i)\subset p^{-r_{\alpha}} F_{\mathcal{U}_{\alpha}}\otimes_{\mathcal{O}_{\mathcal{U}_{\alpha}}}\Omega^1_{\mathcal{U}_{\alpha}/W}(\log S\cap \mathcal{U}_{\alpha})$$
 for each $i$. As $X$ is noetherian, there exists a finite set of open affines $\{\mathcal{U}_\alpha\}_{\alpha}$, which cover $X$. By the finiteness of $\alpha$'s, there is a \emph{uniform} $r$ such that
 \begin{equation}\label{equ:extension_Higgs}
 \displaystyle p^ r\theta_{X_{K}}\colon F_{X}\rightarrow F_{X}\otimes_{\mathcal{O}_{X}}\Omega^1_{X/W}(\log S).
 \end{equation}

 Set $\Theta_{X}:=p^r\theta_{X_K}$; this makes sense by the above formula~\eqref{equ:extension_Higgs}. Then $(F_{X},\Theta_{X})$ is a Higgs sheaf on $X$. We claim that $(F_{X},\Theta_{X})_{X_K}\cong (E_{X_K},\theta_{X_K})$. Indeed, as $ (E_{X_K},\theta_{X_K})$ is a graded Higgs bundle, there exists an isomorphism $$ (E_{X_K},p^r\theta_{X_K})\cong (E_{X_K},\theta_{X_K}).$$
 Finally, as $F_X$ is torsion-free, it is automatically $W$-flat.
 \item [Step 2.]
 We now claim that there exists a logarithmic Higgs sheaf $(E_{X},\theta_{X})$ extending $(E_{X_K},\theta_{X_K})$ such that the logarithmic Higgs sheaf $(E_{X_1},\theta_{X_1})$ on the special fiber $X_1$ is semistable.
Choose an arbitrary reflexive extension $F = (\widetilde{E}_X, \widetilde{\theta}_X)$ as in the Step 1.
 Set $L$ to be the smooth Lie algebroid whose underlying coherent sheaf is the logarithmic tangent sheaf, and whose bracket and anchor maps are trivial.
 By \cite[Lemma 3.2]{Lan14} $F$ can be viewed as a $W$-flat $\mathcal{O}_X$-coherent $L$-module of relative pure dimension. By \cite[Theorem 5.1]{Lan14}, there exists an sub-$L$-module $E = (E_X, \theta_X) \subset F$ such that $E_K=(E_{X_K}, \theta_{X_K})$ and $E_k = (E_{X_1}, \theta_{X_1})$ is semistable. We emphasize that $(E_{X},\theta_{X})$ is merely a Higgs sheaf; we don't yet know it is a vector bundle.

 The $(E_{X},\theta_{X})$ constructed by Langer's theorem is torsion-free by design, hence also $W$-flat. We claim that the Chern classes of $(E_{X_1},\theta_{X_1})$ vanish because the Chern classes of $(E_{X_K},\theta_{X_K})$ do; this is what we do in \autoref{sect:local_constancy}.

 \item[Step 3.]
 The special fiber $(E_{X_1},\theta_{X_1})$ is a semistable logarithmic Higgs sheaf with rank $r\leq p$. Moreover, the Chern classes all vanish. Then it directly follows from \cite[Theorem 2.2]{Lan19} that $(E_{X_1},\theta_{X_1})$ is locally free.

 \item[Step 4.]
 Finally, it follows from the easy argument below that $E_{X}$ is a vector bundle. \qedhere
 \end{enumerate}
\end{proof}

\begin{lemma}Let $\mathcal{F}_X$ be a torsion-free and coherent sheaf over $X$. If $\mathcal{F}_X\mid_{X_1}$ is locally free, then $\mathcal{F}_X$ is locally free.
\end{lemma}
\begin{proof}
By \cite[Lemma 2.1.7]{HuLe10}, the localization $\mathcal{F}_x$ is free over $\mathcal{O}_{X,x}$ for any $x\in X_1$. As $X_1$ contains all closed points in $X$, $\mathcal{F}$ is locally free.
\end{proof}

We emphasize that we do not know the Higgs sheaf $(E_{X_1},\theta_{X_1})$ is graded in general. We will deduce this in our situation using a Lefschetz-style theorem.

\section{Local constancy of Chern classes}\label{sect:local_constancy}

To explain the local constancy of this section, we first need some preliminaries. Let $S$ be a separated scheme and let $f\colon X\rightarrow S$ be smooth projective morphism. Let $\ell$ be an odd prime number which is invertible on $S$. For any non-negative integer $i$ and any integer $j\in \mathbb{Z}$, the $\ell$-adic sheaf $R^if_*\mathbb{Z}_{\ell}(j)$ is a constructable, locally constant sheaf, i.e., a lisse $\ell$-adic sheaf, by the smooth base change theorem. Moreover, for any point $s\in S$, the proper base change theorem implies that the natural morphism $R^if_*\mathbb{Z}_\ell(j)_s\rightarrow R^i(f_s)_*\mathbb{Z}_\ell(j)$ is an isomorphism. In particular, if $\xi\in H^0(S,R^if_*\mathbb{Z}_\ell(j))$ then it is automatically ``locally constant". The following proposition is surely well-known. As we could not find a reference, we document a proof.

 \begin{proposition}\label{bundle_constancy}
 Let $S$ be a separated irreducible noetherian scheme. Let $f\colon X\rightarrow S$ be a smooth proper morphism. Let $\mathcal{E}$ be a vector bundle on $X$. If $\overline s$ is a geometric point such that the $\ell$-adic Chern classes of $\mathcal{E}_{\overline s}$ vanish, then the Chern classes of $\mathcal{E}$ vanish at every geometric point.
 \end{proposition}

 \begin{proof}
As discussed above, the result directly follows if we show that
\begin{itemize}

\item for every geometric point $\bar{s}$ of $S$ and non-negative integers $i, j$, the natural map $H^0(S, R^if_*\mathbb{Z}_{\ell}(j))\rightarrow H^i(X_{\bar s}, \mathbb{Z}_{\ell}(j))$ is injective; and
\item the Chern classes live in $H^0(S, R^{i}f_*\mathbb{Z}_\ell(j))$ for appropriate $i,j$.
\end{itemize}

We first explain how to show the former. As explained in the preamble, by a combination of the smooth and proper base change theorems, the map
$$R^if_*\mathbb{Z}_\ell(j)_{\bar s}\rightarrow R^i(f_{\bar s})_*\mathbb{Z}_\ell(j)$$ is an isomorphism of lisse $\mathbb{Z}_{\ell}$-sheaves on $X_{\bar s}$. On the other hand, by the equivalence of local systems and representations of fundamental group (which, as $S$ is noetherian, by definition reduces to the torsion case \cite[0DV4]{stacks-project}), the $\mathbb{Z}_{\ell}$-module $H^0(S, R^if_*\mathbb{Z}_{\ell}(j))$ is precisely the set of $\pi_1^\et(S, \bar{s})$-invariants of $R^if_*\mathbb{Z}_\ell(j)_{\bar s}$. By the explicit construction of the equivalence of categories, the induced map $$ H^0(S, R^if_*\mathbb{Z}_{\ell}(j))\cong H^i(X_{\bar s}, \mathbb{Z}_{\ell}(j))^{\pi_1^\et(S, \bar{s})}\rightarrow H^i(X_{\bar s}, \mathbb{Z}_{\ell}(j))$$ is simply the inclusion of the invariants.

The second bullet point amounts to \emph{defining} Chern classes in this level of generality. Again, this is well-known, but we indicate how to do this.
The key is the \emph{splitting principle}; see, e.g., \cite[02UK]{stacks-project} for a reference (in the context of Chow groups). Let $f\colon \mathbb{F}\mathcal{E}\rightarrow X$ be the full flag scheme associated to the vector bundle $\mathcal{E}$; then the following two properties hold:
\begin{itemize}
\item The induced map $f^*\colon H^i_{\et}(X,\mathbb{Q}_\ell)\rightarrow H^i_{\et}(\mathbb{F}\mathcal{E},\mathbb{Q}_\ell)$ is injective.
\item The vector bundle $f^*\mathcal{E}$ has a filtration whose subquotients are line bundles.
\end{itemize}

We first construct the \emph{first Chern class} for line bundles in the appropriate cohomology group. Let $L$ be a line bundle on $X$. Then the isomorphism class of $L$ gives a class in $H^1_{\et}(X,\mathcal{O}_X^*)$. Consider the Leray spectral sequence:
$$E^{pq}_2:=H^p_{\et}(S,R^qf_{*,\et}\mathcal{O}_X^*)\Rightarrow H_{\et}^{p+q}(X,\mathcal{O}_X^*)$$
By the low-degree terms of the Leray spectral sequence, there is a natural exact sequence
$$0\rightarrow H^1_{\et}(S,f_*\mathcal{O}_X^*)\rightarrow H^1_{\et}(X,\mathcal{O}_X^*)\rightarrow H^0(S,R^1f_*\mathcal{O}_X^*).$$

(See e.g. \cite[Eqns. 9.2.11.3 and 9.2.11.4 on p. 256-257]{FGAexplained}.) For a line bundle, we define the $\ell$-adic Chern class via the Kummer sequence, which induces a map

$$ R^1f_*(\mathcal{O}_X^*)\rightarrow R^2f_{*,\et}\mathbb{Z}_{\ell}(-1).$$
This map is compatible with base change and hence agrees with the notion of $\ell$-adic first Chern class for smooth projective varieties over fields. In particular, we obtain a class $c_1(L)$ in $H^0_{\et}(S,R^2f_{*}\mathbb{Z}_{\ell}(-1))$.

We now define Chern classes for rank $r>1$ vector bundles $\mathcal{E}$. By the splitting principle, we may suppose that $\mathcal{E}$ has a filtration whose subquotients are line bundles $\mathcal{L}_i$. Then the total Chern class is multiplicative in such iterated extensions:
$$\displaystyle 1+c_1(\mathcal{E}) + \cdots +c_r(\mathcal{E})= \prod_{i=1}^r (1 + c_1(\mathcal{L}_i)),$$
where multiplication is interpreted via the cup-product map, $R^{2i}f_{*}\mathbb{Z}_{\ell}(-i))\otimes R^{2j}f_{*}\mathbb{Z}_{\ell}(-j))\rightarrow R^{2(i+j)}f_{*}\mathbb{Z}_{\ell}(-(i+j)))$. In particular, the $j$-th Chern class of $\mathcal{E}$ is an element of $H^0_{\et}(X,R^{2k}f_{*}\mathbb{Z}_{\ell}(-k))$
and the result follows.
\end{proof}

\begin{corollary}\label{cor_const_chern}
Let $S$ be a separated irreducible noetherian scheme. Let $f\colon X\rightarrow S$ be a smooth proper morphism.
Let $\mathcal{E}$ be a coherent sheaf over $X$ with a finite locally free resolution
 \[0\rightarrow \mathcal{E}^m \rightarrow \mathcal{E}^{m-1} \rightarrow \cdots \rightarrow \mathcal{E}^{0} \rightarrow \mathcal{E}\rightarrow 0.\]
 Suppose $\mathcal{E}$ is flat over $S$, i.e. for all $x\in X$, the stalk of $\mathcal{E}$ at $x$ is flat over the local ring $\mathcal{O}_{S,s}$ where $s=f(x)\in S$. Suppose there is a point $s\in S$ such that the $\mathbb{Q}_{\ell}$-Chern classes of $\mathcal{E}_{s}$ vanish. Then for every point $s\in S$, the $\mathbb{Q}_\ell$-Chern classes of $\mathcal{E}_s$ vanish.

\end{corollary}
\begin{proof}
 We only need to show that the restriction on the fiber $X_s$ for any point $s\in S$ induces the following exact sequence
 \begin{equation} \label{exact_seq_restrict}
 0\rightarrow \mathcal{E}^m \mid_{X_s}\rightarrow \mathcal{E}^{m-1} \mid_{X_s} \rightarrow \cdots\rightarrow \mathcal{E}^0 \mid_{X_s} \rightarrow \mathcal{E} \mid_{X_s} \rightarrow 0.
 \end{equation}
 where $\mathcal{E}^{i} \mid_{X_s}$ is the restriction of $\mathcal{E}^{i}$ on $X_s$. Indeed, the Whitney sum formula implies that
 \[c\left(\mathcal{E} \mid_{X_s}\right) = \prod_{i=0}^{m} c\left(\mathcal{E}^i \mid_{X_s}\right)^{(-1)^i},\]
 where $c(\cdot)$ is the total Chern class, and then \autoref{bundle_constancy} would imply the result.

 In the following we prove the exactness of \eqref{exact_seq_restrict}.
Since exactness is a local property, we may reduce to the case where both $X=\mathrm{Spec}(B)$ and $S=\mathrm{Spec}(A)$ are affine. In this setting, we have $\mathcal{E}^i\mid_{X_s}(X_s) = \mathcal{E}^i(X)\otimes_Ak_s$, where $k_s$ denotes the residue field at $s\in S$. By shrinking $S$ and $X$ if necessary, there exists a finite free resolution of $B$-module
\begin{equation} \label{eq_glex}
0\to \mathcal{E}^m(X)
 \to \mathcal{E}^{m-1}(X)
 \to \cdots
 \to \mathcal{E}^{0}(X)
 \to \mathcal{E}(X)\to 0.
\end{equation}
By our assumption, all terms in \eqref{eq_glex} are flat over $A$, so the exactness of \eqref{eq_glex} is preserved under tensoring with any $A$-module. In particular, the sequence
\begin{equation}
0\to \mathcal{E}^m(X)\otimes_Ak_s
 \to \mathcal{E}^{m-1}(X)\otimes_Ak_s
 \to \cdots
 \to \mathcal{E}^{0}(X)\otimes_Ak_s
 \to \mathcal{E}(X)\otimes_Ak_s\to 0.
\end{equation}
is exact, which is precisely the affine version of \eqref{exact_seq_restrict}.
\end{proof}

In the proof of \autoref{thm:langton}, note that $E_X$ is perfect because $X$ is regular. Hence we may plug in \autoref{cor_const_chern}, as explained in Steps 3 and 4 of the proof of \autoref{thm:langton} to deduce that $(E_{\mathcal{X}},\theta_{\mathcal{X}})$ is a \emph{Higgs bundle}, i.e., such that the underlying vector bundle is locally free.

 \section{A Lefschetz-style theorem for morphisms of Higgs bundles, after Arapura}\label{section:arapura}

 In this section, we temporarily change the notation. Let $Y/k$ be a $d$-dimensional smooth projective variety defined over an algebraically closed field. Let $S$ be a normal crossing divisor. We first review a vanishing theorem of Arapura.

 Recall that a logarithmic Higgs bundle over $(Y,S)$ is a vector bundle $E$ over $Y$ together with an $\mathcal{O}_Y$-linear map
 \[\theta\colon E\rightarrow E\otimes \Omega_Y^1(\log S)\]
 such that $\theta\wedge \theta = 0$. This integrability condition induces a \emph{Higgs complex}
 \[DR(E,\theta) = (0 \rightarrow E \overset\theta\rightarrow E \otimes \Omega^1_Y(\log S)
 \overset\theta\rightarrow E \otimes \Omega^2_Y(\log S)
 \overset\theta\rightarrow \cdots).\]
 We set the \emph{Higgs cohomology} to be:
 \[ H^*_{\mathrm{Hig}}(Y,(E,\theta)):=\mathbb{H}(Y,DR(E,\theta)) \]
 One verifies directly that
\begin{equation} \label{H0Hig}
H^0_{\mathrm{Hig}}(Y,(E,\theta)) = \{e\in H^0(Y,E) \mid \theta(e)=0\}.
\end{equation}
 The following is a fundamental vanishing theorem due to Arapura, see \cite[Theorem 1 and Lemma 4.3]{Ar19}.
 \begin{theorem}\label{thm:Ara} Let $(E,\theta)$ be a nilpotent semistable logarithmic Higgs bundle on $(Y,S)$ with vanishing Chern classes in $H_\et^*(Y,\mathbb{Q}_\ell)$. Let $L$ be an ample line bundle on $Y$. Suppose that either $\mathrm{char}(k)=0$, or $\mathrm{char}(k)=p$, $(Y,S)$ is liftable modulo $p^2$, and $d+\mathrm{rank} E<p$. Then
 \begin{enumerate}
\item for any $i>d$, one has $\mathbb{H}^i(Y,DR(E,\theta)\otimes L) = 0$;
\item for any $i<d$, $\mathbb{H}^i\Big(Y,DR(E,\theta)\otimes L^\vee\big(-S\big)\Big) = 0$.
\end{enumerate}
\end{theorem}

 As Arapura notes, all one really needs to assume is that $c_1(E)=0$ and $c_2(E).L^{d-2}=0$ in $H_\et^*(Y,\mathbb{Q}_\ell)$. Note that if a Higgs bundle is semistable with vanishing Chern classes, then so is the dual. The second term follows directly from the first via Grothendieck-Serre duality.

 We use \autoref{thm:Ara} to prove a Lefschetz theorem.
\begin{setup}\label{setup:vanishing}Let $Y/k$ be a smooth projective variety over a perfect field of characteristic $p$ and of dimension $d\geq 2$. Let $S\subset Y$ be a simple normal crossings divisor (possibly empty). Let $ D \subset Y$ be a smooth ample divisor that meets $S$ transversely and such that $\mathcal{O}(D-S)$ is also ample. Let $j\colon D\rightarrow Y$ denote the closed embedding.

We suppose that $(Y,D,S)$ has a lifting $(\widetilde{Y},\widetilde{D},\widetilde{S})$ over $W_2(k)$, i.e. $\widetilde{Y}$ is a scheme over $W_2(k)$ with two closed subschemes $\widetilde{D}$, $\widetilde{S}$, where all three of the preceding schemes are flat over $W_2(k)$, such that $(\widetilde{Y},\widetilde{D},\widetilde{S})\otimes_{W(k)} k \cong (Y,D,S)$.
\end{setup}

 Let $(E,\theta)$ be a logarithmic Higgs bundle over $Y$. The Higgs field $\theta\mid_{D}$ on $E\mid_{D}$ is defined as the composite map as in the following diagram:
 \begin{equation*}
 \xymatrix{
 E\mid_{D} \ar[r]^(0.4){j^*\theta} \ar[dr]_{\theta\mid_{D}}
 & E\mid_{D}\otimes j^*\Omega^1_Y (\log S)
 \ar[d]^{\mathrm{id}\otimes\mathrm{d}j}\\
 & E\mid_{D} \otimes \Omega^1_{D} (\log (D\cap S)).
 }
 \end{equation*}
Then one has the following result.
 \begin{lemma}\label{lem:CohEquiv} Setup as in \autoref{setup:vanishing}. Let $(E,\theta)$ be a nilpotent semistable logarithmic Higgs bundle on $Y$ of rank $r$ with trivial Chern classes and semistable restriction on $D$. Suppose further that $d+r < p$. Then the restriction functor induces isomorphisms
 \[res\colon H^i_{\mathrm{Hig}}(Y,(E,\theta))
 \overset{\sim}{\rightarrow} H^i_{\mathrm{Hig}}(D,(E,\theta)\mid_{D})\]
 for all $i\leq d-2$ and an injection
 \[res\colon H^{d -1}_{\mathrm{Hig}}(Y,(E,\theta))
 \hookrightarrow
 H^{d -1}_{\mathrm{Hig}}(D,(E,\theta)\mid_{D})\]
 \end{lemma}

Let $(E,\theta)$ and $(E,\theta)'$ be two logarithmic Higgs bundles over $Y$. Denote $\mathcal{E}=\mathcal{H}\mathrm{om}(E,E')$, and for any local section $f\in \mathcal{E}$, denote $\Theta(f) = \theta'\circ f - f\circ \theta$. Then $\Theta$ defines a Higgs field on the vector bundle $\mathcal{E}$. We denote
\[\mathcal{H}\mathrm{om}((E,\theta),(E,\theta)') := (\mathcal{E},\Theta).\]
By \eqref{H0Hig}, we have $H^0_{\rm Hig}(\mathcal{E},\Theta)=\{f\colon E\to E'\mid \theta'\circ f = f\circ \theta\} =:\mathrm{Hom}((E,\theta),(E,\theta)')$.

 \begin{corollary}\label{cor:lefschetz_semistable_higgs}
 Setup as in \autoref{setup:vanishing}. Let $(E,\theta)$ and $(E,\theta)'$ be two nilpotent semistable logarithmic Higgs bundles over $Y$ of rank $r$ and $r'$ respectively, where $d+rr'< p$. Suppose further that both Higgs bundles have trivial $\mathbb{Q}_\ell$-Chern classes and semistable restrictions to $D$. Then one has
\begin{enumerate}
 \item an isomorphism
 \[\mathrm{Hom}((E,\theta),(E,\theta)') \simeq
 \mathrm{Hom}((E,\theta)\mid_{D},(E,\theta)'\mid_{D})\]
 \item an injection
 \[H^1_{\mathrm{Hig}}(Y,\mathcal{H}\mathrm{om}((E,\theta),(E,\theta)'))
 \hookrightarrow
 H^1_{\mathrm{Hig}}(D,\mathcal{H}\mathrm{om}((E,\theta)\mid_{D},(E,\theta)'\mid_{D})).\]
\end{enumerate}
 \end{corollary}

 \begin{proof}
 Since $\mathcal{H}\mathrm{om}((E,\theta),(E,\theta)')\mid_D \simeq \mathcal{H}\mathrm{om}((E,\theta)\mid_{D},(E,\theta)'\mid_{D})$, the Corollary follows from \autoref{lem:CohEquiv}, once we establish the semistability of both $\mathcal{H}\mathrm{om}((E,\theta),(E,\theta)')$ and $\mathcal{H}\mathrm{om}((E,\theta)\mid_{D},(E,\theta)'\mid_{D})$.

 By \autoref{footnote:preperiodic}, there exist two preperiodic Higgs-de Rham flows
\[\xymatrix{
& (V,\nabla,\mathrm{Fil})_1 \ar[dr] && (V,\nabla,\mathrm{Fil})_2 \ar[dr] && \cdots \\
(E,\theta) \ar[ur] && (E,\theta)_1 \ar[ur] && (E,\theta)_2 \ar[ur] &\\}\]
and
\[\xymatrix{
& (V,\nabla,\mathrm{Fil})'_1 \ar[dr] && (V,\nabla,\mathrm{Fil})'_2 \ar[dr] && \cdots \\
(E,\theta)' \ar[ur] && (E,\theta)'_1 \ar[ur] && (E,\theta)'_2 \ar[ur] &\\}\]
with initial terms by $(E,\theta)$ and $(E,\theta)'$ respectively.
Taking the Hom sheaves of the corresponding terms in these two flows yields a new preperiodic Higgs-de Rham flow with initial term
$\mathcal{H}\mathrm{om}((E,\theta),(E,\theta)')$
 \[\xymatrix@C=0cm{
& \mathcal{H}\mathrm{om} \left( (V,\nabla,\mathrm{Fil})_1, (V,\nabla,\mathrm{Fil})'_1 \right) \ar[dr]
&& \cdots \\
\mathcal{H}\mathrm{om} \left( (E,\theta)',(E,\theta)\right) \ar[ur] && \mathcal{H}\mathrm{om} \left( (E,\theta)'_1,(E,\theta)_1\right) \ar[ur] &\\}\]
 where the Hom of filtered integrable connections is defined analogously to that of Higgs bundles(see details in \cite[Section 2.1]{KYZ}). By \autoref{footnote:preperiodic} again, $\mathcal{H}\mathrm{om}((E,\theta),(E,\theta)')$ is semistable.
 Similarly, $\mathcal{H}\mathrm{om}((E,\theta),(E,\theta)')\mid_{D}$ is also semistable.
 \end{proof}

 \begin{proof}[Proof of \autoref{lem:CohEquiv}]
 In the following, we need to show the restriction functor induces an isomorphism between these two Higgs cohomology groups.
 By assumption in \autoref{setup:vanishing} that $\mathcal{O}(D-S)$ is ample,
 applying \autoref{thm:Ara} with $L=\mathcal{O}(D-S)$,
 one has
 \[\mathbb{H}^i(Y,DR(E,\theta)\otimes \mathcal{O}_Y(-D)) =0,\]
 for all $i<d$.
 The following exact sequence of complexes
 {\tiny{\begin{equation*} \xymatrix{
 0 \ar[d]&& 0 \ar[d] & 0 \ar[d] & 0 \ar[d] & \\
 DR(E,\theta)\otimes \mathcal{O}_Y(-D) \ar[d] \ar@{}[r]|(0.7){=} & 0 \ar[r] & E(-D) \ar[r]^(0.35){\theta} \ar[d]
 & E(-D) \otimes \Omega^1(\log S) \ar[r]^{\theta} \ar[d]
 & E(-D) \otimes \Omega^2(\log S) \ar[r]^(0.75){\theta} \ar[d]
 & \cdots \\
 DR(E,\theta) \ar[d] \ar@{}[r]|(0.7){=} & 0 \ar[r] & E \ar[r]^(0.35){\theta} \ar[d]
 & E \otimes \Omega^1(\log S) \ar[r]^{\theta} \ar[d]
 & E \otimes \Omega^2(\log S) \ar[r]^(0.75){\theta} \ar[d]
 & \cdots \\
 DR(E,\theta)\otimes j_*\mathcal{O}_{D} \ar[d] \ar@{}[r]|(0.7){=} & 0 \ar[r] & j_*(E\mid_{D}) \ar[r]^(0.35){\theta} \ar[d]
 & j_*(E\mid_{D}) \otimes \Omega^1(\log S) \ar[r]^{\theta} \ar[d]
 & j_*(E\mid_{D}) \otimes \Omega^2(\log S) \ar[r]^(0.75){\theta} \ar[d]
 & \cdots \\
 0&& 0 & 0 & 0 & \\
 }\end{equation*}}}
 induces
 \begin{equation*}
 H^i_{\mathrm{Hig}}(Y,(E,\theta)) \overset{\sim}{\longrightarrow}
 \mathbb{H}^i(Y,DR(E,\theta)\otimes j_*\mathcal{O}_{D})
 \end{equation*}
 for all $i\leq d -2$ and an injection
 \begin{equation*}
 H^{d-1}_{\mathrm{Hig}}(Y,(E,\theta)) \hookrightarrow
 \mathbb{H}^{d-1}(Y,DR(E,\theta)\otimes j_*\mathcal{O}_{D}).
 \end{equation*}
 On the other hand, let $\mathcal{I}$ denote the defining ideal sheaf of $D$ in $Y$. By assumption, the conormal bundle $\mathcal{I}/\mathcal{I}^2 = \mathcal{O}(-D)\mid_{D}$ is more negative than $\mathcal{O}_{D}(-D\cap S)$. Applying \autoref{thm:Ara} with $L=(\mathcal{I}/\mathcal{I}^2)^\vee\otimes \mathcal{O}_D(-D\cap S)$, we obtain, for all $i<d$, the vanishing
 \[\mathbb{H}^i(D,DR((E,\theta)\mid_{D}) \otimes \mathcal{I}/\mathcal{I}^2 [-1])=\mathbb{H}^{i-1}(D, DR((E,\theta)\mid_{D}) \otimes \mathcal{I}/\mathcal{I}^2)=0.\]
 The long exact sequence of the following exact sequence of complexes
 {\tiny{\begin{equation*}\xymatrix{
 0 \ar[d]&& 0 \ar[d] & 0 \ar[d] & 0 \ar[d] & \\
 DR((E,\theta)\mid_{D}) \otimes \mathcal{I}/\mathcal{I}^2 [-1] \ar[d] \ar@{}[r]|(0.7){=}
 & 0 \ar[r] & 0 \ar[r] \ar[d]
 & E\mid_{D} \otimes \mathcal{I}/\mathcal{I}^2 \ar[r]^{\theta\mid_{D}} \ar[d]
 & E\mid_{D} \otimes \Omega_{D}^1(\log S) \otimes \mathcal{I}/\mathcal{I}^2 \ar[r] \ar[d]
 & \cdots\\
 DR(E,\theta)\mid_{D} \ar[d] \ar@{}[r]|(0.7){=}
 & 0 \ar[r] & E\mid_{D} \ar[r]^(0.35){j^*\theta} \ar[d]
 & E\mid_{D} \otimes j^*\Omega^1(\log S) \ar[r]^(0.50){j^*\theta} \ar[d]
 & E\mid_{D} \otimes j^*\Omega^2(\log S) \ar[r]^(0.75){j^*\theta} \ar[d]
 & \cdots \\
 DR((E,\theta)\mid_{D}) \ar[d] \ar@{}[r]|(0.7){=}
 & 0 \ar[r] & E\mid_{D} \ar[r]^(0.35){\theta\mid_{D}} \ar[d]
 & E\mid_{D} \otimes \Omega_{D}^1(\log S) \ar[r]^{\theta\mid_{D}} \ar[d]
 & E\mid_{D} \otimes \Omega_{D}^2(\log S) \ar[r]^(0.75){\theta\mid_{D}} \ar[d]
 & \cdots \\
 0&& 0 & 0 & 0 & \\
 }\end{equation*}}}
 induces
 \begin{equation*}
 \mathbb{H}^i(D,DR(E,\theta)\mid_{D})
 \overset{\sim}{\longrightarrow}
 H^i_{\mathrm{Hig}}(D,(E,\theta)\mid_{D})
 \end{equation*}
 for all $i\leq d-2$ and an injection
 \begin{equation*}
 \mathbb{H}^{d-1}(D,DR(E,\theta)\mid_{D})
 \hookrightarrow
 H^{d-1}_{\mathrm{Hig}}(D,(E,\theta)\mid_{D}).
 \end{equation*}
 Then the lemma follows from the fact that
 \begin{equation*}
 \mathbb{H}^i(Y,DR(E,\theta)\otimes j_*\mathcal{O}_{D})
 \overset{\sim}{\longrightarrow}
 \mathbb{H}^i(D,DR(E,\theta)\mid_{D}). \qedhere
 \end{equation*}
 \end{proof}

 \section{The Higgs bundle \texorpdfstring{$(E_{\mathcal{X}}, \theta_{\mathcal{X}})$}{(EX, thetaX)} is graded.}\label{section:graded}

In this section, we aim to prove the Higgs bundle constructed in \autoref{thm:langton} is graded.

 \begin{proposition}\label{prop:graded}
 Setup as in \autoref{theorem:main}. Let $(E_{X_K},\theta_{X_K})$ be the logarithmic Higgs bundle attached to $\rho_X$ in \autoref{section:crystalline_dR_Higgs}. Let $(E_{\mathcal{X}}, \theta_{\mathcal{X}})$ be the $p$-adic completion of the Higgs bundle constructed in \autoref{thm:langton}. Then the Higgs bundle $(E_{\mathcal{X}},\theta_{\mathcal{X}})$ is graded.
 \end{proposition}

 To prove this, recall that a Higgs bundle $(E,\theta)$ is graded if and only if it is invariant under the $\mathbb{G}_m$ action, i.e., if
 \[(E,\theta) \simeq (E,t\theta) \quad \text{ for all } t\in \mathbb{G}_m.\]

Before giving the proof of \autoref{prop:graded}, we introduce some lemmas.
\begin{lemma}\label{lem_ss_EThetaXmidD1} Setup as in \autoref{prop:graded}.
The restriction of $(E_{\mathcal{X}},\theta_{\mathcal{X}})$ on the special fiber $D_1$ of $\mathcal{D}$ is still semistable.
\end{lemma}

\begin{proof}
By assumption,
 $(E_{\mathcal{X}},\theta_{\mathcal{X}})$ is a Higgs bundle with trivial Chern classes with a graded generic fiber and with a semistable special fiber. Since $(E_{X_1},\theta_{X_1})=(E_{\mathcal{X}},\theta_{\mathcal{X}})\mid_{X_1}$ is a graded semistable Higgs bundle with trivial Chern classes and rank $\leq p$, it is preperiodic by \autoref{footnote:preperiodic}, that is, there exists a preperiodic Higgs-de Rham flow $\mathrm{HDF}_{X_1}$ with initial term $(E_{X_1},\theta_{X_1})$.
 Restricting this preperiodic flow to $D_1$, yields a preperiodic Higgs-de Rham flow with initial term $(E_{X_1},\theta_{X_1})\mid_{D_1}$.
 Thus $(E_{\mathcal{X}},\theta_{\mathcal{X}})\mid_{D_1}=(E_{X_1},\theta_{X_1})\mid_{D_1}$ is semistable by \autoref{footnote:preperiodic}.
\end{proof}

The hypothesis that $\rho_D$ is geometrically absolutely residually irreducible implies the following stability.
\begin{lemma}\label{lem_s_EThetaDmidD1}
Setup as in \autoref{prop:graded}
The restriction of $(E_{\mathcal{D}},\theta_{\mathcal{D}})$ on the special fiber $D_1$ of $\mathcal{D}$ is stable.
\end{lemma}

\begin{proof}
By the setup of \autoref{ques:Raju}, we assume that $\rho_D$ is crystalline. \autoref{thm:periodic_crystalline} then implies that the $p$-adic Higgs bundle $(E_{\mathcal{D}},\theta_{\mathcal{D}})$ is periodic. In particular, its restriction on the special fiber $(E_{D_1},\theta_{D_1})_0:=(E_{\mathcal{D}},\theta_{\mathcal{D}})\mid_{D_1}$ initiates a periodic Higgs-de Rham flow:
 {\tiny{\begin{equation*} \label{eq_92}
 \xymatrix{
 & (V_{D_1},\nabla_{D_1}, \mathrm{Fil}_{D_1})_0 \ar[dr]
 && (V_{D_1},\nabla_{D_1}, \mathrm{Fil}_{D_1})_1 \ar[dr]
 && \cdots \\
 (E_{D_1},\theta_{D_1})_0 \ar[ur]
 && (E_{D_1},\theta_{D_1})_1 \ar[ur]
 && (E_{D_1},\theta_{D_1})_2 \ar[ur]\\
 } \end{equation*}}}
 Suppose $(E_{D_1},\theta_{D_1})_0$ is semistable yet not stable. Then it admits a sub-Higgs bundle $(E'_{D_1},\theta_{D_1})_0$ of strictly lower rank with zero slope. Running the Higgs-de Rham flow starting from this sub-Higgs bundle within the original flow yields a proper sub-flow
{\tiny{\begin{equation} \label{eq_93}
 \xymatrix{
 & (V'_{D_1},\nabla_{D_1}, \mathrm{Fil}_{D_1})_0 \ar[dr]
 && (V'_{D_1},\nabla_{D_1}, \mathrm{Fil}_{D_1})_1 \ar[dr]
 && \cdots \\
 (E'_{D_1},\theta_{D_1})_0 \ar[ur]
 && (E'_{D_1},\theta_{D_1})_1 \ar[ur]
 && (E'_{D_1},\theta_{D_1})_2 \ar[ur]\\
 } \end{equation}}}

By periodicity, we may identify $(E_{D_1},\theta_{D_1})_{mf}$ with $(E_{D_1},\theta_{D_1})_{0}$ for all $m\geq0$, where $f$ denotes the periodicity of the first flow. Under this identification, we get infinitely many slope-$0$ sub-Higgs bundles $\{(E'_{D_1},\theta_{D_1})_{mf}\mid m\geq0\}$ of $(E_{D_1},\theta_{D_1})_{0}$. We consider the moduli space $M$ of slope zero sub-bundle of $E'_{D_1,0}$. By Grothendieck's boundedness \cite[Lemma 1.7.9]{HuLe10}, $M$ is a quasi-projective variety defined over $k$. Since $k$ is finite, $M$ contains only finitely many $k$-points. In other words, there exists only finitely many different sub-bundles of degree $0$ defined over $k$. In particular, there exists some integer $m_1>m_0\geq0$ such
\[(E'_{D_1},\theta_{D_1})_{m_1f} = (E'_{D_1},\theta_{D_1})_{m_0f}.\]
After replacing $(E'_{D_1},\theta_{D_1})_{0}$ with $(E'_{D_1},\theta_{D_1})_{m_0f}$, the subflow \eqref{eq_93} is $(m_1-m_0)f$-periodic.
By \cite[Theorem 1.8]{LSZ13a}, one gets a proper sub-$\mathbb{F}_{p^{(m_1-m_0)f}}$-representation of $\rho_D \otimes_{\mathbb{Z}_{p^f}} \mathbb{F}_{p^{(m_1-m_0)f}}$,
contradicting the fact that $\rho_D$ is geometrically absolutely residually irreducible.
\end{proof}

\begin{lemma}[Lan-Sheng-Zuo]
\label{Lemma_Langton_style}
Let $R$ be a discrete valuation ring with fraction field $K$ and residue field $k$. Let $X/R$ be a smooth projective scheme and let $S\subset R$ be a relative (strict) normal crossings divisor. Let $(E_1,\theta_1)_X$ and $(E_2,\theta_2)_X$ be logarithmic Higgs bundles on $X$ that are isomorphic over $K$. Suppose that $(E_1,\theta_1)_{X_k}$ is semistable and $(E_2,\theta_2)_{X_k}$ is stable. Then $(E_1,\theta_1)_X$ and $(E_2,\theta_2)_X$ are isomorphic.
\end{lemma}

\begin{proof}
The non-logarithmic version may be found as Lemma~6.4 in \href{https://arxiv.org/pdf/1311.6424v1.pdf}{arXiv 1311.6424v1}. The argument in the logarithmic setting requires no changes. As the proof is short, we reproduce it here. Let $\alpha_K\colon (E_1,\theta_1)_{X_K}\rightarrow (E_2,\theta_2)_{X_K}$ be an isomorphism over the generic fiber. Let $\pi$ be a uniformizer of $R$. As $X$ is a noetherian topological space, there exists a minimal integer $a$ such that $\pi^a\alpha_K$ extends to a homomorphism:
\[\beta := \pi^a\alpha_K \colon (E_1,\theta_1)_X\rightarrow (E_2,\theta_2)_X\]
of Higgs bundles over $X$. (Note: $a$ could be negative.) By the minimality of $a$, $\beta\mod \pi$ must be non-zero.
As $(E_2,\theta_2)_{X_k}$ is Higgs stable and $(E_1,\theta_1)_{X_k}$ is Higgs semistable, one has
\[\mu((E_1,\theta_1)_{X_k})\leq \mu(\mathrm{im}(\beta\mod \pi)) \leq \mu((E_2,\theta_2)_{X_k}),\]
here $\mu$ stands for the slope.\footnote{See \cite[Definition 1.2.11]{HuLe10} and the discussion following it. Here, we select and fix an ample divisor, e.g. $D$, to define the degree of bundles over $X$.}
By the constancy of the Chern classes,
\[\mu((E_1,\theta_1)_{X_k}) = \mu((E_1,\theta_1)_{X_K})= \mu((E_2,\theta_2)_{X_K})= \mu((E_2,\theta_2)_{X_k}).\]
Thus $\mu(\mathrm{im}(\beta\mod \pi))=\mu((E_2,\theta_2)_{X_k})$. Since $(E_2,\theta_2)_{X_k}$ is stable, one has
\[\mathrm{im}(\beta\mod \pi) = (E_2,\theta_2)_{X_k}.\]
Thus $\beta\mod \pi$ is surjective. Since $(E_1,\theta_1)_{X_k}$ and $(E_2,\theta_2)_{X_k}$ have the same rank, $\beta\mod \pi$ is an isomorphism. Therefore $\beta$ is an isomorphism, as desired.
\end{proof}

\begin{corollary} \label{lem_identification}
Setup as in \autoref{prop:graded}.
The two Higgs bundles $(E_{\mathcal{X}},\theta_{\mathcal{X}})\mid_{\mathcal{D}}$ and $(E_{\mathcal{D}},\theta_{\mathcal{D}})$ over $\mathcal{D}$ are isomorphic.
\end{corollary}
\begin{proof}
 By \autoref{lem_ss_EThetaXmidD1} and \autoref{lem_s_EThetaDmidD1}, $(E_{\mathcal{X}},\theta_{\mathcal{X}})\mid_{D_1}$ is semistable and $(E_{\mathcal{D}},\theta_{D_1})$ is stable. Then the statement follows from \autoref{Lemma_Langton_style}.
 \end{proof}

 \begin{proof}[Proof of \autoref{prop:graded}]
 Recall that $(E_{\mathcal{D}}, \theta_{\mathcal{D}})$ is graded.
 For any $t\in \mathbb{G}_m$, one has
 \[f_{\mathcal{D}}\colon (E_{\mathcal{D}},\theta_{\mathcal{D}}) \simeq (E_{\mathcal{D}},t\theta_{\mathcal{D}}).\]
By \autoref{lem_identification} and \autoref{cor:lefschetz_semistable_higgs}, one gets $f_{X_1}\colon (E_{X_1},\theta_{X_1}) \simeq (E_{X_1},t\theta_{X_1})$. View $(E_{X_2},t\theta_{X_2})$ as a second lifting of $(E_{X_1},\theta_{X_1})$ via the mapping
 \[(E_{X_2},t\theta_{X_2}) \twoheadrightarrow (E_{X_1},t\theta_{X_1}) \xrightarrow[\simeq]{f_{X_1}^{-1}} (E_{X_1},\theta_{X_1}).\]
 Then $f_{X_1}$ can be lifted to $W_2$ if and only if $(E_{X_2},t\theta_{X_2})$ and $(E_{X_2},\theta_{X_2})$ are isomorphic liftings of $(E_{X_1},\theta_{X_1})$. By \cite[Proposition 4.2(2)]{KYZ}, the difference $c$ of the two liftings are located in
 \[H^1_{\mathrm{Hig}}\Big(\mathcal{E}\mathrm{nd}(E_{X_1},\theta_{X_1})\Big) \cong H^1_{\mathrm{Hig}}\Big(\mathcal{H}\mathrm{om}\big((E_{X_1},\theta_{X_1}),(E_{X_1},t\theta_{X_1})\big)\Big).\]
 Thus the class $c$
 \footnote{One can also construct the obstruction class directly. More precisely, choosing local liftings $f_i$ of $f_{X_1}$ (only as local $\mathcal{O}_{X_2}$-linear morphisms), one gets a $1$-cocyle $\Big(\big(\frac{f_i-f_j}{p}\big)_{i,j},\big(\frac{f_i\circ\theta_{X_2}-t\theta_{X_2}\circ f_i}{p}\big)_{i}\Big)$. This yields a well-defined class $c$ in $H^1_{\mathrm{Hig}}\Big(\mathcal{H}\mathrm{om}\big((E_{X_1},\theta_{X_1}),(E_{X_1},t\theta_{X_1})\big)\Big)$, i.e., it is independent of the choice of the local liftings. The proof is routine and similar as the proof of \cite[Theorem 4.1. 2)]{KYZ}.}
 is the obstruction to lift $f_{X_1}$ to $W_2$.

By \autoref{cor:lefschetz_semistable_higgs}, one has an injective map
 \[res\colon H^1_{\mathrm{Hig}}\Big(\mathcal{H}\mathrm{om}\big((E_{X_1},\theta_{X_1}),(E_{X_1},t\theta_{X_1})\big)\Big)
 \hookrightarrow
 H^1_{\mathrm{Hig}}\Big(\mathcal{H}\mathrm{om}\big((E_{D_1},\theta_{D_1}),(E_{D_1},t\theta_{D_1})\big)\Big).\]
 Since $f_{D_1}$ is liftable, the image of $c$ under $res$ vanishes. Thus $c=0$ and there is a lifting of $f_{X_1}$
 \[f_{X_{2}}'\colon (E_{X_{2}},\theta_{X_{2}})\longrightarrow (E_{X_{2}},t\theta_{X_{2}}).\]
 In general $f_{X_{2}}'\mid_{D_{2}}\neq f_\mathcal{D}\mid_{D_{2}}$. We consider the difference $c'\in H^0_{\rm Hig}(D_1,\mathcal{E}\mathrm{nd}(E_{D_1},\theta_{D_1}))$ between $f_{X_{2}}'\mid_{D_{2}}$ and $f_\mathcal{D}\mid_{D_{2}}$. Since the restriction map in 1) of \autoref{cor:lefschetz_semistable_higgs} is a bijection, we have a unique preimage $res^{-1}(c')$ of $c'$. Then we get a new isomorphism
 \[f_{X_{2}}\colon (E_{X_{2}},\theta_{X_{2}})\longrightarrow (E_{X_{2}},t\theta_{X_{2}}).\]
 by modifying the lifting $f_{X_{2}}'$ via $res^{-1}(c')$. By inductively lifting on each truncated level, we deduce that we may lift $f_{X_1}$ to an unique isomorphism $f_{\mathcal{X}}\colon (E_{\mathcal{X}},\theta_{\mathcal{X}}) \simeq (E_{\mathcal{X}},t\theta_{\mathcal{X}})$ such that $f_{\mathcal{X}}\mid_{\mathcal{D}}=f_\mathcal{D}$. Thus $(E_{\mathcal{X}},\theta_{\mathcal{X}})$ is graded.
 \end{proof}

 \section{Proof of \autoref{theorem:main} -- Flow \texorpdfstring{$\mathrm{HDF}_{X_1}$}{HDFX1} initialized with \texorpdfstring{$(E_{X_1},\theta_{X_1})$}{E,theta}}\label{section:HDR_X1}

 Setup as in \autoref{theorem:main}. By the work in \autoref{section:graded}, there is a \emph{graded} logarithmic Higgs bundle $(E_{\mathcal{X}},\theta_{\mathcal{X}})$ such that $(E_{X_K},\theta_{X_K}):=(E_{\mathcal{X}},\theta_{\mathcal{X}})\mid_{X_K}$ are semistable.
Our goal is to construct an $f$-periodic Higgs-de Rham flow over $\mathcal{X}$ with initial term $(E_{\mathcal{X}},\theta_{\mathcal{X}})$ such that its restriction on $\mathcal{D}$ is isomorphic to the $f$-periodic Higgs-de Rham flow in \eqref{equ:HDF} attached to the representation $\rho_D$ under \autoref{thm:periodic_crystalline}.
In this section, we will address the periodicity of $(E_{X_1},\theta_{X_1})$.

To ensure that consistent notation for the Higgs bundles in the to-appear Higgs-de Rham flows, we will re-denote $(E_{\mathcal{D}},\theta_{\mathcal{D}})$ as $(E_{\mathcal{D}},\theta_{\mathcal{D}})_0$,
$(E_{\mathcal{X}},\theta_{\mathcal{X}})$ as $(E_{\mathcal{X}},\theta_{\mathcal{X}})_0$,
$(E_{X_1},\theta_{X_1})\mid_{D_1}=(E_{\mathcal{X}},\theta_{\mathcal{X}})\mid_{D_1}$ as $(E_{D_1},\theta_{D_1})_0'$, and $(E_\mathcal{D},\theta_\mathcal{D})\mid_{D_1}$ as $(E_{D_1},\theta_{D_1})_0$.

 Since $(E_{X_1},\theta_{X_1})$ is a graded semistable (logarithmic) Higgs bundle with trivial Chern classes and $N\leq p$, it is preperiodic by \autoref{footnote:preperiodic}, i.e., there exists a preperiodic Higgs-de Rham flow $\mathrm{HDF}_{X_1}$ with initial term $(E_{X_1},\theta_{X_1})$. Restricting this preperiodic flow to $D_1$, one obtains a preperiodic Higgs-de Rham flow $\mathrm{HDF}_{X_1}\mid_{D_1}$ with initial term $(E_{D_1},\theta_{D_1})'_0$
{\tiny{\begin{equation}\xymatrix{
 & (V_{D_1},\nabla_{D_1}, \mathrm{Fil}_{D_1})'_0 \ar[dr]
 && (V_{D_1},\nabla_{D_1}, \mathrm{Fil}_{D_1})'_1 \ar[dr]
 && \cdots \\
 (E_{D_1},\theta_{D_1})'_0 \ar[ur]
 && (E_{D_1},\theta_{D_1})'_1 \ar[ur]
 && (E_{D_1},\theta_{D_1})'_2 \ar[ur]\\
 }\end{equation}}}
On the other hand, the restriction flow $\mathrm{HDF}_{\mathcal{D}}\mid_{D_1}$ is $f$-periodic with initial term $(E_{D_1},\theta_{D_1})_0$
 {\tiny{\begin{equation}\xymatrix{
 & (V_{D_1},\nabla_{D_1}, \mathrm{Fil}_{D_1})_0 \ar[dr]
 && (V_{D_1},\nabla_{D_1}, \mathrm{Fil}_{D_1})_1 \ar[dr]
 && \cdots \\
 (E_{D_1},\theta_{D_1})_0 \ar[ur]
 && (E_{D_1},\theta_{D_1})_1 \ar[ur]
 && (E_{D_1},\theta_{D_1})_2 \ar[ur]\\
 }\end{equation}}}

 \begin{proposition}
\label{prop_HDF_D1_coincide} The Higgs-de Rham flow $\mathrm{HDF}_{X_1}$ is $f$-periodic and satisfies
 $\mathrm{HDF}_{X_1}\mid_{D_1} = \mathrm{HDF}_{\mathcal{D}}\mid_{D_1}$.
 \end{proposition}

Before given the proof of \autoref{prop_HDF_D1_coincide}, we give two useful lemmas. The first one is a slightly strengthened version of \cite[Lemma 7.1]{LSZ13a}. The proof is exactly the same as that in \cite{LSZ13a}, and we leave it to the readers.
\begin{lemma} \label{Lem_LSZ_unique_Fil}
Let $Y$ be a smooth projective variety with a normal crossing divisor $S$, and let $(V, \nabla)$ be a flat bundle over $(Y,S)$. If there exists a Griffiths transverse filtration $\mathrm{Fil}$ on $(V,\nabla)$ such that the associated graded Higgs module $(E,\theta)$ is stable, then $\mathrm{Fil}$ is the unique gr-semistable filtration on $V$ up to a shift of index.
\end{lemma}

\begin{remark} \autoref{Lem_LSZ_unique_Fil} strengthens \cite[Lemma 7.1]{LSZ13a} in three respects: (1) it is formulated in the logarithmic setting; (2) it does not require the base field to be algebraically closed; and (3) it asserts the uniqueness of the gr-semistable filtration, rather than only the stable case.
\end{remark}

The second Lemma is about the uniqueness of Higgs-de Rham flows initial with a given stable Higgs bundle.
 \begin{lemma} \label{lem_uniqueHDF}
 Let $Y_2$ be a smooth $W_2$-scheme with a relative normal crossing divisor $S_2$. Let $(F,\theta)_0$ be a Higgs bundle over $(Y_1,S_1):=(Y_2,S_2)\otimes_{W_2} k$, which is stable and periodic. Then
\begin{enumerate}
 \item all Higgs bundles and integrable connections appeared in any periodic Higgs-de Rham flow initialized with $(F,\theta)_0$ are stable;
 \item periodic Higgs de Rham flows initialized with $(F,\theta)$ are unique up to isomorphism.
\end{enumerate}
 \end{lemma}

 \begin{proof} Let \begin{equation}\label{flow_Ftheta}
\xymatrix{
 & (V,\nabla,\mathrm{Fil})_0 \ar[dr]
 && (V,\nabla, \mathrm{Fil})_2 \ar[dr]
 && \cdots \\
 (F,\theta)_0 \ar[ur]
 && (F,\theta)_1 \ar[ur]
 && (F,\theta)_2 \ar[ur]\\
 }
\end{equation}
be a periodic Higgs-de Rham flow initialized with $(F,\theta)_0$.
By the periodicity, (1) follows from the following two points:
 \begin{itemize}
 \item If $(F,\theta)_{i+1}$ is stable, then $(V,\nabla)_i$ is stable.
 \item $(V,\nabla)_i$ is stable if and only if $(E,\theta)_i$ is stable.
 \end{itemize}
We explain each of these bullet points in turn. For the first, suppose that $(V',\nabla')$ a proper integrable sub-connection of $(V,\nabla)_i$ with $\mu(V')\geq \mu(V_i)$. Set $\mathrm{Fil}' = \mathrm{Fil}_i\mid_{V'}$. Then one gets a subsheaf $\mathrm{Gr}(V',\nabla',\mathrm{Fil}') \subset (E,\theta)_{i+1}$ with $\mu(\mathrm{Gr}(V',\mathrm{Fil}'))\geq \mu((E,\theta)_{i+1})$. Since $(E,\theta)_{i+1}$ is stable, $\mathrm{Gr}(V',\mathrm{Fil}') = 0$ or $=E_{i+1}$. Hence $V'=0$ or $V_{i}$. Thus $V_{i}$ is also stable.
For the second bullet point, the proof is the same as the first using the equivalence of categories furnished by the inverse Cartier functor.

For the proof of (2), we assume
\begin{equation}\label{flow_Ftheta'}
\xymatrix{
 & (V,\nabla,\mathrm{Fil})'_0 \ar[dr]
 && (V,\nabla, \mathrm{Fil})'_2 \ar[dr]
 && \cdots \\
 (F,\theta)_0 \ar[ur]
 && (F,\theta)'_1 \ar[ur]
 && (F,\theta)'_2 \ar[ur]\\
 }
\end{equation}
 is a second periodic flow initialized with $(F,\theta)_0$.
Taking the inverse Cartier functor, one has $(V,\nabla)'_0=(V,\nabla)_0$.
By \autoref{Lem_LSZ_unique_Fil},
$\mathrm{Fil}_{0}'= \mathrm{Fil}_{0}$. Taking the associated graded, one gets $(E,\theta)_1'=(E,\theta)_1$. Inductively, one shows that the two flows coincide with each other.
 \end{proof}

\begin{proof}[Proof of \autoref{prop_HDF_D1_coincide}]
By \autoref{lem_identification},
 $(E_{\mathcal{D}},\theta_{\mathcal{D}})_0\simeq (E_{\mathcal{X}},\theta_{\mathcal{X}})_0\mid_\mathcal{D}$. Now, we identify $(E_{\mathcal{D}},\theta_{\mathcal{D}})_0$ with $(E_{\mathcal{X}},\theta_{\mathcal{X}})_0\mid_{\mathcal{D}}$ with this isomorphism. In particular, one has
 \[(E_{D_1},\theta_{D_1})'_0 =(E_{\mathcal{X}},\theta_{\mathcal{X}})_0\mid_{D_1}= (E_{\mathcal{D}},\theta_{\mathcal{D}})_0\mid_{D_1}=:(E_{D_1},\theta_{D_1})_0.\]
 Then the equality $\mathrm{HDF}_{X_1}\mid_{D_1} = \mathrm{HDF}_{\mathcal{D}}\mid_{D_1}$ follows from \autoref{lem_uniqueHDF}.

 By the full faithfulness of the restriction functor as in \autoref{cor:lefschetz_semistable_higgs}, the map $\varphi_{D_1}\colon (E_{D_1},\theta_{D_1})_f \overset\sim\rightarrow (E_{D_1},\theta_{D_1})_0$ that witnesses the periodicity of $\mathrm{HDF}_{D_1}$ can be lifted canonically to a map $\varphi_{X_1}\colon (E_{X_1},\theta_{X_1})_f \overset\sim\rightarrow (E_{X_1},\theta_{X_1})_0$. This implies that the Higgs-de Rham flow $\mathrm{HDF}_{X_1}$ is also $f$-periodic.
\end{proof}

 \section{Proof of \autoref{theorem:main} -- Flow \texorpdfstring{$\mathrm{HDF}_{\mathcal{X}}$}{HDFX} initialized with \texorpdfstring{$(E_{\mathcal{X}},\theta_{\mathcal{X}})_0$}{EX,thetaX}}\label{section:HDR_X}

 In this section, we use two results of Krishnamoorthy-Yang-Zuo~\cite{KYZ} to lift $\mathrm{HDF}_{X_1}$, constructed in \autoref{section:HDR_X1}, onto $\mathcal{X}$.

 \begin{proposition}\label{prop_lift_HDF_W}
 There is an unique $f$-periodic Higgs-de Rham flow $\mathrm{HDF}_{\mathcal{X}}$ over $\mathcal{X}$, which lifts $\mathrm{HDF}_{X_1}$, with initial term $(E_{\mathcal{X}},\theta_{\mathcal{X}})$ and satisfying $\mathrm{HDF}_{\mathcal{X}}\mid_\mathcal{D} \simeq \mathrm{HDF}_{\mathcal{D}}$.
 \end{proposition}

 We prove this result inductively on the truncated level; in particular, we may assume we have already lifted $\mathrm{HDF}_{X_1}$ to an $f$-periodic Higgs-de Rham flow $\mathrm{HDF}_{X_n}$ over $X_n$, where $n\geq 1$ is a positive integer:
{\tiny{\begin{equation}\xymatrix{
 & (V_{X_n},\nabla_{X_n}, \mathrm{Fil}_{X_n})_0 \ar[dr]
 && (V_{X_n},\nabla_{X_n}, \mathrm{Fil}_{X_n})_{f-1} \ar[dr] \\
 (E_{X_n},\theta_{X_n})_0 \ar[ur]
 && \cdots \ar[ur]
 && (E_{X_n},\theta_{X_n})_f \ar@/^20pt/[llll]^{\varphi_n}\\
 }\end{equation}}}
 satisfying $\mathrm{HDF}_{X_n}\mid_{D_n} \simeq \mathrm{HDF}_{\mathcal{D}}\mid_{D_n}$ and with initial term $(E_{X_{n}},\theta_{X_{n}})_0:=(E_{X_n},\theta_{X_n})$.
 We only need to lift $\mathrm{HDF}_{X_n}$ to an $f$-periodic Higgs-de Rham flow $\mathrm{HDF}_{X_{n+1}}$ over $X_{n+1}$ satisfying $\mathrm{HDF}_{X_{n+1}}\mid_{D_{n+1}} \simeq \mathrm{HDF}_{\mathcal{D}}\mid_{D_{n+1}}$ and with initial term $(E_{X_{n+1}},\theta_{X_{n+1}})_0:=(E_{X_{n+1}},\theta_{X_{n+1}})$ as follows.

 First, taking the $(n+1)$-truncated level Cartier inverse functor $C_{n_1}^{-1}$(for the definition, see \cite[Section 4]{LSZ13a}), one gets
 \[(V_{X_{n+1}},\nabla_{X_{n+1}})_0 := C^{-1}_{n+1}((E_{X_{n+1}},\theta_{X_{n+1}})_0,(V_{X_n},\nabla_{X_n}, \mathrm{Fil}_{X_n})_{f-1},\varphi_n)\]
 which is an integrable connection. This integrable connection lifts $(V_{X_n},\nabla_{X_n})_0 $ and satisfies $(V_{X_{n+1}},\nabla_{X_{n+1}})_0 \mid_{D_{n+1}} = (V_{\mathcal{D}},\nabla_{\mathcal{D}})_0 \mid_{D_{n+1}}$.

 \begin{proposition}\label{prop:Fil_lifting}
 Notation as above. Inductively, for any $i\geq 0$,
\begin{enumerate}
 \item there is a unique Hodge filtration (see \autoref{def_Hodgefiltration}) $\mathrm{Fil}_{X_{n+1},i}$ on $(V_{X_{n+1}},\nabla_{X_{n+1}})_i$, which lifts $\mathrm{Fil}_{X_{n},i}$ and satisfies \[\mathrm{Fil}_{X_{n+1},i}\mid_{D_{n+1}} = \mathrm{Fil}_{\mathcal{D},i}\mid_{D_{n+1}}.\]

 \item taking the associated graded, we obtain a Higgs bundle
 \[(E_{X_{n+1}},\theta_{X_{n+1}})_{i+1} = \mathrm{Gr}(V_{X_{n+1}},\nabla_{X_{n+1}},\mathrm{Fil}_{X_{n+1}})_i,\]
 which lifts $(E_{X_{n}},\theta_{X_{n}})_{i+1}$ and satisfies
 \[(E_{X_{n+1}},\theta_{X_{n+1}})_{i+1}\mid_{D_{n+1}} = (E_{D_{n+1}},\theta_{D_{n+1}})_{i+1}.\]

 \item taking the $(n+1)$-truncated level Cartier inverse functor, one gets an integrable connection
 \[(V_{X_{n+1}},\nabla_{X_{n+1}})_{i+1} := C^{-1}_{n+1}((E_{X_{n+1}},\theta_{X_{n+1}})_i,(V_{X_n},\nabla_{X_n}, \mathrm{Fil}_{X_n})_{i},\mathrm{id})\]
 which lifts $(V_{X_n},\nabla_{X_n})_{i+1}$ and satisfies
 \[(V_{X_{n+1}},\nabla_{X_{n+1}})_{i+1} \mid_{D_{n+1}} = (V_{\mathcal{D}},\nabla_{\mathcal{D}})_{i+1} \mid_{D_{n+1}}.\]
\end{enumerate}
 \end{proposition}

 To prove this result, we need a result of Krishnamoorthy-Yang-Zuo~\cite{KYZ} about the obstruction to lifting the Hodge filtration. We recall this result.

 Let $(V,\nabla,\mathrm{F}^*)$ be a filtered de Rham bundle, see \autoref{def_Hodgefiltration} over $X_n/W_n$. Denote its modulo $p$ reduction by $({\overline{V}},{\overline{\nabla}},{\overline{\mathrm{F}}}^*)$. Let $(\widetilde{V},\widetilde{\nabla})$ be a lifting of the flat bundle $(V,\nabla)$ on $X_{n+1}$. The ${\overline{\mathrm{F}}}^*$ induces a Hodge filtration ${\overline{\mathrm{Fil}}}^*$ on $(\mathcal{E}\mathrm{nd}({\overline{V}}),{\overline{\nabla}}^{\mathrm{End}})$ defined by
 \[\overline{\mathrm{Fil}}^\ell \mathcal{E}\mathrm{nd}(\overline{V}) =\sum_{\ell_1}(\overline{V}/\overline{\mathrm{F}}^{\ell_1})^\vee \otimes \overline{\mathrm{F}}^{\ell_1+\ell-1}.\]
 As $(\mathcal{E}\mathrm{nd}({\overline{V}}),{\overline{\nabla}}^{\mathrm{End}},{\overline{\mathrm{Fil}}}^*)$ satisfies Griffiths transversality, the de Rham complex $(\mathcal{E}\mathrm{nd}( \overline{V}) \otimes \Omega^*_{X_1},\overline{\nabla}^{\mathrm{End}})$ induces the following complex, which we denote by $\mathscr C$:
 \[0\rightarrow
 \mathcal{E}\mathrm{nd}( \overline{V})/\overline{\mathrm{Fil}}^0 \mathcal{E}\mathrm{nd}( \overline{V})
 \overset{{\overline{\nabla}}^{\mathrm{End}}}{\rightarrow}
 \mathcal{E}\mathrm{nd}( \overline{V})/\overline{\mathrm{Fil}}^{-1} \mathcal{E}\mathrm{nd}( \overline{V})
 \otimes \Omega_{X_1}^1
 \overset{{\overline{\nabla}}^{\mathrm{End}}}{\rightarrow}
 \mathcal{E}\mathrm{nd}( \overline{V})/\overline{\mathrm{Fil}}^{-2} \mathcal{E}\mathrm{nd}( \overline{V})
 \otimes \Omega_{X_1}^2
 \rightarrow \cdots\]
 which is the quotient of $(\mathcal{E}\mathrm{nd}( \overline{V}) \otimes \Omega^*_{X_1},\overline{\nabla}^{\mathrm{End}})$ by the sub-complex $({\overline{\mathrm{Fil}}}^{-*}\mathcal{E}\mathrm{nd}({\overline{V}}) \otimes \Omega^*_{X_1},{\overline{\nabla}}^{\mathrm{End}})$, i.e., we have the following exact sequence of complexes of sheaves over $X_{1}$:
 \[0 \rightarrow
 ({\overline{\mathrm{Fil}}}^{-*}\mathcal{E}\mathrm{nd}({\overline{V}}) \otimes \Omega^*_{X_1},{\overline{\nabla}}^{\mathrm{End}})
 \rightarrow
 (\mathcal{E}\mathrm{nd}( \overline{V}) \otimes \Omega^*_{X_1},\overline{\nabla}^{\mathrm{End}})
 \rightarrow
 \mathscr C
 \rightarrow 0.
 \]
 Denote $(\overline{E},\overline{\theta}) = \mathrm{Gr}(\overline{V},\overline{\nabla},\overline{\mathrm{F}}^*)$. Then $(\mathcal{E}\mathrm{nd}(\overline{E}),\overline{\theta}^{\mathrm{End}}) = \mathcal{E}\mathrm{nd}((\overline{E},\overline{\theta}))$ is also a graded Higgs bundle. Here is the key observation about the complex $\mathscr C$: it is a successive extension of direct summands of the Higgs complex
 of the graded Higgs bundle $(\mathcal{E}\mathrm{nd}(\overline{E}),\overline{\theta}^{\mathrm{End}})$
 \[\left\{\left.(\mathcal{E}\mathrm{nd}(\overline{E})^{a-*,-a+*}\otimes \Omega^*_{X_1},{\overline\theta}^{\mathrm{End}})\right| a=0,1,2,\cdots \right\}.\]

 \begin{theorem}\cite[Theorem 3.9]{KYZ}\label{thm:KYZ} Notation as above.
\begin{enumerate}
 \item The obstruction of lifting the filtration $\mathrm{F}^*$ onto $(\widetilde{V},\widetilde{\nabla})$ with Griffiths transversality is located in $\mathbb{H}^1(X_1,\mathscr C)$.
 \item If the above the obstruction vanishes, then the lifting space is an $\mathbb{H}^0(X_1,\mathscr C)$-torsor.
\end{enumerate}
 \end{theorem}

 Associated to the filtered de Rham bundle $(V_{X_1},\nabla_{X_1},\mathrm{Fil}_{X_1})$ and $(V_{D_1},\nabla_{D_1},\mathrm{Fil}_{D_1}) = (V_{X_1},\nabla_{X_1},\mathrm{Fil}_{X_1})\mid_{D_1}$, one has complex $\mathscr C_{X_1}$ of sheaves over $X_1$
 \[0\rightarrow
 \mathcal{E}\mathrm{nd}( V_{X_1})/\overline{\mathrm{Fil}}^0 \mathcal{E}\mathrm{nd}(V_{X_1})
 \overset{\nabla_{X_1}^{\mathrm{End}}}{\longrightarrow}
 \mathcal{E}\mathrm{nd}(V_{X_1})/\overline{\mathrm{Fil}}^{-1} \mathcal{E}\mathrm{nd}(V_{X_1})
 \otimes \Omega_{X_1}^1
 \overset{\nabla_{X_1}^{\mathrm{End}}}{\longrightarrow} \cdots\]
 and complex $\mathscr C_{D_1}$ of sheaves over $D_1$
 \[0\rightarrow
 \mathcal{E}\mathrm{nd}( V_{D_1})/\overline{\mathrm{Fil}}^0 \mathcal{E}\mathrm{nd}(V_{D_1})
 \overset{\nabla_{D_1}^{\mathrm{End}}}{\longrightarrow}
 \mathcal{E}\mathrm{nd}(V_{D_1})/\overline{\mathrm{Fil}}^{-1} \mathcal{E}\mathrm{nd}(V_{D_1})
 \otimes \Omega_{D_1}^1
 \overset{\nabla_{D_1}^{\mathrm{End}}}{\longrightarrow} \cdots\]
 satisfying
 \[\mathscr C_{D_1} = \mathscr C_{X_1} \mid_{D_1}.\]

 \begin{proposition}\label{prop:res_obs&tor} The restriction induces
 \begin{enumerate}
 \item an isomorphism $\mathbb{H}^0(X_1,\mathscr C_{X_1}) \overset{\sim}\longrightarrow \mathbb{H}^0(D_1,\mathscr C_{D_1})$, and
 \item an injection $\mathbb{H}^1(X_1,\mathscr C_{X_1}) \hookrightarrow \mathbb{H}^1(D_1,\mathscr C_{D_1})$.
 \end{enumerate}
 \end{proposition}
 \begin{proof}
 Since the complex $\mathscr C$ is a successive extension of direct summands of the Higgs complex, one gets the results by \autoref{lem:CohEquiv} and the five lemma.
 \end{proof}

\begin{proof}[Proof of \autoref{prop:Fil_lifting}]
Set
\[(V,\nabla,\mathrm{F}^*):=(V_{X_n},\nabla_{X_n}, \mathrm{Fil}_{X_n})_{i}, \qquad (\widetilde{V},\widetilde{\nabla}):=(V_{X_{n+1}},\nabla_{X_{n+1}})_{i}.\]
The uniqueness of the Hodge filtration follows \autoref{thm:KYZ} and \autoref{prop:res_obs&tor}. For existence, we denote by $c\in \mathbb{H}^1(X_1,\mathscr C_{X_1})$ the obstruction to lift $\mathrm{F}^*$ onto $X_{n+1}$.
 Since $\mathrm{F}^*\mid_{D_n}=\mathrm{Fil}_{D_n,i}$ is liftable, the image $res(c) \in \mathbb{H}^1(D_1,\mathscr C_{D_1})$ of $c$ under $res$ vanishes. Since $res$ is an injection by \autoref{prop:res_obs&tor}, $c=0$. Thus, $\mathrm{F}^*$ is also liftable. We choose a lifting $\widetilde{\mathrm{F}}^*{}'$. Denote by $c'\in \mathbb{H}^0(D_1,\mathscr C_{D_1})$ the difference between $\widetilde{\mathrm{F}}^*{}'\mid_{D_{n+1}}$ and $\mathrm{Fil}_{\mathcal{D},i}\mid_{D_{n+1}}$. Since $res$ is an isomorphism, one uses $res^{-1}(c')$ to modify the original filtration $\widetilde{\mathrm{F}}^*{}'$. Then one gets a new filtration $\widetilde{\mathrm{F}}^*$ such that $\widetilde{\mathrm{F}}^*\mid_{D_{n+1}} = \mathrm{Fil}_{\mathcal{D},i}\mid_{D_{n+1}}$.

 The terms 2) and 3) follows from term 1) directly.
 \end{proof}

 Run the Higgs-de Rham flow with initial term
 \[((E_{X_{n+1}},\theta_{X_{n+1}})_0,(V_{X_n},\nabla_{X_n}, \mathrm{Fil}_{X_n})_{f-1},\varphi_n)\]
 together with the Hodge filtrations constructed as in \autoref{prop:Fil_lifting}. Then one constructs a Higgs-de Rham flow $\mathrm{HDF}_{X_{n+1}}$ over $X_{n+1}$
 {\tiny{\begin{equation}\xymatrix@C=0mm{
 & (V_{X_{n+1}},\nabla_{X_{n+1}}, \mathrm{Fil}_{X_{n+1}})_0 \ar[dr]
 && (V_{X_{n+1}},\nabla_{X_{n+1}}, \mathrm{Fil}_{X_{n+1}})_{f-1} \ar[dr]
 && \cdots \\
 (E_{X_{n+1}},\theta_{X_{n+1}})_0 \ar[ur]
 && \cdots \ar[ur]
 && (E_{X_{n+1}},\theta_{X_{n+1}})_f \ar[ur] \\
 }\end{equation}}}
 satisfying
 \begin{itemize}
 \item with initial term $((E_{X_{n+1}},\theta_{X_{n+1}})_0,(V_{X_n},\nabla_{X_n}, \mathrm{Fil}_{X_n})_{f-1},\varphi_n)$;
 \item $\mathrm{HDF}_{X_{n+1}}\mid_{D_{n+1}} \simeq \mathrm{HDF}_{\mathcal{D}}\mid_{D_{n+1}}$.
 \end{itemize}

 For the last step, we need to show this lifted flow is again $f$-periodic. In other words, we need the following result.
 \begin{lemma}\label{lem:periodic_map} There exists an isomorphism
 \[\varphi_{X_{n+1}}\colon (E_{X_{n+1}},\theta_{X_{n+1}})_f \overset\sim\longrightarrow (E_{X_{n+1}},\theta_{X_{n+1}})_0\]
 which lifts $\varphi_{X_n}$ and satisfying $\varphi_{X_{n+1}}\mid_{D_{n+1}} = \varphi_{D_{n+1}}$.
 \end{lemma}
 To prove this result, we need another result on the obstruction class to lifting a Higgs bundle.

 \begin{theorem}\cite[Proposition 4.2]{KYZ} \label{obs:Higgs} Let $(E,\theta)$ be a logarithmic Higgs bundle over $X_n$. Denote $(\overline{E},\overline{\theta}) = (E,\theta)\mid_{X_1}$. Then
 \begin{enumerate}
 \item if $(E,\theta)$ has a lifting $(\widetilde{E},\widetilde{\theta})$ to $X_{n+1}$, then the lifting set is an $H^1_{\rm Hig}(X_1,\mathcal{E}\mathrm{nd}((\overline{E},\overline{\theta})))$-torsor;
 \item the infinitesimal automorphism group of $(\widetilde{E},\widetilde{\theta})$ over $(E,\theta)$ is $H^0_{\rm Hig}(X_1,\mathcal{E}\mathrm{nd}((\overline{E},\overline{\theta})))$.
 \end{enumerate}
 \end{theorem}

The following Proposition follows directly from \autoref {cor:lefschetz_semistable_higgs}.
 \begin{proposition}\label{prop:res_obs&tor_Higgs} The restriction induces
 \begin{enumerate}
 \item an isomorphism $res\colon H^0_{\rm Hig}(X_1,\mathcal{E}\mathrm{nd}(E_{X_1},\theta_{X_1})) \overset{\sim}\longrightarrow H^0_{\rm Hig}(D_1,\mathcal{E}\mathrm{nd}(E_{D_1},\theta_{D_1}))$, and
 \item an injection $res\colon H^1_{\rm Hig}(X_1,\mathcal{E}\mathrm{nd}(E_{X_1},\theta_{X_1})) \hookrightarrow H^1_{\rm Hig}(D_1,\mathcal{E}\mathrm{nd}(E_{D_1},\theta_{D_1}))$.
 \end{enumerate}
 \end{proposition}

 \begin{proof}[Proof of \autoref{lem:periodic_map}] We identify $(E_{X_{n}},\theta_{X_{n}})_0$ and $(E_{X_{n}},\theta_{X_{n}})_f$ via $\varphi_{X_n}$. Since both $(E_{X_{n+1}},\theta_{X_{n+1}})_0$ and $(E_{X_{n+1}},\theta_{X_{n+1}})_f$ lift $(E_{X_{n}},\theta_{X_{n}})_0$, they differ by an element \[c\in H^1_{\rm Hig}(X_1,\mathcal{E}\mathrm{nd}(E_{X_1},\theta_{X_1})).\]
 Since $\mathrm{HDF}_{X_{n+1}}\mid_{D_{n+1}}$ is $f$-periodic, one has $res(c) =0 \in H^1_{\rm Hig}(D_1,\mathcal{E}\mathrm{nd}(E_{D_1},\theta_{D_1}))$. By the injection of the restriction map in 2) of \autoref{prop:res_obs&tor_Higgs}, $c=0$ and there is an isomorphism
 \[\varphi_{X_{n+1}}'\colon (E_{X_{n+1}},\theta_{X_{n+1}})_f \longrightarrow (E_{X_{n+1}},\theta_{X_{n+1}})_0.\]
 In general $\varphi_{X_{n+1}}'\mid_{D_{n+1}}\neq \varphi_{D_{n+1}}$. We consider the difference
 \[c'\in H^0_{\rm Hig}(D_1,\mathcal{E}\mathrm{nd}(E_{D_1},\theta_{D_1}))\]
 between $\varphi_{X_{n+1}}'\mid_{D_{n+1}}$ and $\varphi_{D_{n+1}}$. Since the restriction map in 1) of \autoref{prop:res_obs&tor_Higgs} is a bijection, we have a unique preimage $res^{-1}(c')$ of $c'$. We obtain a new isomorphism
 \[\varphi_{X_{n+1}}\colon (E_{X_{n+1}},\theta_{X_{n+1}})_f \longrightarrow (E_{X_{n+1}},\theta_{X_{n+1}})_0.\]
 by modifying the lifting $\varphi_{X_{n+1}}'$ via $res^{-1}(c')$. Then $\varphi_{X_{n+1}}$ satisfies our required property, i.e. $\varphi_{X_{n+1}}\mid_{D_{n+1}}=\varphi_{D_{n+1}}$.
 \end{proof}

This concludes the proof of \autoref{prop_lift_HDF_W}, which in turn proves \autoref{theorem:main} and \autoref{theorem:main_proj} by following the five steps outlined at the very end of \autoref{section:introduction}.

\section{Applications}\label{applications}
In this section, we provide the aforementioned applications. The notation in this section differs somewhat from \autoref{section:LSZ}-\autoref{section:HDR_X}. We have done this because we want the applications to be self-contained. Therefore, the notation for each application is contained within the statement in \autoref{section:introduction}.

\begin{proof}[Proof of \autoref{corollary:abelian_mixed}]
It is clear that if the abelian scheme extends, then the local system extends. We must prove that if the local system extends, then the abelian scheme extends. By \autoref{theorem:main}, $\rho_X$ is crystalline. As crystalline representations have (locally) constant Hodge-Tate weights, the Hodge-Tate weights are in particular in $[0,1]$, i.e., $\rho_X$ is a crystalline representation whose associated Fontaine-Faltings module lives in $\mathcal{M}\mathcal{F}^{\nabla}_{[0,1]}(X)$.

There is a $p$-divisible group $G_X$ over $X$ whose $p$-adic Tate module over $X_K$ is isomorphic to $\rho_X$ by \cite[Theorem 7.1]{Fal89}.
We claim that a polarization on $A_D\rightarrow D$ naturally induces a quasi-polarization on $G_{X}$.
To see this, first of all: if $Y/W$ is a smooth scheme,
then the functor $\mathbb{D}\colon \mathcal{MF}^{\nabla}_{[0,1]}(Y)\rightarrow \mathrm{Rep}_{\pi_1^\et(Y_K)}(\mathbb{Z}_p)$ is fully faithful by \cite[Theorem 2.6]{Fal89}.
In our case, we have a surjection $\pi_1^\et(D_K)\twoheadrightarrow \pi_1^\et(X_K)$.
It follows that a skew-symmetric isogeny $\rho_D\rightarrow \rho_D^*$ extends to a skew-symmetric isogeny $\rho_X\rightarrow \rho_X^*$,
and therefore any skew-symmetric isogeny $G_D\rightarrow G_D^t$ extends to a skew-symmetric isogeny $G_X\rightarrow G_X^t$.
It follows that any polarization $\lambda_D$ on $A_D\rightarrow D$ naturally yields a quasi-polarization on $G_X$.

\cite[Corollary 8.6]{KP20} implies that, as there is an abelian scheme $A_{D_1}\rightarrow D_1$ whose $p$-divisible group extends to $X_1$, there exists an abelian scheme $A_{X_1}\rightarrow X_1$ with $A_{X_1}[p^{\infty}]\cong G_{X_1}$.\footnote{While the statement of Corollary 8.6 of \emph{loc. cit.} takes as input curves rather than divisors, we can easily put ourselves in that situation by simply intersecting $D_1$ with ample divisors until we end up with a smooth ample curve. Alternatively, the \emph{proof technique} of \cite{KP20} applies mutatis mutandis.} Moreover, by the proof of Corollary 8.6 of \emph{loc. cit.} any polarization $\lambda_{D_1}$ on $A_{D_1}$ extends to a polarization on $A_{X_1}$ after possibly multiplying by a power of $p$. We may therefore pick a polarization $\lambda_D$ such that the polarization $\lambda_{D_1}$ on $A_{D_1}\rightarrow D_1$ extends to a polarization $\lambda_{X_1}$ on $A_{X_1}\rightarrow X_1$.

By our choice of polarization, it follows that $G_X$ is a deformation of $A_{X_1}[p^{\infty}]$ as a \emph{quasi-polarized} $p$-divisible group. Then Serre-Tate theory \cite[Theorem 1.2.1]{Kat81} implies that there is a formal abelian scheme $\mathcal{A}_{\mathcal{X}} \rightarrow \mathcal{X}$ that is polarizable.\footnote{For a reference, see e.g. \cite[1.4.5.3, 1.4.5.4]{Con06}, but we briefly explain how this works. Serre-Tate theory will imply that there is an isogeny $\mathcal{A}_{\mathcal{X}} \rightarrow \mathcal{A}^t_{\mathcal{X}}$ extending the polarizing isogeny $A_{X_1}\rightarrow A^t_{X_1}$. By pulling back the Poincar\'e bundle, this induces a line bundle on $\mathcal{A}_{\mathcal{X}}$ restricting to a line bundle on $A_{X_1}$ algebraically equivalent to the original line bundle and is in particular algebraic. However, the ampleness of a line bundle is open, so the induced line bundle on $\mathcal{A}_{\mathcal{X}}$ is ample, as desired.} Grothendieck's algebraization theorem \cite[III. Th\'eor\`eme 5.4.5]{Gro68SGA3} then implies that $\mathcal{A}_{\mathcal{X}}$ uniquely algebraizes to an abelian scheme $A_X\rightarrow X$, as desired.
\end{proof}

\begin{proof}[Proof of \autoref{corollary:abelian_C}]
We first show the equivalence of the first two items. Let $X$, $D$, $f_D\colon A_D\rightarrow D$, and $(E_X,\theta_X)$ be given as in \autoref{corollary:abelian_C}. We use the following fact: for any finitely generated $\mathbb{Z}$ algebra $R$ that is an integral domain, there exist infinitely many primes $p$ such that $R$ embeds in $\mathbb{Z}_p$. This follows from Cassels' embedding theorem \cite{Cas76}. Indeed, pick a set of generators $t_i$ of $R$ over $\mathbb{Z}$. Then there exists infinitely many $p$ such that $\mathrm{Frac}(R)$ embeds in $\mathbb{Q}_p$ and moreover such that the image of each of the $t_i$ is in $\mathbb{Z}_p^{\times}$. This implies the result as the $t_i$ generated $R$ as a $\mathbb{Z}$ algebra. We note that one can extend $R\to \mathbb{C}$ to an embedding $\mathbb{Z}_p\to \mathbb{C}$. Let $S\subset \mathbb{Q}_p$ and $T\subset \mathbb{C}$ be transcendence bases over $\mathrm{Frac}(R)$. There exists an injection $S\to T$, which extends uniquely to an embedding of fields $\mathbb{Q}(S)\to \mathbb{Q}(T)\subset \mathbb{C}$. Since $\mathbb{Q}_p$ is algebraic over $\mathbb{Q}(S)$ and $\mathbb{C}$ is algebraically closed, this embedding extends to an embedding $\mathbb{Q}_p\to \mathbb{C}$.

By spreading-out and the above observation, it follows that there exists infinitely many primes $p$ such that the following holds.
\begin{itemize}
\item $X$, $D$, $f_D\colon A_D\rightarrow D$, and $(E_X,\theta_X)$ may all be defined over a copy of $\mathbb{Z}_p\subset \mathbb{C}$, and moreover $X$, $D$, and $f_D\colon A_D\rightarrow D$ are all smooth and projective over $\mathbb{Z}_p$.
\item The Higgs bundle $(E_D,\theta_D)$ is stable modulo $p$. (Stability is an open condition.)
\item The Higgs bundle $(E_D,\theta_D)$ is 1-periodic over $\mathbb{Z}_p$. (The Higgs bundle is the associated graded of de Rham cohomology of a smooth projective morphism over $D$.)
\item If $A_D\rightarrow D$ has relative dimension $g$, then $p>4g^2+\dim(X)$.
\end{itemize}

Now, by applying \autoref{theorem:lefschetz_higgs}, we deduce that $(E_X,\theta_X)$ is 1-periodic. It follows from \cite[Theorem 1.4]{LSZ13a} that the hypotheses of \autoref{corollary:abelian_mixed} are all satisfied. We deduce that $A_D\rightarrow D$ (over $\mathbb{Z}_p$) extends to an abelian scheme $A_X\rightarrow X$. The result follows.

To prove the second part, we simply repeat the spreading out argument above and note that $\pi_1(D)\rightarrow \pi_1(X)$ is an isomorphism by the classical Lefschetz theorem. Then \autoref{corollary:abelian_mixed} implies that $A_D\rightarrow D$ extends to an abelian scheme $A_X\rightarrow X$.
\end{proof}

To prove our final application, we first require several preliminaries.
\begin{definition}
Let $X/k$ be a smooth variety over a field $k$ of characteristic $0$, and suppose that $X$ admits a $k$-rational point. The category $\mathrm{MIC}(X)$ consists of objects $(V, \nabla)$, where $V$ is a vector bundle on $X$ and $\nabla$ is an integrable connection on $V$; the morphisms in this category are $\mathcal{O}_X$-linear maps that are horizontal with respect to the connections.
\end{definition}

In fact, the category $\mathrm{MIC}(X)$ is a $k$-linear Tannakian category, which is neutralized by the fiber at any $k$-point $x$ of $X$, written $\omega_x$ \cite[Section 2]{Esn12}. We set the \emph{algebraic fundamental group} to be the Tannakian group $$\mathrm{Aut}^{\otimes}(\omega_x)=:\pi_1^{\mathrm{alg}}(X).$$
We emphasize that this group does not satisfy base change, see e.g. \cite[p. 3-4]{Esn12}.

\begin{lemma}\label{Lemma:MIC_extension}
Let $X/k$ be a smooth projective variety of dimension at least 2 over a field of characteristic 0. Let $D\subset X$ be a smooth ample divisor. Then the natural functor
$$\mathrm{MIC}(X)\rightarrow \mathrm{MIC}(D)$$
is fully faithful, and when $\dim(X)\geq 3$ and $k=\overline{k}$, it is an equivalence of categories.
\end{lemma}

\begin{proof}
We first prove the Lemma in the case $k=\overline{k}$. This argument is essentially contained in \cite[Proposition 2.1]{Esn12}. More precisely, when $\dim(X)\geq 3$, one may directly apply \cite[Proposition 2.1]{Esn12}, which shows that $\pi_1^{\mathrm{alg}}(X_{\overline k})\rightarrow \pi_1^{\mathrm{alg}}(D_{\overline k})$ is an isomorphism. When $\dim(X)=2$, then one simply repeats the first part of the argument, replacing the use of the isomorphism furnished by \cite[Th\'eor\`eme 1.2(b)]{Gro70} with the surjection furnished by \cite[Th\'eor\`eme 1.2(a)]{Gro70} to prove that $\pi_1^{\mathrm{alg}}(X_{\overline k})\rightarrow \pi_1^{\mathrm{alg}}(D_{\overline k})$ is surjective.\footnote{Both of these go under the name ``Grothendieck-Malcev'' theorem, and they use the fact that the topological fundamental group of a smooth, quasi-projective variety is finitely generated.}

Let $k$ be a field of characteristic $0$ (not necessarily algebraically closed). For any two integrable connections $(V,\nabla)$ and $(V',\nabla')$ in $\mathrm{MIC}(X)$, we must show the $k$-linear morphism
\[\varphi\colon \mathrm{Hom}_{\mathrm{MIC(X)}}((V,\nabla),(V',\nabla')) \to \mathrm{Hom}_{\mathrm{MIC(D)}}((V,\nabla)\mid_{D},(V',\nabla')\mid_{D})\]
induced by restriction is an isomorphism. By base change to the algebraic closure $\overline{k}$, we have the canonical identifications:
\[\mathrm{Hom}_{\mathrm{MIC(X_{\overline{k}})}}((V,\nabla)_{\overline{k}},(V',\nabla')_{\overline{k}}) = \mathrm{Hom}_{\mathrm{MIC(X)}}((V,\nabla),(V',\nabla'))\otimes_k\overline{k}\]
\[\mathrm{Hom}_{\mathrm{MIC(D_{\overline{k}})}}((V,\nabla)_{\overline{k}}\mid_{D_{\overline{k}}},(V',\nabla')_{\overline{k}}\mid_{D_{\overline{k}}}) = \mathrm{Hom}_{\mathrm{MIC(D)}}((V,\nabla)\mid_{D},(V',\nabla')\mid_{D})\otimes_k\overline{k}.\]
Under these identifications, the restriction morphism over $\overline{k}$
\[\varphi_{\overline{k}}\colon \mathrm{Hom}_{\mathrm{MIC(X_{\overline{k}})}}((V,\nabla)_{\overline{k}},(V',\nabla')_{\overline{k}}) \to \mathrm{Hom}_{\mathrm{MIC(D_{\overline{k}})}}((V,\nabla)_{\overline{k}}\mid_{D_{\overline{k}}},(V',\nabla')_{\overline{k}}\mid_{D_{\overline{k}}})\]
corresponds to $\varphi\otimes_k\mathrm{id}_{\overline{k}}$. By algebraically closed case (already proven), $\varphi_{\overline{k}}$ is an isomorphism.
Since base change to $\overline{k}$ is a conservative functor, it follows that $\varphi$ is also an isomorphism.
\end{proof}

We recall a theorem about the existence of gr-semistable Griffiths-transverse filtration from \cite[Theorem 2.5]{Sim10} and \cite[Theorem A.4]{LSZ13a}
\begin{theorem}[Simpson, Lan-Sheng-Zuo] \label{thm_SimpsonFil}
Suppose $(V,\nabla)$ be a $\nabla$-semistable bundle over a smooth projective variety $Y$.
Then there exists a Griffiths transverse filtration $\mathrm{Fil}$ such that the graded Higgs module associated to $(V,\nabla,\mathrm{Fil})$ is semistable.
\end{theorem}

\begin{remark}
Simpson provided the construction only for curves, but it indeed applies to higher-dimensional varieties over arbitrary fields. This is precisely what Lan-Sheng-Zuo accomplished.
While their theorem specifies that the base field should be an algebraically closed field, this condition was not actually utilized in their proof.
The outline of the proof involves using the maximal destabilizing sub-object to modify a given Griffiths transverse filtration; this process is repeated iteratively.
Notably, the condition that the base field is algebraically closed is not required.
We call the filtration in the theorem the \emph{Simpson filtration}, which is, by construction, defined over the field of definition of the integrable connection; moreover, the construction is compatible with change of base field.
\end{remark}

\begin{proof}[Proof of \autoref{corollary:almost_all_p}]
We remind the reader that in this corollary, the notation is somewhat different from the rest of the paper.
In particular, $X/\mathcal{O}_K[1/\mathcal{N}]$ is a smooth projective scheme over a ring of integers in a number field, and $D\subset X$ is a relative smooth ample divisor.
By \autoref{Lemma:MIC_extension}, the integrable connection $(\mathcal{H}^i_\mathrm{dR}(Y_D/D),\nabla_{\mathrm{GM}})_{\bar K}$ extends to an integrable connection $(\mathcal{H}, \nabla)$ on $X_{\bar K}$.
However, integrable connections are finitary objects; hence after replacing $K$ by a finite extension, we know that the extension $(\mathcal{H}, \nabla)$ may be defined over $K$.
Moreover, by increasing $\mathcal{N}$, we can assume that $(\mathcal{H}, \nabla)$ may be defined over $\mathcal{O}_K[1/\mathcal{N}]$.
Pick an embedding $\iota\colon K\hookrightarrow \mathbb{C}$.
The integrable connection $(\mathcal{H},\nabla)_{\mathbb{C}}$ corresponds under Riemann-Hilbert, by construction (i.e., the proof of \cite[Proposition 2.1]{Esn12}), to the topological local system
\[\pi_1(X_{\mathbb{C}})\cong \pi_1(D_{\mathbb{C}})\rightarrow \mathrm{GL}_N(\mathbb{Z})\subset \mathrm{GL}_N(\mathbb{C}).\]
The local system factors through $ \mathrm{GL}_N(\mathbb{Z})$ because it originates from a family $f_D\colon Y_D\to D$ by taking relative Betti cohomology.
By \autoref{theorem:simpson}, the representation $\pi_1(X_{\mathbb{C}})\rightarrow \mathrm{GL}_N(\mathbb{C})$ underlies a $\mathbb{Z}$-PVHS; moreover, we claim that this $\mathbb{Z}$-PVHS extends the $\mathbb{Z}$-PVHS on $D_{\mathbb{C}}$.
This follows from the functoriality of the non-abelian Hodge correspondence.
In particular, there is an induced Griffiths transverse filtration $\mathrm{Fil}_{\mathbb{C}}$ on $(\mathcal{H},\nabla)_{\mathbb{C}}$, with
\[\mathrm{Fil}_{\mathbb{C}}\mid_{D_K} = \mathrm{Fil}_{\rm Hodge}\times_{K,\iota}\mathbb{C}.\]
To prove (1) and (4), we only need to show $\mathrm{Fil}_{\mathbb{C}}$ has a unique $K$-descent $\mathrm{Fil}$ such that $\mathrm{Fil}\mid_{D_K} = \mathrm{Fil}_{\rm Hodge}$.
Consider the associated graded Higgs bundle $$(E,\theta)_{\mathbb{C}}:=\mathrm{Gr}_{\mathrm{Fil}_{\mathbb{C}}}(\mathcal{H},\nabla)_{\mathbb{C}}$$ on $X_{\mathbb{C}}$.\footnote{It follows from \cite[p. 331]{Sim91} that under the complex Simpson correspondence, $(\mathcal{H},\nabla)_{\mathbb{C}}$ is simply sent to the associated graded Higgs bundle under the filtration described above.}
Since the representation is irreducible, by the classical Riemann-Hilbert correspondence, $(\mathcal{H},\nabla)_{\mathbb{C}}$ is stable; by the nonabelian Hodge theorem \cite{Sim92}, the associated Higgs bundle $(E,\theta)_{\mathbb{C}}$ is stable.

Denote by $\mathrm{Fil}'$ the Simpson filtration on $(\mathcal{H},\nabla)$ in \autoref{thm_SimpsonFil}, whose construction is compatible with change of base field. The Simpson filtration is, by construction, defined over the field of definition of the integrable connection; moreover, the construction is compatible with change of base field. The associated graded with respect to $\mathrm{Fil}'$ yields a semistable Higgs bundle $(E',\theta')_{\mathbb{C}}$. On the other hand, $(E,\theta)_{\mathbb{C}}$ is stable. By \cite[Lemma 7.1]{LSZ13a}\footnote{Lemma 7.1 in \cite{LSZ13a} is somewhat vague. This lemma can be restated as follows: Let there be two Griffiths transverse filtrations on $(V,\nabla)$, both of which are gr-semistable. If one of them is gr-stable, then they differ by a shift in index. This can be more precisely understood from its proof.}, $\mathrm{Fil}_{\mathbb{C}}$ and $\mathrm{Fil}'_{\mathbb{C}}$ differ by at most a shift of index. Since $\mathrm{Fil}'$ is defined over $K$, $\mathrm{Fil}_{\mathbb{C}}$ descends to $K$; we denote this descent by $\mathrm{Fil}$. This descent is precisely an index shift of $\mathrm{Fil}'$, which is then unique and independent of the choice of $\iota$.

We now have two filtrations $\mathrm{Fil}\mid_D$ and $\mathrm{Fil}_{\rm Hodge}$ on $(\mathcal{H}^i_\mathrm{dR}(Y_D/D),\nabla_{\mathrm{GM}})$. Given that
\[\mathrm{Fil}\mid_D \times_{K,\iota} \mathbb{C} = \mathrm{Fil}_{\mathbb{C}} \mid_{D_K} = \mathrm{Fil}_{\rm Hodge}\times_{K,\iota}\mathbb{C},\]
it follows that $\mathrm{Fil}\mid_D = \mathrm{Fil}_{\rm Hodge}$.

We come to (2) and (3). Let $\mathfrak p$ be an unramified prime ideal with sufficiently large residue characteristic. Denote by $\mathcal{O}_{\mathfrak p}$ the ring of integers in the $\mathfrak p$-adic field $K_{\mathfrak p}$. Then the $\mathfrak p$-adic completion $(\mathcal{H}_{\mathcal{O}_{\mathfrak p}},\nabla_{\mathcal{O}_{\mathfrak p}})\mid_{D_{\mathcal{O}_{\mathfrak p}}}$ of $(\mathcal{H}_{\rm dR}(Y_D/D),\nabla_{\rm GM})$ underlies a Fontaine-Faltings module. (This is because it comes from the cohomology of a smooth, proper family and $\mathfrak p\gg 0$.) In other words, there is a crystalline representation $\rho_{D,\mathfrak p}\colon \pi_1^\et(D_{K_{\mathfrak p}})\rightarrow \mathrm{GL}_N(\mathbb{Z}_p)$ with corresponding $1$-periodic Higgs-de Rham flow of form

\begin{equation}\label{equ: 11.1}
 \xymatrix{
 & (\mathcal{H}_{\mathcal{O}_{\mathfrak p}},\nabla_{\mathcal{O}_{\mathfrak p}},\mathrm{Fil})\mid_{D_{\mathcal{O}_{\mathfrak p}}} \ar[dr] & \\
 (E_{\mathcal{O}_{\mathfrak p}},\theta_{\mathcal{O}_{\mathfrak p}})_{D_{\mathcal{O}_{\mathfrak p}}} \ar[ur]
 && (E_{\mathcal{O}_{\mathfrak p}},\theta_{\mathcal{O}_{\mathfrak p}})_{D_{\mathcal{O}_{\mathfrak p}}} \ar@{=}[ll]\\
 }
\end{equation}
where $(E_{\mathcal{O}_{\mathfrak p}},\theta_{\mathcal{O}_{\mathfrak p}})$ is the $\mathfrak p$-adic completion of $(E,\theta)$. Since $(E,\theta)\mid_{D_K}$ is stable, and stability is a Zariski open condition in the moduli space, for all $\mathfrak p\gg 0$, the restriction $(E_{\mathcal{O}/{\mathfrak p}},\theta_{\mathcal{O}/{\mathfrak p}})\mid_{D_{\mathcal{O}/{\mathfrak p}}}$ is also stable. This implies that the representation $\rho_{D,\mathfrak p}$ is geometrically absolutely residually irreducible. The representation $\rho_{D,\mathfrak p}\colon\pi_1^\et(D_{K_{\mathfrak p}})\rightarrow \mathrm{GL}_N(\mathbb{Z}_p)$ extends to a representation $\rho_{X,\mathfrak p}\colon \pi_1^\et(X_{K_{\mathfrak p}})\rightarrow \mathrm{GL}_N(\mathbb{Z}_p)$ by the classical Lefschetz theorem; using our main \autoref{theorem:main}, one deduces that it is crystalline. Following the construction of this crystalline representation $\rho_{X,\mathfrak p}$ in the proof of our main theorem, there exists a $1$-periodic Higgs-de Rham flow over the $\mathfrak p$-adic completion $X_{\mathcal{O}_{\mathfrak p}}$
\begin{equation}\label{equ: 11.2}
 \xymatrix{
 & (\mathcal{H}'_{\mathcal{O}_{\mathfrak p}},\nabla'_{\mathcal{O}_{\mathfrak p}},\mathrm{Fil}'_{\mathcal{O}_{\mathfrak p}}) \ar[dr] & \\
 (E'_{\mathcal{O}_{\mathfrak p}},\theta'_{\mathcal{O}_{\mathfrak p}}) \ar[ur]
 && \mathrm{Gr}(\mathcal{H}'_{\mathcal{O}_{\mathfrak p}},\nabla'_{\mathcal{O}_{\mathfrak p}},\mathrm{Fil}'_{\mathcal{O}_{\mathfrak p}}) \ar@/^10pt/[ll]^{\simeq}\\
 }
\end{equation}
such that the its restriction on $D_{\mathcal{O}_{\mathfrak p}}$ is equal to that in \eqref{equ: 11.1}.
In particular, one has $(E'_{\mathcal{O}_{\mathfrak p}},\theta'_{\mathcal{O}_{\mathfrak p}})\mid_{D_{K_{\mathfrak p}}}\cong (E_{\mathcal{O}_{\mathfrak p}},\theta_{\mathcal{O}_{\mathfrak p}})\mid_{D_{K_{\mathfrak p}}}$.

We have constructed two Higgs bundles over the $p$-adic formal scheme $X_{\mathcal{O}_{\mathfrak p}}$: $(E_{\mathcal{O}_{\mathfrak p}},\theta_{\mathcal{O}_{\mathfrak p}})$ and $(E'_{\mathcal{O}_{\mathfrak p}},\theta'_{\mathcal{O}_{\mathfrak p}})$.
(The former was constructed to live over a ring of integers in $K$, the latter was constructed via the main theorem of this article and \`a priori only lives on the $p$-adic formal scheme.)
To prove (3), we only need to show that these two Higgs bundles are isomorphic.
Since the graded Higgs bundle $(E,\theta)_{\mathbb{C}}$ originates from a PVHS over a smooth projective variety, its Chern classes vanish. Similarly, $(E'_{\mathcal{O}_{\mathfrak p}},\theta'_{\mathcal{O}_{\mathfrak p}})$ being periodic implies trivial Chern classes by \autoref{footnote:preperiodic}.
Both $(E,\theta)_{\mathcal{O}_{\mathfrak p}}\mid_{X_{k_{\mathfrak p}}}$
and $(E'_{\mathcal{O}_{\mathfrak p}},\theta'_{\mathcal{O}_{\mathfrak p}})\mid_{X_{k_{\mathfrak p}}}$
possess trivial Chern classes and extend the same Higgs bundle $(E_{\mathcal{O}_{\mathfrak p}},\theta_{\mathcal{O}_{\mathfrak p}})\mid_{D_{k_{\mathfrak p}}}$ over the special fiber
$D_{k_{\mathfrak p}}$ of $D$.
By \autoref{cor:lefschetz_semistable_higgs}, we conclude that
\[(E,\theta)_{\mathcal{O}_{\mathfrak p}}\mid_{X_{k_{\mathfrak p}}} \cong (E'_{\mathcal{O}_{\mathfrak p}},\theta'_{\mathcal{O}_{\mathfrak p}})\mid_{X_{k_{\mathfrak p}}}.\]
We prove $(E,\theta)_{\mathcal{O}_{\mathfrak p}}\mid_{X_{\mathcal{O}_{\mathfrak p}/\mathfrak p^n}} \cong (E'_{\mathcal{O}_{\mathfrak p}},\theta'_{\mathcal{O}_{\mathfrak p}})\mid_{X_{\mathcal{O}_{\mathfrak p}/\mathfrak p^{n}}}$ inductively on $n$. Suppose
\[(E,\theta)_{\mathcal{O}_{\mathfrak p}}\mid_{X_{\mathcal{O}_{\mathfrak p}/\mathfrak p^n}} \cong (E'_{\mathcal{O}_{\mathfrak p}},\theta'_{\mathcal{O}_{\mathfrak p}})\mid_{X_{\mathcal{O}_{\mathfrak p}/\mathfrak p^{n}}}\]
holds true for some $n\geq 1$. We identify these two Higgs bundles via the isomorphism.
By \autoref{obs:Higgs}, the difference between $(E,\theta)_{\mathcal{O}_{\mathfrak p}}\mid_{X_{\mathcal{O}_{\mathfrak p}/\mathfrak p^{n+1}}}$ and $(E'_{\mathcal{O}_{\mathfrak p}},\theta'_{\mathcal{O}_{\mathfrak p}})\mid_{X_{\mathcal{O}_{\mathfrak p}/\mathfrak p^{n+1}}}$ corresponds to a cohomology class we may denote by $\epsilon_n \in H^1_{\rm Hig}(X_{K_{\mathfrak p}},\mathcal{E}\mathrm{nd}((E,\theta)_{\mathcal{O}_{\mathfrak p}}\mid_{X_{k_{\mathfrak p}}})).$ Since
\[(E,\theta)_{\mathcal{O}_{\mathfrak p}}\mid_{D_{\mathcal{O}_{\mathfrak p}/\mathfrak p^{n+1}}} \cong (E'_{\mathcal{O}_{\mathfrak p}},\theta'_{\mathcal{O}_{\mathfrak p}})\mid_{D_{\mathcal{O}_{\mathfrak p}/\mathfrak p^{n+1}}},\]
one has $res(\epsilon_n)=0$. By \autoref{prop:res_obs&tor_Higgs}, $\epsilon_n=0$. Thus
\[(E,\theta)_{\mathcal{O}_{\mathfrak p}}\mid_{X_{\mathcal{O}_{\mathfrak p}/\mathfrak p^{n+1}}} \cong (E'_{\mathcal{O}_{\mathfrak p}},\theta'_{\mathcal{O}_{\mathfrak p}})\mid_{X_{\mathcal{O}_{\mathfrak p}/\mathfrak p^{n+1}}}.\]
Then taking projective limits, one obtains $(E,\theta)_{\mathcal{O}_{\mathfrak p}}\cong (E'_{\mathcal{O}_{\mathfrak p}},\theta'_{\mathcal{O}_{\mathfrak p}})$. This verifies (3).

Finally, to prove (2), we only need to show the filtered de Rham bundle $(\mathcal{H}'_{\mathcal{O}_{\mathfrak p}},\nabla'_{\mathcal{O}_{\mathfrak p}},\mathrm{Fil}'_{\mathcal{O}_{\mathfrak p}})$ in the flow is isomorphic to the $\mathfrak p$-adic completion $(\mathcal{H}_{\mathcal{O}_{\mathfrak p}},\nabla_{\mathcal{O}_{\mathfrak p}},\mathrm{Fil}_{\mathcal{O}_{\mathfrak p}})$ of $(\mathcal{H},\nabla,\mathrm{Fil})$.
By (3), we may identify the two Higgs bundles $(E'_{\mathcal{O}_{\mathfrak p}},\theta'_{\mathcal{O}_{\mathfrak p}})\mid_{D_{\mathcal{O}_{\mathfrak p}}}$ and $(E_{\mathcal{O}_{\mathfrak p}},\theta_{\mathcal{O}_{\mathfrak p}})\mid_{D_{\mathcal{O}_{\mathfrak p}}}$.
This Higgs bundle is 1-periodic and stable modulo $p$.
Thus up to an isomorphism, it initiates a \emph{unique} 1-periodic Higgs-de Rham flow by \autoref{lem_uniqueHDF}.
So the restriction of the flow in \eqref{equ: 11.2} onto $D_{K_{\mathfrak p}}$ is isomorphic to the flow in \eqref{equ: 11.1}.
In particular, one has an isomorphism $\tau_{D_{\mathcal{O}_{\mathfrak p}}}\colon (\mathcal{H}'_{\mathcal{O}_{\mathfrak p}},\nabla'_{\mathcal{O}_{\mathfrak p}},\mathrm{Fil}'_{\mathcal{O}_{\mathfrak p}})\mid_{D_{\mathcal{O}_{\mathfrak p}}} \cong (\mathcal{H}_{\mathcal{O}_{\mathfrak p}},\nabla_{\mathcal{O}_{\mathfrak p}},\mathrm{Fil}_{\mathcal{O}_{\mathfrak p}})\mid_{D_{\mathcal{O}_{\mathfrak p}}}$ as filtered de Rham bundles over $D_{\mathcal{O}_{\mathfrak p}}$.

The reductions modulo $\mathfrak p$ of both $(\mathcal{H}'_{\mathcal{O}_{\mathfrak p}},\nabla'_{\mathcal{O}_{\mathfrak p}})$ and $(\mathcal{H}_{\mathcal{O}_{\mathfrak p}},\nabla_{\mathcal{O}_{\mathfrak p}})$ are stable integrable connections.
By \autoref{Lemma:MIC_extension}, we may extend $\tau_{D_{\mathcal{O}_{\mathfrak p}}}$ to an isomorphism
\begin{equation}
\tau_{X_{K_{\mathfrak p}}}\colon
 (\mathcal{H}'_{\mathcal{O}_{\mathfrak p}},\nabla'_{\mathcal{O}_{\mathfrak p}})\mid_{X_{K_{\mathfrak p}}}
\cong
 (\mathcal{H}_{\mathcal{O}_{\mathfrak p}},\nabla_{\mathcal{O}_{\mathfrak p}})\mid_{X_{K_{\mathfrak p}}},
\end{equation}
Again using \cite[Lemma 7.1]{LSZ13a}, over $X_{K_{\mathfrak p}}$, the filtrations $\mathrm{Fil}$ and $\mathrm{Fil}'$ only differ by a shift under $\tau_{X_{K_{\mathfrak p}}}$ because their associated graded Higgs bundles are stable.
As $\tau_{D_{\mathcal{O}_{\mathfrak p}}}$ preserves the filtrations, it follows that $\mathrm{Fil}$ and $\mathrm{Fil}'$ coincide (i.e., the difference-shift in the filtrations is zero), i.e. $\tau_{X_{K_{\mathfrak p}}}$ preserves the filtrations.

Now, the modulo $p$ reductions of both $(\mathcal{H}'_{\mathcal{O}_{\mathfrak p}},\nabla'_{\mathcal{O}_{\mathfrak p}})$ and $(\mathcal{H}_{\mathcal{O}_{\mathfrak p}},\nabla_{\mathcal{O}_{\mathfrak p}})$ are stable and the two integrable connections are isomorphic over $X_{K_{\mathfrak p}}$.
It then follows from Langer's Langton-style theorem \cite[Theorem 5.2]{Lan14} that $\tau_{X_{K_{\mathfrak p}}}$ extends to an isomorphism
\[\tau_{X_{\mathcal{O}_{\mathfrak p}}}\colon (\mathcal{H}'_{\mathcal{O}_{\mathfrak p}},\nabla'_{\mathcal{O}_{\mathfrak p}})\cong p^r(\mathcal{H}_{\mathcal{O}_{\mathfrak p}},\nabla_{\mathcal{O}_{\mathfrak p}})\]
for some $r\in \mathbb{Z}$.
As $\tau_D$ is an isomorphism over $D$, it follows that $r=0$, i.e., that $\tau_{X_{\mathcal{O}_{\mathfrak p}}}\colon (\mathcal{H}',\nabla')_{\mathcal{O}_{\mathfrak p}}\cong (\mathcal{H}_{\mathcal{O}_{\mathfrak p}},\nabla_{\mathcal{O}_{\mathfrak p}})$.
We may once again apply \cite[Proposition 7.5]{LSZ13a} to deduce that $\tau_{X_{\mathcal{O}_{\mathfrak p}}}(\mathrm{Fil}'_{\mathcal{O}_{\mathfrak p}})=\mathrm{Fil}_{\mathcal{O}_{\mathfrak p}}$.
Thus $\tau_{X_{\mathcal{O}_{\mathfrak p}}}$ is the desired isomorphism.
\end{proof}

\appendix
\section{Crystalline representations}\label{section FFM}
In this appendix, we review the theory of crystalline local systems. We follow closely (both in notation and in exposition) the foundational work of Faltings \cite{Fal89}.
For another review of this material, see \cite[2.6]{lovering}.
For a longer treatise with all of the details, see the recent \cite[Sections 2-4]{Tsu20}. (This source proves in detail many of the necessary statements in commutative algebra to set up the period rings. It also reproves several of Faltings' theorems using the recent $\mathbb{A}_{\mathrm{inf}}$ technology.)

\begin{definition}\label{def_Hodgefiltration}
Let $S$ be a locally Noetherian scheme, $Y\to S$ a smooth morphism, and $Z\subset Y$ an $S$-relative normal crossing divisor in $Y$. Denote by $\Omega^1_{Y/S}(\log Z)$ the sheaf of relative logarithmic differentials of $(Y,Z)$ over $S$.
\begin{enumerate}
 \item Recall that an \emph{integrable log-connection} on a coherent sheaf $V$ over $(Y,Z)/S$ is an $\mathcal{O}_S$-linear map
 \[\nabla\colon V\to V\otimes \Omega_{Y/S}(\log Z)\]
 satisfying the Leibniz rule and integrability condition $\nabla\circ \nabla=0$. In this case, the pair $(V,\nabla)$ is called an \emph{integrable log-connection}, or a \emph{logarithmic de Rham module/sheaf} over $(Y,Z)/S$. If $V$ is a vector bundle, it is also called a \emph{logarithmic de Rham bundle}. When $Z=\emptyset$, these are simply termed an \emph{integrable connection} or \emph{de Rham module/sheaf/bundle} over $Y/S$.
\item Given integers $a\leq b$, a \emph{Hodge filtration of level $[a,b]$} on an integrable (log-)connection $(V,\nabla)$ is a decreasing filtration $\mathrm{Fil}$ such that each $\mathrm{Fil}^i V$ is a locally split subsheaf of $V$; the filtration satisfies the Griffiths transversality condition:
 \[\nabla(\mathrm{Fil}^iV)\subset \mathrm{Fil}^{i-1}V\otimes \Omega^1_{Y/S}(\log Z),\]
 and the filtration is exhaustive and separated,
 with
 \[V=\mathrm{Fil}^aV\supset \mathrm{Fil}^{a+1}V \supset\cdots \supset \mathrm{Fil}^bV\supset \mathrm{Fil}^{b+1}V=0.\]
 \item A \emph{filtered (logarithmic) de Rham module/sheaf/bundle} is a triple $(V,\nabla,\mathrm{Fil})$ consisting of an integrable (log-)connection $(V,\nabla)$ and a Hodge filtration $\mathrm{Fil}$ on it.
\end{enumerate}
 These definitions extend straightforwardly to smooth formal schemes and rigid analytic spaces; we leave the details to the reader.
\end{definition}

\begin{setup}\label{setup:scheme_rigid1} Let $Y$ be a smooth scheme over $W$ (not necessary projective) with geometrically connected generic fiber. Denote by $\mathcal{Y}$ the $p$-adic formal completion of $Y$ along the special fiber $Y_1$ and by $\mathcal{Y}_K$ the rigid-analytic space associated to $\mathcal{Y}$, which is an open subset of $Y_K^{\rm an}$.
\end{setup}

\subsection{Fontaine-Faltings modules over a small affine base.}\label{sect:FF/small}
We first recall the notion of a Fontaine-Faltings module over a small affine scheme.
Assume $Y$ is a connected, small affine scheme, i.e., $Y$ is connected and if $Y=\mathrm{Spec}(R)$, then there exists an \'etale map \[W[T_1^{\pm1},T_2^{\pm1},\cdots, T_{d}^{\pm1}]\rightarrow R,\] over $W$
(see \cite[p. 27]{Fal89}). In general, a smooth affine scheme over $W$ is not always small but it can be covered by a system of small affine open subsets. By the existence of the \'etale chart there exists some $\Phi:\widehat{R}\rightarrow\widehat{R}$ which lifts the absolute Frobenius on $R/pR$, where $\widehat{R}$ is the $p$-adic completion of $R$.

A \emph{Fontaine-Faltings module} over the $p$-adic formal completion $\mathcal{Y}=\mathrm{Spf}(\widehat{R})$ of $Y$ with Hodge-Tate weights in $[a,b]$ is a quadruple $(V,\nabla,\mathrm{Fil},\varphi)$, where
\begin{itemize}
 \item[-] $(V,\nabla)$ is a de Rham $\widehat{R}$-module;
 \item[-] $\mathrm{Fil}$ is a Hodge filtration on $(V,\nabla)$ of level in $[a,b]$;
 \item[-] $\widetilde{V}$ is the quotient $\bigoplus\limits_{i=a}^b\mathrm{Fil}^i/\sim$ with $px\sim y$ for $x\in\mathrm{Fil}^iV$ with $y$ being the image of $x$ under the natural inclusion $\mathrm{Fil}^iV\hookrightarrow\mathrm{Fil}^{i-1}V$;
 \item[-] $\varphi$ is an $\widehat{R}$-linear isomorphism \[\varphi:\widetilde{V}\otimes_{\Phi}\widehat{R} \longrightarrow V,\]
 \item[-] The relative Frobenius $\varphi$ is horizontal with respect to the connections.
\end{itemize}
(The fact that $\varphi$ is an isomorphism is sometimes known as \emph{strong $p$-divisibility.}) A morphism between Fontaine-Faltings modules is a morphism between the underlying de Rham bundles which is strict for the filtrations and commutes with the $\varphi$-structures. Denote by $\mathcal{MF}_{[a,b]}^{\nabla,\Phi}(\mathcal{Y}/W)$ the category of Fontaine-Faltings modules over $\mathcal{Y}$ with Hodge-Tate weights in $[a,b]$. The $p$-primary torsion version of this definition was first written down in \cite[p. 30-31]{Fal89}; here we follow \cite[Section 3]{Fal99}, see also \cite[Section 2]{SYZ21} and \cite[Section 2]{LSZ13a}.

\paragraph{\emph{The gluing functor.}} In the following, we recall the gluing functor of Faltings. In other words, up to a canonical equivalence of categories, if $b-a\leq p-2$, the category $\mathcal{MF}_{[a,b]}^{\nabla,\Phi}(\mathcal{Y}/W)$ does not depend on the choice of $\Phi$. More explicitly, the functor yielding an equivalence is given as follows. Let $\Psi$ be another lifting of the absolute Frobenius. For any filtered de Rham module $(V,\nabla,\mathrm{Fil})$, Faltings~\cite[Theorem~2.3]{Fal89} shows that there is a canonical isomorphism by the Taylor formula
\[\alpha_{\Phi,\Psi}: \widetilde{V}\otimes_\Phi\widehat{R} \simeq \widetilde{V}\otimes_\Psi \widehat{R},\]
which is parallel with respect to the connection, satisfies the cocycle conditions and induces an equivalence of categories
\begin{equation}
 \xymatrix@R=0mm{ \mathcal{MF}_{[a,b]}^{\nabla,\Psi}(\mathcal{Y}/W)\ar[r] & \mathcal{MF}_{[a,b]}^{\nabla,\Phi}(\mathcal{Y}/W).\\
 (V,\nabla,\mathrm{Fil},\varphi)\ar@{|->}[r] & (V,\nabla,\mathrm{Fil},\varphi\circ\alpha_{\Phi,\Psi})\\}
\end{equation}

\subsection{Fontaine-Faltings modules over a global base.} \label{sect:FF/gl}
In this section, we do not assume $Y$ is small, but we maintain the assumption that $Y$ has geometrically connected generic fiber. Let $I$ be the index set of all pairs $(\mathcal{U}_i,\Phi_i)$, where $\mathcal{U}_i$ is a connected small affine open subset of $\mathcal{Y}$, and $\Phi_i$ is a lift of the absolute Frobenius on $\mathcal{O}_{\mathcal{Y}}(\mathcal{U}_i)\otimes_W k$. Clearly, $\{\mathcal{U}_i\}_{i\in I}$ is a small affine covering of $\mathcal{Y}$. Recall that the category $\mathcal{MF}_{[a,b]}^{\nabla}(\mathcal{Y}/W)$ is constructed by gluing the categories $\mathcal{MF}_{[a,b]}^{\nabla,\Phi_i}(\mathcal{U}_i/W)$. Actually $\mathcal{MF}_{[a,b]}^{\nabla}(\mathcal{Y}/W)$ can be described more precisely as follows. A Fontaine-Faltings module over $\mathcal{Y}$ of Hodge-Tate weights in $[a,b]$ is a tuple $(V,\nabla,\mathrm{Fil},\{\varphi_i\}_{i\in I})$, i.e. a filtered de Rham sheaf $(V,\nabla,\mathrm{Fil})$ over $\mathcal{Y}$ together with $\varphi_i: \widetilde{V}(\mathcal{U}_i)\otimes_{\Phi_i} \widehat{\mathcal{O}_{\mathcal{Y}}(\mathcal{U}_i)}\rightarrow V(\mathcal{U}_i)$ such that
\begin{itemize}
 \item[-] $M_i:=(V(\mathcal{U}_i),\nabla,\mathrm{Fil},\varphi_i)\in \mathcal{MF}_{[a,b]}^{\nabla,\Phi_i}(\mathcal{U}_i/W)$.
 \item[-] For all $i,j\in I$, on the open intersection $\mathcal{U}_i\cap \mathcal{U}_j$, the Fontaine-Faltings modules $M_i\mid_{\mathcal{U}_{i}\cap \mathcal{U}_j}$ and $M_j\mid_{\mathcal{U}_{i}\cap \mathcal{U}_j}$ are associated to each other under the above equivalence of categories with respect to the two Frobenius liftings $\Phi_i$ and $\Phi_j$ on $\mathcal{U}_i\cap \mathcal{U}_j$.
\end{itemize}
Denote by $\mathcal{MF}_{[a,b]}^{\nabla}(\mathcal{Y}/W)$ the category of all Fontaine-Faltings modules over $\mathcal{Y}$ of Hodge-Tate weights in $[a,b]$.

Faltings constructed a fully faithful contravariant functor from the category of Fontaine-Faltings modules to the category of local systems.
\begin{theorem}[Faltings {\cite[Theorem 2.6*]{Fal89}}] Setup as in \autoref{setup:scheme_rigid1}. Suppose $b-a\leq p-2$. Then there is a fully faithful functor:
\begin{equation*}
 \mathbb{D}\colon \mathcal{MF}_{[a,b]}^{\nabla}(\mathcal{Y}/W)\rightarrow \mathrm{Loc}_{\mathbb{Z}_p}(\mathcal{Y}_K).
\end{equation*}
The essential image of the functor $\mathbb{D}$ is closed under sub-objects and quotients. Local systems in the essential image are called crystalline representations.
\end{theorem}

\begin{remark}
 If moreover $Y$ is proper, then by rigid GAGA, the local systems over $\mathcal{Y}_K$ are algebraic, i.e., one obtains $\mathbb{Z}_p$-local systems on $Y_K$. This has the following upshot: if $b-a\leq p-2$, there is a functor, by abusing notation we still denote it by $\mathbb{D}$,
 \begin{equation}
 \mathbb{D}\colon \mathcal{MF}_{[a,b]}^{\nabla}(\mathcal{Y}/W)\rightarrow \mathrm{Loc}_{\mathbb{Z}_p}(Y_K),
 \end{equation}
 from the category of Fontaine-Faltings modules (with $b-a\leq p-2$) to the category of finite dimensional lisse $\mathbb{Z}_p$ sheaves on $Y_K$. Since $Y_K$ is connected, picking a base point, this is equivalent to the category $\mathrm{Rep}_{\mathbb{Z}_p}(\pi_1^\et(Y_K))$ of continuous finite free $\mathbb{Z}_p$-representations of $\pi_1^\et(Y_K)$. It is a fundamental result of Faltings that this functor is fully faithful; a lisse $\mathbb{Z}_p$ sheaf in the essential image of $\mathbb{D}$ is called \emph{crystalline (with Hodge-Tate weights in $[a,b]$).}
\end{remark}

\section{Logarithmic crystalline representations}\label{section:log_FFM}

\subsection{Logarithmic Fontaine-Faltings modules}
Faltings claimed in \cite[i) p.43]{Fal89} that the theory of crystalline representations extends to the logarithmic context. However, it seems as though the details of this construction have never appeared in the literature. In this appendix, we carefully formulate the category of logarithmic Fontaine-Faltings modules in the local setting. (One may also find the definition of a logarithmic Fontaine-Faltings module in \cite{LSZ13a}.) We then formulate two conjectures: one which essentially says that Faltings' construction indeed extends to the logarithmic setting, and a second asking for compatibility with the logarithmic Riemann-Hilbert functor constructed in \cite{DLLZ}.

\begin{setup}\label{setup:log_FF}Let $Y/W$ be a smooth scheme (not necessarily proper), and let $Z\subset Y$ be a $W$-flat simple relative normal crossings divisor. Set $U:=Y\setminus Z$. Construct spaces $Z_K$, $\mathcal{Z}$, $\mathcal{Z}_K$, $U_K$, $\mathcal{U}$ and $\mathcal{U}_K$ exactly analogously to those for $Y$ in \autoref{setup:scheme_rigid1}. Denote $\mathcal{Y}^\circ_K:=\mathcal{Y}_K-\mathcal{Z}_K$.
\end{setup}
We will first construct the category of logarithmic Fontaine-Faltings modules under the following assumptions.
\begin{setup}\label{setup:log_FF_affine}
Let $Y=\mathrm{Spec}R$ be an affine $W$-scheme with an \'etale map $$W[T_1,T_2,\cdots, T_{d}]\rightarrow R,$$ over $W$, let $Z$ be the divisor in $Y$ defined by $T_1\cdots T_r=0$, for some $r\leq d$. Therefore, $U$ is a small affine scheme. In this context, we say that $(Y,Z)$ is \emph{logarithmically small}.
Denote by $\widehat{R}$ the $p$-adic completion of $R$, so $\mathcal{Y}=\mathrm{Spf}(\widehat{R})$. Denote by $\Phi:\widehat{R}\rightarrow\widehat{R}$ a lifting of the absolute Frobenius on $R/pR$ such that $\Phi(T_i) = w_iT_i^p$ for some $w_i\in \widehat{R}^\times$.\footnote{This means that the lifting $\Phi$ is compatible with the logarithmic structure. For example, one can take $w_i=1$ for all $i=1,2,\cdots,d$. We note that since $\Phi$ lifts the absolute Frobenius, $w_i\equiv 1\pmod{p}$. One can therefore write $w_i=1+pu_i$ for some $u_i\in \widehat{R}$.}
\end{setup}

Before giving the definition of logarithmic Fontaine-Faltings module, we will first introduce two objects associated to a a finitely generated filtered logarithmic de Rham sheaf $(V,\nabla,\mathrm{Fil})$ over $\mathcal{Y}$ with logarithmic poles along $\mathcal{Z}$:
\begin{itemize}
\item a logarithmic $p$-connection, $(\widetilde{V},\widetilde{\nabla})$; and
\item a logarithmic connection $(\widetilde{V}\otimes_\Phi \widehat{R},\Phi_*(\widetilde{\nabla}))$.
\end{itemize}
The module $\widetilde{V}$ is defined in the same way as in the case without logarithmic poles, i.e., it is the quotient $\bigoplus\limits_{i=a}^b\mathrm{Fil}^iV/\sim$ with $px\sim y$ for $x\in\mathrm{Fil}^iV$ with $y$ being the image of $x$ under the natural inclusion $\mathrm{Fil}^iV\hookrightarrow\mathrm{Fil}^{i-1}V$; In other words, if we denote by $[v]_i$ the image of $v\in \mathrm{Fil}^iV$ in $\widetilde{V}$ under the natural morphism $\mathrm{Fil}^iV\rightarrow \widetilde{V}$, then $[v]_{i-1} = p\cdot [v]_i$. We now construct the $p$-connection.

Consider the composition
\[\widetilde{\nabla}_i \colon V_i\xrightarrow{\nabla}V_{i-1}\otimes\Omega^1_{\widehat{R}}(\log Z) \xrightarrow{[\cdot]_{i-1}\otimes \mathrm{id}} \widetilde{V}\otimes\Omega^1_{\widehat{R}}(\log Z).\]
Since $[\cdot]_{i-1}=p[\cdot]_i$, the following diagram commutes
\[\xymatrix{& V_i \ar[rd]^{p} \ar@{^(->}[ld] & \\
V_{i-1} \ar[rd]^{\widetilde{\nabla}_{i-1}} && V_{i} \ar[ld]_{\widetilde{\nabla}_{i}}\\
& \widetilde{V}\otimes \Omega^1_{\widehat{R}}(\log Z)}\]
Applying the universal property of direct limits for $\widetilde{V}$, one gets a map
\[\widetilde{\nabla}:\widetilde{V}\rightarrow \widetilde{V}\otimes\Omega^1_{\widehat{R}}(\log Z).\]
This map can be written in a more explicit way: since $\Omega^1_{\widehat{R}}(\log Z)$ is free over $\widehat{R}$ generated by $\{\mathrm{d}\log T_1,\cdots,\mathrm{d}\log T_r,\mathrm{d} T_{r+1},\cdots,\mathrm{d}T_d\}$, for any $v\in \mathrm{Fil}^iV$, the element $\nabla(v)$ can be uniquely written in form
\begin{equation}\label{equ_form}
 \nabla(v)=\sum_{j=1}^rv_j\otimes \mathrm{d}\log T_j + \sum_{j=r+1}^d v_j\otimes \mathrm{d}T_j \in V\otimes \Omega^1_{\widehat{R}}(\log Z).
\end{equation}
Since $\nabla$ satisfies Griffiths transversality, one gets that $v_j\in \mathrm{Fil}^{i-1}V$ for all $j=1,\cdots,d$. Then
\[\widetilde{\nabla}([v]_i) = \sum_{j=1}^r[v_j]_{i-1}\otimes \mathrm{d}\log T_j + \sum_{j=r+1}^d [v_j]_{i-1}\otimes \mathrm{d}T_j \in \widetilde{V}\otimes \Omega^1_{\widehat{R}}(\log Z).\]
This map $\widetilde{\nabla}$ is indeed a $p$-connection, since for any $f\in \widehat{R}$ and $v\in \mathrm{Fil}^iV$,
\[\widetilde{\nabla}(f[v]_i)=([\cdot]_{i-1}\otimes\mathrm{id})\circ\nabla(fv)=([\cdot]_{i-1}\otimes\mathrm{id})\Big(f\nabla(v)+v\otimes\mathrm{d}f\Big)=f\widetilde{\nabla}([v]_i)+p[v]_{i}\otimes\mathrm{d}f.\]
We have thus constructed a logarithmic $p$-connection $(\widetilde{V},\widetilde{\nabla})$.

\begin{remark}
The complex definition above primarily addresses cases involving $p$-primary torsion. When $V$ is $p$-torsion free, the situation simplifies: $\widetilde{V}$ becomes naturally isomorphic to $\sum_{i\in\mathbb{Z}} \frac{\mathrm{Fil}^iV}{p^i} \subset V\otimes_{\mathbb{Z}_p}\mathbb{Q}_p$ via the map $[v]_i \mapsto \frac{v}{p^i}$. Explicitly, this identifies $\widetilde{V}$ with the subspace $V' := \sum_i \frac{1}{p^i}\mathrm{Fil}^iV \subset V\otimes_{\mathbb{Z}_p}\mathbb{Q}_p$. While the connection $\nabla$ extends $\mathbb{Q}_p$-linearly to $V\otimes_{\mathbb{Z}_p}\mathbb{Q}_p$, it cannot restrict to a connection on $V'$; however, the \emph{$p$-connection} $p\nabla$ does restrict to a \emph{$p$-connection} on $V'$. Under this identification, we have $\widetilde{\nabla} = p\nabla$.
\end{remark}

We now construct the connection $\Phi_*(\widetilde{\nabla})$ on $\widetilde{V}\otimes_\Phi\widehat{R}$. For $v\in \mathrm{Fil}^iV$, let $\{v_j\}$ (defined in \eqref{equ_form}) be the coefficients in the expansion of $\nabla(v)$ in the coordinate basis. Define
\[\Phi_*(\widetilde{\nabla})([v]_i\otimes_\Phi1)
:=\sum_{j=1}^r([v_j]_{i-1}\otimes_\Phi1)\otimes \frac{\mathrm{d}\log \Phi(T_j)}{p}+\sum_{j=r+1}^d ([v_j]_{i-1}\otimes_\Phi1)\otimes \frac{\mathrm{d}\Phi(T_j)}{p},\]
which extends uniquely to a connection on $\widetilde{V}\otimes_\Phi\widehat{R}$ via the Leibniz rule. Direct verification shows that for any $f\in\widehat{R}$:
\begin{equation*}
\Phi_*(\widetilde{\nabla})([fv]_i\otimes_\Phi1) =
\Phi(f)\cdot\Phi_*(\widetilde{\nabla})([v]_i\otimes_\Phi1) +
([v]_i\otimes_\Phi1)\otimes \mathrm{d}\Phi(f).
\end{equation*}
The condition $\Phi(T_i)=w_iT_i^p$ ensures $\frac{\mathrm{d}\log \Phi(T_j)}{p},\frac{\mathrm{d}\Phi(T_j)}{p} \in \Omega_Y^1(\log Z)$ and making $\Phi_*(\widetilde{\nabla})$ a logarithmic connection along $T_1\cdots T_r=0$.

\textbf{Coordinate independence:}
\begin{itemize}
 \item For $p$-torsion free $V$, which is our main case of interest, identifying $\widetilde{V}$ with $\sum\limits_{i\in\mathbb{Z}} \frac{\mathrm{Fil}^iV}{p^i}\subset V\otimes \mathbb{Q}_p$ yields $\Phi_*(\widetilde{\nabla}) = \Phi_*(\nabla)$, where $\Phi_*(\nabla)$ is the push forward of a connection along a ring homomorphism, which is certainly coordinate-independent. For the equality, one checks that for any $v\in\mathrm{Fil}^iV$,
\begin{equation*}
\begin{split}
\Phi_*(\widetilde{\nabla}) ([v]_i\otimes_\Phi1) &= \sum_{j=1}^r ([v_j]_{i-1}\otimes_\Phi1) \otimes \frac{\mathrm{d}\log \Phi(T_j)}{p} + \sum_{j=r+1}^d ([v_j]_{i-1}\otimes_\Phi1)\otimes \frac{\mathrm{d}\Phi(T_j)}{p} \\ =& \sum_{j=1}^r(\frac{v_j}{p^{i-1}}\otimes_\Phi1)\otimes \frac{\mathrm{d}\log \Phi(T_j)}{p}+\sum_{j=r+1}^d (\frac{v_j}{p^{i-1}}\otimes_\Phi1)\otimes \frac{\mathrm{d}\Phi(T_j)}{p} \\
=& \frac1{p^i}\left(\sum_{j=1}^r(v_j\otimes_\Phi1)\otimes \mathrm{d}\log \Phi(T_j)+\sum_{j=r+1}^d (v_j\otimes_\Phi1)\otimes \mathrm{d}\Phi(T_j)\right) \\
=& \frac1{p^i}\Phi_*\left(\sum_{j=1}^r v_j\otimes \mathrm{d}\log T_j +\sum_{j=r+1}^d v_j \otimes \mathrm{d} T_j \right) \\
=& \frac1{p^i}\Phi_* (\nabla(v))=\Phi_* (\nabla([v]_i)) = \Phi_*(\nabla)([v]_{i}\otimes_\Phi1).\\
\end{split}
\end{equation*}
 \item For the general case, locally choose a short exact sequence of filtered de Rham modules
\begin{equation} \label{ses1}
0\rightarrow (\widehat{V}',\widehat{\nabla},\widehat{\mathrm{Fil}})\longrightarrow (\widehat{V}, \widehat{\nabla}, \widehat{\mathrm{Fil}}) \longrightarrow (V,\nabla,\mathrm{Fil}) \rightarrow 0
\end{equation}
 with $(\widehat{V}, \widehat{\nabla}, \widehat{\mathrm{Fil}})$ chosen to be $p$-torsion free. Taking the tilde functor and the operator $\Phi_*$ we defined above, one obtains a short exact sequence ofde Rham bundles
 \begin{equation} \label{ses2}
 0\rightarrow
\left(\widetilde{\widehat{V}'}\otimes_\Phi\widehat{R},\Phi_*(\widetilde{\widehat{\nabla}})\right)
\longrightarrow \left(\widetilde{\widehat{V}}\otimes_\Phi\widehat{R},\Phi_*(\widetilde{\widehat{\nabla}})\right) \longrightarrow \left(\widetilde{V}\otimes_\Phi\widehat{R},\Phi_*(\widetilde{\nabla})\right) \rightarrow 0.
\end{equation}
The exactness for the underlying modules follows, and that the morphisms preserve the connections can be checked directly by definition.

Since the first two terms in \eqref{ses1} are $p$-torsion free, the first two terms in\eqref{ses2} are coordinate independent. Hence the third term $\left(\widetilde{V}\otimes_\Phi\widehat{R},\Phi_*(\widetilde{\nabla})\right)$ is also coordinate independent.
\end{itemize}

One defines a logarithmic Fontaine-Faltings module over $(\mathcal{Y},\mathcal{Z})$ in the same way as in the non-logarithmic case, simply using differentials with logarithmic poles instead regular differentials. More explicitly, a \emph{logarithmic Fontaine-Faltings module} over the $p$-adic formal completion $(\mathcal{Y},\mathcal{Z})$ of $(Y,Z)$ with Hodge-Tate weights in $[a,b]$ is a quadruple $(V,\nabla,\mathrm{Fil},\varphi)$, where
\begin{itemize}
 \item[-] $(V,\nabla)$ is a finitely generated de Rham $\widehat{R}$-module\footnote{We note that here we do not require the underlying module being locally free a priori. But by the existence of Frobenius structure, if the underlying module is $p$-torsion-free, then it must be locally free. This follows from the fact that $(V/p^n,\mathrm{Fil},\varphi)$ is located in $\MF{R}$ for each $n\geq1$.} with logarithmic poles along $T_1\dots T_d=0$;
 \item[-] $\mathrm{Fil}$ is a Hodge filtration on $(V,\nabla)$ of level in $[a,b]$ as in the text after Setup \ref{setup:scheme_rigid1};
 \item[-] $\varphi$ is an $\widehat{R}$-linear isomorphism \[\varphi:\widetilde{V}\otimes_{\Phi}\widehat{R} \longrightarrow V,\]
 which is horizontal with respect to the connections $\Phi_*(\widetilde{\nabla})$ and $\nabla$, i.e., $\varphi$ is a morphism between two de Rham $\widehat{R}$-modules with logarithmic poles along $T_1\cdots T_r=0$.

\end{itemize}

 In particular, a logarithmic Fontaine-Faltings module whose underlying de Rham $\widehat{R}$-module $V$ is locally free may be considered as a filtered logarithmic $F$-crystal in finite, locally free modules. Note that our definition of a logarithmic Fontaine-Faltings module also includes the case when $V$ is $p$-primary torsion.
Denote by $\mathcal{MF}_{[a,b]}^{\nabla,\Phi}((\mathcal{Y},\mathcal{Z})/W)$ the category of logarithmic Fontaine-Faltings modules over $(\mathcal{Y},\mathcal{Z})$ with Hodge-Tate weights in $[a,b]$. For the rest of what follows, we assume that $b-a\leq p-2$.

\begin{conjecture}[Logarithmic analog of Faltings' gluing theorem and compatibility with Higgs-de Rham flows] \label{conj:log_gluing}
Notation as in Setup \ref{setup:log_FF_affine}. Assume $0\leq b-a\leq p-2$, and $p>2$. Then
\begin{enumerate}
\item For any two choices of $\Phi, \Psi$ of Frobenius lifts, satisfying the conditions specified in Setup \ref{setup:log_FF_affine}, there is an equivalence between the corresponding categories $$\mathcal{MF}_{[a,b]}^{\nabla,\Psi}((\mathcal{Y},\mathcal{Z})/W)\rightarrow \mathcal{MF}_{[a,b]}^{\nabla,\Phi}((\mathcal{Y},\mathcal{Z})/W).$$ These equivalences satisfy the obvious cocycle condition, given a third Frobenius lift. Therefore for any $(Y, Z)/W$ as in \autoref{setup:log_FF}, we can define the category $\mathcal{MF}_{[a,b]}^{\nabla}((\mathcal{Y},\mathcal{Z})/W)$ by gluing.
\item For any $(Y, Z)/W$ as in \autoref{setup:log_FF}, there is an equivalence of categories between $\mathcal{MF}_{[a,b]}^{\nabla}((\mathcal{Y},\mathcal{Z})/W)$ and the category of 1-periodic logarithmic Higgs-de Rham flows, as in \cite[Appendix]{LSYZ14}, extending the non-logarithmic equivalence of \cite{LSZ13a}.
\end{enumerate}
\end{conjecture}

\begin{remark}As explained in Appendix \ref{section FFM}, in the non-logarithmic setting the key to the analogous result is the Taylor formula comparing two different Frobenius lifts. Similarly, here the key will be a logarithmic Taylor formula. The computations are indeed formidable.
\end{remark}

\begin{remark}The equivalence between $\mathcal{MF}_{[a,b]}^{\nabla}((\mathcal{Y},\mathcal{Z})/W)$ and the category of 1-periodic logarithmic Higgs-de Rham flows should indeed be straightforward if $(Y,Z)$ satisfies \autoref{setup:log_FF_affine}.
Let $(V,\nabla,\mathrm{Fil},\varphi)$ be a logarithmic Fontaine-Faltings module. It can be represented as
\[\xymatrix{
(V,\nabla,\mathrm{Fil}) \ar[dr]^{\widetilde{(\cdot)}} && \Phi_*(\widetilde{V},\widetilde{\nabla}) \ar@/_10pt/[ll]_{\varphi} \\
&(\widetilde{V},\widetilde{\nabla}). \ar[ur]^{\Phi_*}
}\]
Then the associated $1$-periodic Higgs-de Rham flow is of the form
\[\xymatrix{
 & \Phi_*(\widetilde{V},\widetilde{\nabla}) \ar[dr]^{\mathrm{Gr}}\\
\mathrm{Gr}(V,\nabla,\mathrm{Fil}) \ar[ur]^{\mathcal{C}^{-1}} && \mathrm{Gr}(\Phi_*(\widetilde{V},\widetilde{\nabla})). \ar@/^10pt/[ll]_{\mathrm{Gr}(\varphi)}}
\]
\end{remark}

\subsection{The \texorpdfstring{$\mathbb{D}^{\log}$}{Dlog} functor}
In the following conjecture, we remind the reader that when we say ``$\mathbb{Z}_p$-local system'', we are referring to a continuous representation of the \'etale fundamental group of the rigid space $\mathcal{Y}^{\circ}_K$, i.e., the profinite group associated to the category of \emph{(connected) finite \'etale covers} of $\mathcal{Y}^{\circ}_K$.
\begin{conjecture}[Existence of Faltings $\mathbb{D}^{\log}$-functor]\label{conj:Faltings_Dlog}
Let $Y, Z, W$ be as in \autoref{setup:log_FF}, and assume \autoref{conj:log_gluing}. Fix $a\leq b \leq a+p-1$. Let $\mathcal{MF}_{[a,b]}^{\nabla}((\mathcal{Y},\mathcal{Z})/W)$ be the resulting category of logarithmic Fontaine-Faltings modules on $(\mathcal{Y}, \mathcal{Z})$. Then
\begin{enumerate}
\item there is a \emph{fully faithful} functor:

$$\mathbb{D}^{\log}\colon \mathcal{MF}_{[a,b]}^{\nabla}((\mathcal{Y},\mathcal{Z})/W)\rightarrow \mathrm{Loc}_{\mathbb{Z}_p}(\mathcal{Y}^{\circ}_K),$$
the essential image of which we call \emph{logarithmic crystalline representations}.
\item For another $Y',Z'/W$ as in \autoref{setup:log_FF}, together with a map of logarithmic pairs $(Y',Z')\rightarrow (Y,Z)$,
the functor $\mathbb{D}^{\log}$ is functorial, i.e., the diagram
 \begin{equation*}
 \xymatrix@C=2cm{
 \mathcal{MF}_{[a,b]}^{\nabla}((\mathcal{Y}',\mathcal{Z}')/W) \ar[d] \ar[r]^-{\mathbb{D}^{\log}} & \mathrm{Loc}_{\mathbb{Z}_p}(\mathcal{Y}'^{\circ}) \ar[d]\\
 \mathcal{MF}_{[a,b]}^{\nabla}((\mathcal{Y},\mathcal{Z})/W) \ar[r]^-{\mathbb{D}^{\log}} & \mathrm{Loc}_{\mathbb{Z}_p}(\mathcal{Y}^{\circ}) \\
 }
 \end{equation*}
 2-commutes.
\end{enumerate}
\end{conjecture}
\begin{remark}\label{rem:algebraic}Note that if $Y/W$ is projective, then $\pi_1^\et(\mathcal{Y}^\circ_K)\cong \pi_1^\et(Y^{\circ}_K)$ by \cite[Theorem 3.1]{Lut93}.
\end{remark}
\begin{remark} Suppose \autoref{conj:Faltings_Dlog} holds. Similar to the non-logarithmic case, for every $f\geq 1$, one also has an equivalence of categories between the category of logarithmic Fontaine-Faltings modules with $\mathbb{Z}_{p^f}$-endomorphism structure and the category of $f$-periodic logarithmic Higgs-de Rham flows on $(\mathcal{Y},\mathcal{Z})$. Hence there is a correspondence between logarithmic crystalline $\mathbb{Z}_{p^f}$ representations and $f$-periodic logarithmic Higgs-de Rham flows.
\end{remark}
\subsection{Compatibility with the work of \cite{DLLZ}}

Finally, we pose a basic conjecture in logarithmic $p$-adic Hodge theory. The conjecture roughly says the following: two natural filtered de Rham bundles associated to a logarithmic crystalline representation are isomorphic.

\begin{conjecture}\label{conjecture:compatibility_filtered_dR}Let $Y/W$ be a smooth proper scheme with geometrically connected generic fiber and let $Z\subset Y$ be a relative simple normal crossings divisor. Suppose \autoref{conj:Faltings_Dlog} holds. Let $\mathbb{L}$ be a logarithmic crystalline local system on $(Y,Z)$, with associated logarithmic Fontaine-Faltings module $(M,\nabla,\mathrm{Fil},\varphi)_{\mathcal{Y}_K}$. As $Y/W$ is proper, the triple $(M,\nabla,\mathrm{Fil})$ is indeed algebraic, and hence defined over the scheme $Y_K$. By \autoref{rem:de Rham crystalline} and the rigidity of de Rham local systems, the local system satisfies the condition in \cite[Theorem 1.1]{DLLZ}. Then there is an isomorphism of filtered logarithmic de Rham bundles on $Y_K$
 \[(M,\nabla,\mathrm{Fil})_{\mathcal{Y}}\mid_{Y_K}= D^{\rm alg}_{\rm dR}(\mathbb{L}\otimes \mathbb{Q}_p)^\vee,\]
where the latter is as in \cite[Theorem 1.1]{DLLZ}.
\end{conjecture}

\begin{remark}
Suppose one had a logarithmic version of the crystalline period sheaf $\mathcal{O}\mathbb{B}_{\rm cris,\log}$ satisfying a logarithmic version of the Tan-Tong theorem. This means that
\begin{enumerate}
 \item a local system $\mathbb{L}$ is logarithmic crystalline if and only if there exist a filtered logarithmic $F$-isocrystal $\mathcal{E}$ satisfying
 \begin{equation}\label{equ:conj ass1}
 \mathcal{E}\otimes \mathcal{O}\mathbb{B}_{\rm cris,\log} \simeq \mathbb{L}\otimes \mathcal{O}\mathbb{B}_{\rm cris,\log};
 \end{equation}

 \item there is a natural injective map $\mathcal{O}\mathbb{B}_{\mathrm{cris},\log}\hookrightarrow \mathcal{O}\mathbb{B}_{\mathrm{dR},\log}$, where the latter is the logarithmic de Rham period sheaf defined in \cite{DLLZ}.
\end{enumerate}
 Then we claim that the conjecture would hold. Indeed, by forgetting the Frobenius structure on both sides in \eqref{equ:conj ass1} and extending the coefficients to $\mathcal{O}\mathbb{B}_{\mathrm{dR},\log}$ one gets the required equation to ensure the $\mathbb{L}$ is logarithmic de Rham.
\end{remark}

\bigskip

\noindent\textbf{Data Availability} Data sharing not applicable to this article as no datasets were generated or analysed during the current study.

\bigskip

\noindent\textbf{\large Declarations}

\bigskip

\noindent\textbf{Conflict of interest} On behalf of all authors, the corresponding author states that there is no conflict of interest.

\bigskip

\newcommand{\etalchar}[1]{$^{#1}$}

\end{document}